\documentclass[english, 10pt]{amsart}
\usepackage{amsmath}
\usepackage{graphicx}
\usepackage{amssymb}
\usepackage{amsfonts, amsthm}
\usepackage[colorlinks=true,linkcolor=blue,citecolor=blue,urlcolor=blue]{hyperref}  

\newcommand{\ka}{\left(}
\newcommand{\kz}{\right)}

\newcommand{\nn}{\nonumber}

\newcommand{\ep}{\epsilon}
\newcommand{\al}{\alpha}
\newcommand{\be}{\beta}
\newcommand{\ga}{\gamma}
\newcommand{\de}{\delta}
\newcommand{\la}{\lambda}

\newcommand{\tp}{\tilde \psi}
\newcommand{\bg}{\bar g}
\newcommand{\auskommentieren}[1]{}

\newcommand{\Riemann}{\bar R_{\al\be\ga\de}}

\newcommand{\Ricci}{\bar R_{\al\be}}

\newcommand{\Fij}{F^{ij}}
\newcommand{\vp}{\varphi}
\newcommand{\beq}{\begin{equation}}
\newcommand{\eeq}{\end{equation}}
\newcommand{\bea}{\begin{equation}\begin{aligned}}
\newcommand{\eea}{\end{aligned}\end{equation}}

\newcommand{\Fz}{F^{-2}}
\newcommand{\Fg}{F^{ij}g_{ij}}

\newcommand{\pab}{{\psi}_{\al\be}}

\newcommand{\D}{\mathcal{D}}

\newtheorem{theorem}{Theorem}[section]
\newtheorem{lem}[theorem]{Lemma}
\newtheorem{cor}[theorem]{Corollary}

\theoremstyle{definition}
\newtheorem{defi}[theorem]{Definition}
\newtheorem{rem}[theorem]{Remark}

\usepackage{array} % mathematische Pakete
\usepackage{amssymb}
\usepackage{latexsym}
\usepackage{epic}
\usepackage{curves}
\usepackage{fancyhdr}
\usepackage{xspace}

\numberwithin{equation}{section}
\DeclareMathOperator{\const}{const}
\DeclareMathOperator{\graph}{graph}

\newcommand{\braket}[1]{ \ensuremath \langle #1 \rangle }
\newcommand{\norm}[1]{ \ensuremath ||| #1 ||| }
\begin{document}
\title{The inverse $F$-curvature flow in ARW spaces}
\author{Heiko Kr\"oner} 
\date{\today}
\maketitle
\footnotetext{
\textsc{Ruprecht-Karls-Universit\"at, Institut f\"ur Angewandte Mathematik, Im Neuenheimer Feld 294, 69120 Heidelberg, Germany}\\
%\textit {E-mail address:} \href{mailto:heiko.kroener@urz.uni-heidelberg.de}{heiko.kroener@urz.uni-heidelberg.de}\\
\bigskip
This work was supported by the DFG and the Heidelberger Graduiertenakademie.}
\begin{abstract}
In this paper we consider the so-called inverse $F$-curvature flow (IFCF)
\beq
\dot x = -F^{-1}\nu
\eeq
in ARW spaces, i.e. in Lorentzian manifolds with a special future singularity. Here, $F$ denotes a curvature function of class $(K^*)$, which is homogenous of degree one, e.g. the $n$-th root of the Gaussian curvature, and $\nu$ the past directed normal. We prove existence of the IFCF for all times and convergence of the rescaled scalar solution in $C^{\infty}(S_0)$ to a smooth function. Using the rescaled IFCF we maintain a transition from big crunch to big bang into a mirrored spacetime.
\end{abstract}
\tableofcontents
%\auskommentieren{
\newpage
\section{Introduction} \label{Section_1}
Let $N=N^{n+1}$ be a ARW space with respect to the future, i.e. $N$ is a globally hyperbolic spacetime and a future end $N_+$ of $N$ can be written as a product $[a,b)\times S_0$, where $S_0$ is a Riemannian space and there exists a future directed time function $\tau= x^0$ such that the metric in $N_+$ can be written as
\beq
d\breve s^2 = e^{2 \tilde \psi}\{-(dx^0)^2+ \sigma_{ij}(x^0, x)dx^idx^j\},
\eeq
where $S_0$ corresponds to 
\beq \label{xhoch0gleicha}
x^0=a,
\eeq
$\tilde \psi$ is of the form 
\beq
\tilde \psi(x^0, x) = f(x^0)+ \psi(x^0,x),
\eeq
and we assume that there exists a positive constant $c_0$ and a smooth Riemannian metric $\bar \sigma_{ij}$ on $S_0$ such that 
\beq
\lim_{\tau \rightarrow b}e^{\psi} = c_0 \quad \wedge \quad \lim_{\tau \rightarrow b}\sigma_{ij}(\tau, x)= \bar \sigma_{ij}(x) \quad \wedge \quad \lim_{\tau \rightarrow b}f(\tau)= -\infty.
\eeq
W.l.o.g. we may assume $c_0=1$. Then $N$ is ARW with respect to the future, if the derivatives of arbitrary order with respect to space and time of $e^{-2f}\breve g_{\al \be}$
converge uniformly to the corresponding derivatives of the following metric
\beq
-(dx^0)^2 + \bar \sigma_{ij}(x)dx^idx^j
\eeq  
when $x^0$ tends to $b$.

We assume furthermore, that $f$ satisfies the following five conditions 
\beq
0 < -f^{'},
\eeq
there exists $\omega \in \mathbb{R}$ such that
\beq \label{limes_ist_gleich_der_masse}
n +\omega-2 >0 \quad \wedge \quad \lim_{\tau \rightarrow b} |f^{'}|^2e^{(n + w-2)f}=m>0.
\eeq
Set $\tilde \ga = \frac{1}{2}(n+\omega-2)$, then there exists the limit
\beq \label{fzweistrich_wiefstrich_quadrat}
\lim_{\tau \rightarrow b}(f^{''}+\tilde \ga |f^{'}|^2)
\eeq
and
\beq \label{assumption_derivative_of_f}
|D^m_{\tau}(f^{''}+\tilde \ga |f^{'}|^2)| \le c_m |f^{'}|^m \quad \forall m \ge 1,
\eeq
as well as
\beq \label{abschaetzung_der_ableitungen_von_f_nach_vor}
|D^m_{\tau}f|\le c_m|f^{'}|^m \quad \forall m \ge 1.
\eeq
If $S_0$ is compact, then we call $N$ a normalized ARW spacetime, if 
\beq
\int_{S_0}\sqrt{\det \bar \sigma_{ij}} = |S^n|.
\eeq
In the following $S_0$ is assumed to be compact.
\begin{rem}
(i) If these assumptions are satisfied, then we shall show that the range of $\tau$ is finite, hence we may\---and shall\---assume w.l.o.g. that $b=0$, i.e.
\beq
a < \tau <0.
\eeq

(ii) Any ARW space with compact $S_0$ can be normalized as one easily checks.
\end{rem}

To guarantee the $C^3$-regularity for the transition flow, see Section \ref{Section_11}, especially (\ref{anwendung_zusatzbedingung_fuer_c3_regularitaet}), we have to impose another technical assumption, namely that the following limit exists
\beq \label{Zusatzbedingung_heiko}
\lim_{\tau \rightarrow 0}(f^{''}+\tilde \ga |f^{'}|^2)^{'}\tau.
\eeq

We furthermore assume that in the case $\tilde\ga <1$ the limit metric $\bar \sigma_{ij}$ has non-negative sectional curvature.

We can now state our main theorem, cf. also Section \ref{Section_2} for notations.

\begin{theorem} \label{main_theorem}
Let $N$ be as above and let $F\in C^{\infty}(\Gamma_+)\cap C^0(\bar \Gamma_+)$ be a curvature function of class $(K^*)$, cf. Definition (\ref{Kstern}), in the positive cone $\Gamma_+\subset \mathbb{R}^n$, which is in addition positiv homogenous of degree one and normalized such that
\beq
F(1, ..., 1) = n.
\eeq
Let $M_0$ be a smooth, closed, spacelike hypersurface in $N$ which can be written as a graph over $S_0$ for which we furthermore assume that it is convex and that it satisfies
\beq \label{M_0_weit_genug_in_Zukunft}
-\epsilon < \inf_{M_0}x^0 < 0,
\eeq
where
\beq
\epsilon = \epsilon (N, \breve g_{\al \be}) > 0.
\eeq

(i) Then the so-called inverse $F$-curvature flow (IFCF) given by the equation
\beq \label{main_evolution_equation}
\dot x = -\frac{1}{F}\nu
\eeq
with initial surface
$x(0) = M_0$
exists for all times. Here, $\nu$ denotes the past directed normal.  

(ii) If we express the flow hypersurfaces $M(t)$ as graphs over $S_0$
\beq
M(t)= \graph u(t, \cdot),
\eeq
and set  
\beq
\tilde u = ue^{\ga t},
\eeq
where $\ga= \frac{1}{n}\tilde \ga$, then there are positive constants $c_1, c_2$ such that
\beq
-c_2 \le \tilde u \le -c_1 <0,
\eeq
and $\tilde u$ converges in $C^{\infty}(S_0)$ to a smooth function, if $t$ goes to infinity.

(iii) Let $(g_{ij})$ be the induced metric of the leaves $M(t)$ of the inverse $F$-curvature flow, then the rescaled metric 
\beq
e^{\frac{2}{n}t} g_{ij}
\eeq
converges in $C^{\infty}(S_0)$ to 
\beq
(\tilde \ga^2 m)^{\frac{1}{\tilde \ga}}(-\tilde u)^{\frac{2}{\tilde \ga}}\bar \sigma_{ij},
\eeq
where we are slightly ambiguous by using the same symbol to denote $\tilde u (t, \cdot)$ and $\lim \tilde u(t, \cdot)$.

(iv) The leaves $M(t)$ of the IFCF get more umbilical, if $t$ tends to infinity, namely
\beq
 F^{-1}|h^j_i-\frac{1}{n}H\delta^j_i | \le ce^{-2\ga t}.
\eeq
In case $n+\omega - 4 > 0$, we even get a better estimate, namely
\beq
|h^j_i-\frac{1}{n}H\delta^j_i| \le c e^{-\frac{1}{2n}(n+\omega -4)t}.
\eeq
\end{theorem}

In \cite{CG} together with \cite{ARW} this theorem is proved when the curvature $F$ is replaced by the mean curvature of the flow hypersurfaces. 

In our proof we go along the lines of  \cite{CG} and \cite{ARW} as far as possible, for Section \ref{Section_5} we use \cite{Indiana}. 

The paper is organized as follows. In the remainder of the present section we list some well-known properties of $f$, cf. \cite[section 7.3]{CP}, which will be used later. In Section \ref{Section_2} we introduce some notations and definitions. In Section \ref{Section_3}, \ref{Section_4} and \ref{Section_5} we prove Theorem \ref{main_theorem} (i), in Section \ref{Section_6}, \ref{Section_7}, \ref{Section_8}, \ref{Section_9} and \ref{Section_10} we prove Theorem \ref{main_theorem} (ii)-(iv) and in Section \ref{Section_11} we will define a so-called transition from big crunch to big bang via the rescaled IFCF into a mirrored universe.

Let us briefly compare our case with the mean curvature case.

Concerning the proof of the existence of the flow the $C^0$-estimates are similar to the mean curvature case and the $C^1$-estimates are even easier in our case, since they follow immediately from the convexity of the flow hypersurfaces. 
For the $C^2$-estimates we prove the important Lemma \ref{Schnittkruemmung_positiv} and obtain with it in Lemma \ref{F_to_infty} the optimal lower bound for the $F$-curvature of the flow hypersurfaces, at which optimality is not seen until Section \ref{Section_8}. The remaining part of the $C^2$-estimates is different from the mean curvature case but can be found in \cite{Indiana}.

Concerning the asymptotic behaviour of the flow the $C^0$-estimates are similar to the mean curvature case. But the $C^1$-estimates in Section \ref{Section_7} and particularly the crucial $C^2$-estimates in Section \ref{Section_8} differ essentially from the mean curvature case. Using the homogeneity of $F$ the $C^2$-estimates lead to very good decay properties of the derivatives of $F$, so that from this time on the difference between our and the mean curvature case is only formal.

I would like to thank Claus Gerhardt for many helpful hints.

\begin{lem}
Let $f \in C^2([a,b))$ satisfy the conditions 
\beq
\lim_{\tau\rightarrow b}f(\tau) = - \infty
\eeq
and
\beq
\lim_{\tau \rightarrow b}|f^{'}|^2e^{2\tilde \ga f}=m,
\eeq
where $\tilde \ga, m$ are positive, then $b$ is finite.
\end{lem}
\begin{cor}
We may--and shall--therefore assume that $b=0$, i.e., the time interval $I$ is given by $I=[a,0)$.
\end{cor}
\begin{lem} \label{absolute_basic_konvergenzproperty_f}
(i) 
\beq
\lim_{\tau \rightarrow 0}\frac{e^{\tilde \ga f}}{\tau}= -\tilde \ga \sqrt{m}.
\eeq

(ii) There holds
\beq
f^{'}e^{\tilde \ga f}+\sqrt{m} \sim c \tau^2,
\eeq
where $c$ is a constant, and where the relation
\beq
\varphi \sim c\tau^2
\eeq
means
\beq
\lim_{\tau \rightarrow 0}\frac{\varphi(\tau)}{\tau^2} = c.
\eeq
\end{lem}
\begin{lem}
The asymptotic relation
\beq \label{asymptotic_relation_for_f_strich}
\tilde \ga f^{'}\tau -1 \sim c \tau^2
\eeq
is valid.
\end{lem}

\section{Notations and definitions}\label{Section_2}

\auskommentieren{
\beq \label{Zusammenhang_fundamentalform_zweite_Ableitung_u}
e^{-\tilde \psi}v^{-1}h_{ij} = -u_{ij} - \bar \Gamma^0_{00}u_iu_j - \bar \Gamma^0_{0j}u_i - \bar \Gamma^0_{0i}u_j + e^{-\tilde \psi}\bar h_{ij}
\eeq
Bemerkungen:

Evolutionsgleichungen, sofern m\"oglich f\"ur allgemeines $\Phi -f$

anderes Symbol als das "$+$"-Zeichen f\"ur $\stackrel{+}{\bar g_{\al\be}}$

Existenz des Flusses auch ohne f\"ur $F\in K$ m\"oglich?

Wenn die Limesmetrik in $N$ positive SchnittrkŸmmung hat hat der KR†MMUNGSTERM EIN VORZEICHEN

Versuche Existenz einer strikt konvexen Funktion in einem Zukunftsende zu beweisen.

}

%}
In this section, where we want to introduce some general notations, we assume for $N$ all properties listed from the beginning of Section \ref{Section_1}  as far as equation (\ref{xhoch0gleicha}) except for being ARW and we write $\psi$ instead of $\tilde \psi$. Let $M\subset N$ be a connected
and spacelike hypersurface with differentiable normal $\nu$
(which is then timelike). Geometric quantities in $N$
are denoted by $\ka {\bar g}_{\al \be}\kz$, $\ka {\bar R}_{\al \be \ga
\de}\kz$ etc. and those in $M$ by $\ka g_{ij}\kz$, $\ka
R_{ijkl}\kz$ etc.. Greek indices range from $0$ to $n$,
Latin indices from $1$ to $n$; summation convention is used. 
Coordinates in $N$ and $M$ are denoted by
$\ka x^{\al}\kz$ and $\ka {\xi}^{i}\kz$ respectively. Covariant derivatives
are written as indices, only in case of possibly confusion we
precede them by a semicolon, i.e.
for a function $u$ the gradient is $\ka u_{\al}\kz$ and $\ka
u_{\al\be} \kz$ the hessian, but for the covariant derivative of the Riemannian curvature tensor we write
${\bar R}_{\al \be \ga \de; \ep}$.

In local coordinates, $\ka x^{\al} \kz$ in $N$ and $\ka {\xi}^i
\kz$ in $M$, the following four important equations hold; the Gauss formular
 \bea \label{GF}
x^{\al}_{ij} = h_{ij}{\nu}^{\al}. 
\eea 

In this implicit definition $(h_{ij})$ is the second fundamental form of $M$ with respect to $\nu$.
Here and in the following a covariant derivative is always a full tensor, i.e. 
\bea
x^{\al}_{ij} = x^{\al}_{,ij} - {\Gamma}^k_{ij}x^{\al}_k+{\bar
{\Gamma}}^{\al}_{\be \ga}x^{\be}_i x^{\ga}_j 
\eea 
and the comma denotes ordinary partial derivatives.

The second equation is the  Weingarten equation 
\bea \label{WG}
{\nu}^{\al}_i = h^k_i x^{\al}_k, 
\eea 
where ${\nu}^{\al}_i$ is a full tensor. The third equation is the Codazzi equation
\bea \label{CG} h_{ij;k}
- h_{ik;j} = {\bar R}_{\al \be \ga \de}{\nu}^{\al}x^{\be}_i
x^{\ga}_j x^{\de}_k 
\eea 
and the fourth is the Gau\ss\ equation 
\bea \label{GG}
R_{ijkl} = -\left\{h_{ik}h_{jl}-h_{il}h_{jk}\right\}+ {\bar
R}_{\al \be \ga \de}x^{\al}_ix^{\be}_j{x}^{\ga}_k{x}^{\de}_l. 
\eea
As an example for the covariant derivative of a full tensor we give
\bea 
{\bar R}_{\al \be \ga
\de; i} = {\bar R}_{\al \be \ga \de; \ep} x^{\ep}_i, 
\eea 
where this identity follows by applying the chain rule from the definition of the covariant derivative of a full tensor; it can be generalized obviously to other quantities. 

Let $\ka x^{\al} \kz$ be a future directed coordinate system in $N$, then the 
contravariant vector $\ka {\xi}^{\al}\kz = \ka 1, 0, ...,0\kz$ is future directed; as well its covariant version
$ \ka{\xi}_{\al}\kz=e^{2\psi} \ka -1, 0, ..., 0\kz$.

Now we want to express normal, metric and second fundamemtal form for 
spacelike hypersurfaces, which can be written as graphs over the Cauchyhypersurface.
Let $M = \graph u|_{S_0}$ be a spacelike hypersurface in $N$, i.e. 
\bea M= \left\{\ka
x^0, x \kz: x^0=u(x), \ x\in S_0\right\}, 
\eea 
then the induced metric is given by
\bea 
\label{metrik_eines_graphen}
g_{ij} = e^{2\psi}\left\{-u_i u_j + {\sigma}_{ij}\right\}, 
\eea
where ${\sigma}_{ij}$ is evaluated at $(u,x)$ and the inverse $\ka g^{ij} \kz = {\ka g_{ij} \kz}^{-1}$ is given by 
\bea 
\label{inverse_Metrik} g^{ij} = e^{-2\psi}
\left\{{\sigma}^{ij} + \frac{u^iu^j}{v^2}\right\}, 
\eea 
where
$\ka{\sigma}^{ij}\kz = {\ka {\sigma}_{ij}\kz}^{-1}$ 
and 
\begin{equation}
\begin{aligned}
u^i = &\ {\sigma}^{ij}u_j \\
v^2 = &\ 1-{\sigma}^{ij}u_iu_j \equiv 1-{|Du|}^2, \quad v>0. 
\end{aligned}
\end{equation}
We define $\tilde v=v^{-1}$.

From (\ref{metrik_eines_graphen}) we conclude that  $\graph u$ is spacelike if and only if $|Du|<1$.

The covariant version of the normal of a graph is 
\bea
\ka {\nu}_{\al}\kz = \pm v^{-1} e^{\psi}\ka 1, -u_i \kz \eea 
and the contravariant version
\bea \ka {\nu}^{\al}\kz = \mp
v^{-1}e^{-\psi} \ka 1, u^i\kz. 
\eea 

We have
\begin{rem}
Let $M$ be a spacelike graph in a future directed coordinate system, then 
\bea \ka {\nu}^{\al}\kz = v^{-1}
e^{-\psi} \ka 1, u^i\kz 
\eea 
is the contravariant future directed normal and 
\bea \label{normal} \ka {\nu}^{\al}\kz =
-v^{-1} e^{-\psi} \ka 1, u^i\kz 
\eea 
the past directed.
\end{rem}

In the following we choose $\nu$ always as the past directed normal.

Let us consider the component $\al =0$ in (\ref{GF}), so we have due to (\ref{normal}) that
\bea \label{gleichung_eins_sechzehn}
e^{-\psi} v^{-1} h_{ij} = -u_{ij}-{\bar \Gamma}^0_{00}u_iu_j
-{\bar \Gamma}^0_{0j}u_i -{\bar \Gamma}^0_{0i}u_j-{\bar
\Gamma}^0_{ij}, \eea 
where $u_{ij}$ are covariant derivatives with respect to $M$.
Choosing $u \equiv \const$, we deduce 
\bea e^{-\psi}{\bar
h}_{ij} = -{\bar \Gamma}^0_{ij}, \eea 
where ${\bar h}_{ij}$ is the second fundamental form of the hypersurface $\left\{x^0 =
\const\right\}$. 
An easy calculation shows
\bea
e^{-\psi}{\bar h}_{ij} =-\frac{1}{2}{\dot \sigma}_{ij}-\dot \psi
{\sigma}_{ij}, 
\eea 
where the dot indicates differentiation with respect to $x^0$.

Now we define the classes $(K)$ and $(K^{*})$, which are special classes of curvature functions; for a more detailed treatment of these classes we refer to \cite[Section 2.2]{CP}. 

For a curvature function $F$ (i.e. symmetric in its variables) in the positive cone $\Gamma_+ \subset \mathbb{R}^n$ we define
\beq \label{Fhij_gleich_Fkappai}
F(h_{ij}) = F(\kappa_i),
\eeq
where the $\kappa_i$ are the eigenvalues of an arbitrary symmetric tensor $(h_{ij})$, whose eigenvalues are in $\Gamma_+$.
\begin{defi}\label{KlasseK}
A symmetric curvature function $F\in C^{2, \al}(\Gamma_+)\cap C^0(\bar \Gamma_+)$, positively homogeneous of degree $d_0 >0$, is said to be of class $(K)$, if
\beq
F_i=\frac{\partial F}{\partial \kappa^i}>0 \quad \text{in} \ \Gamma_+,
\eeq 
\beq
F_{|\partial \Gamma_+}=0,
\eeq
and
\beq
F^{ij,kl}\eta_{ij}\eta_{kl}\le F^{-1}(F^{ij}\eta_{ij})^2-F^{ik}\tilde h^{jl}\eta_{ij}\eta_{kl} \quad \forall\ Ê\eta \in S,
\eeq
where $F$ is evaluated at an arbitrary symmetric tensor $(h_{ij})$, whose eigenvalues are in $\Gamma_+$ and $S$ denotes the set of symmetric tensors.
Here, $F_i$ is a partial derivative of first order with respect to $\kappa_i$ and $F^{ij,kl}$ are second partial derivatives with respect to $(h_{ij})$.
Furthermore $(\tilde h^{ij})$ is the inverse of $(h_{ij})$.
\end{defi}

In Theorem \ref{main_theorem}  the $\kappa_i$ in (\ref{Fhij_gleich_Fkappai}) are the eigenvalues of the second fundamental form $(h_{ij})$ with respect to the metric $(g_{ij})$, i.e. the principal curvatures of the flow hypersurfaces. 

\begin{defi} \label{Kstern}
A curvature function $F \in (K)$ is said to be of class $(K^*)$, if there exists $0<\epsilon_0 = \epsilon_0(F)$ such that
\beq
\epsilon_0 FH \le \Fij h_{ik}h^k_j,
\eeq
for any symmetric $(h_{ij})$ with all eigenvalues in $\Gamma_+$, where $F$ is evaluated at $(h_{ij})$. $H$ represents the mean curvature, i.e. the trace of $(h_{ij})$.
\end{defi}

In the following a '$+$' sign attached to the symbol of a metric of the ambient space refers to the corresponding Riemannian background metric, if attached to an induced metric, it refers to the induced metric relative to the corresponding Riemannian background metric. Let us consider as an example the metrics $\breve g_{\alÊ\be}$ and $g_{ij}$ introduced as above, then
\beq
\stackrel{+}{\breve g}_{\alÊ\be} = e^{2 \tilde \psi}\{(dx^0)^2+ \sigma_{ij}(x^0, x)dx^idx^j\}, \quad 
\stackrel{+}{g}_{ij}=  \stackrel{+}{\breve g}_{\alÊ\be}x^{\al}_ix^{\be}_j.
\eeq

\section{$C^0$-estimates\---Existence for all times} \label{Section_3}
Let $M_{\tau}=\{x^0=\tau\}$ denote the coordinate slices. Then 
\beq
|M_{\tau}| = \int_{S_0}e^{n\tilde \psi(\tau, x)}\sqrt{|\det \sigma_{ij}(\tau, x)|} dx\longrightarrow 0, \quad \tau \rightarrow 0.
\eeq
And for the second fundamental form $ \bar h_{ij}$ of the $M_{\tau}$ we have
\beq \label{convex_coordinate_slice}
\bar h^i_j  = -e^{-\tilde \psi}(\frac{1}{2}\sigma^{ik}\dot \sigma_{kj}+ \dot{\tilde \psi}\delta^i_j),
\eeq
hence there exists $\tau_0$ such that $M_{\tau}$ is convex for all $\tau \ge \tau_0$. 

Choosing $\tau_0$ if necessary larger we have
\bea \label{varphi_strong_F_volume_decay}
e^{\tilde \psi} F |_{M_{\tau}} = e^{\tilde \psi}F(\bar h^i_j) = F(-\frac{1}{2}\sigma^{ik}\dot \sigma_{kj}-\dot {\tilde \psi}\delta^i_j)  \ge -\delta_0 f^{'} =:  \varphi(\tau) \quad \forall \tau \ge \tau_0,
\eea 
where $\delta_0>0$ is a constant.

%\subsection{$C^0$-estimates}
We will show that the flow does not run into the future singularity within finite time. 
\begin{lem} \label{spezielle_Zeitfkt}
There exists a time function $\tilde x^0=\tilde x^0(x^0)$,  so that the $F$-curvature $\bar F$ of the slices $\{\tilde x^0=const\}$ satisfies
\beq
e^{\tilde {\tilde \psi}}\bar F \ge 1.
\eeq
$e^{\tilde {\tilde \psi}}$ is the conformal factor in the representation of the metric with respect to the coordinates $(\tilde x^0, x^i)$, i.e.
\beq \label{metric_in_volume_decay_coordinates}
d\breve s = e^{2 \tilde {\tilde \psi}} \{-(d\tilde x^0)^2 + \tilde \sigma_{ij}(\tilde x^0, x)dx^idx^j\}.
\eeq
Furthermore there holds
\beq
\tilde x^0(\{\tau_0 \le x^0<0\}) = [0, \infty)
\eeq
and the future singularity corresponds to $\tilde x^0=\infty$.
\end{lem}
\begin{proof}
Define $\tilde x^0$ by 
\beq
\tilde x^0 = \int_{\tau_0}^{\tau} \varphi(s)ds = -\int_{\tau_0}^{\tau}\epsilon_0 f^{'} = \epsilon_0 f(\tau_0)-\epsilon_0 f(\tau) \rightarrow \infty, \quad \tau \rightarrow 0,
\eeq
where $\varphi$ is chosen as in (\ref{varphi_strong_F_volume_decay}). For the conformal factor in (\ref{metric_in_volume_decay_coordinates}) we have
\beq
e^{2 \tilde{\tilde \psi}} = e^{2 \tilde \psi}\frac{\partial x^0}{\partial \tilde x^0}\frac{\partial x^0}{\partial \tilde x^0} = e^{2 \tilde \psi}\varphi^{-2}
\eeq
and therefore
\beq
e^{ \tilde{\tilde \psi}}  \bar F = e^{\tilde \psi}\bar F \varphi^{-1}\ge 1.
\eeq
\end{proof}

The evolution problem (\ref{main_evolution_equation}) is a parabolic problem, hence a solution exists on a maximal time interval $[0, T^{*})$, $0 < T^{*}\le \infty$.
\begin{lem} \label{fluss_bleibt_in_endlicher_zeit_im_kompakter_menge}
For any finite $0<T\le T^{*}$  the flow stays in a precompact set $\Omega_T$ for $0 \le t <T$.
\end{lem}
\begin{proof}
For the proof we choose with Lemma \ref{spezielle_Zeitfkt} a time function $x^0$ such that
\beq
e^{\tilde \psi} \bar F \ge 1
\eeq
for the coordinate slices $\{x^0=const\}$. Let 
\beq
M(t) = \graph u(t, \cdot)
\eeq
be the flow hypersurfaces in this coordinate system and
\beq
\varphi(t) = \sup_{S_0}u(t, \cdot) = u(t, x_t)
\eeq
with suitable $x_t \in S_0$. It is well-known that $\varphi$ is Lipschitz continuous and that for a.e. $0 \le t < T$
\beq \label{phi_dot_gleich_u_punkt}
\dot \varphi (t) = \frac{\partial}{\partial t} u(t, x_t).
\eeq
From (\ref{gleichung_eins_sechzehn})
we deduce in $x_t$ the relation
\beq
h_{ij} \ge \bar h_{ij},
\eeq
hence
\beq
F \ge \bar F.
\eeq
We look at the component $\al=0$ in (\ref{main_evolution_equation}) and get
\beq
\dot u = \frac{\tilde v}{Fe^{\tilde \psi}},
\eeq
where
\beq
\dot u = \frac{\partial u}{\partial t} + u_i\dot x^i
\eeq
is a total derivative. This yields
\beq
\frac{\partial u}{\partial t} = e^{-\tilde \psi} v \frac{1}{F},
\eeq
so that we have in $x_t$
\beq
\frac{\partial u}{\partial t} = \frac{1}{e^{\tilde \psi} F}  \le \frac{1}{e^{\tilde \psi} \bar F} \le 1.
\eeq

With (\ref{phi_dot_gleich_u_punkt}) we conclude
\beq
\varphi \le \varphi(0) + t  \quad \forall 0 \le t < T^*,
\eeq
which proves the lemma, since the future singularity corresponds to $x^0 = \infty$.
\end{proof}

\begin{rem} \label{Fluss_laeuft_in_Zukunft_wenn_fuer_alle_Zeiten_existent}
If we choose 
\beq
\varphi (t) = \inf_{S_0}u(t, \cdot)
\eeq
in the proof of Lemma \ref{fluss_bleibt_in_endlicher_zeit_im_kompakter_menge}, we can easily derive that the flow runs into the future singularity,
which means\---in  the coordinate system chosen there\---
\beq
\lim_{t \rightarrow \infty} \inf_{S_0}u(t, \cdot) =\infty,
\eeq
provided the flow exists for all times.
\end{rem}
 
 \section{$C^1$-estimates\---Existence for all times} \label{Section_4}
 
 As a direct consequence of \cite[Theorem 2.7.11]{CP} and the convexity of the flow hypersurfaces we have the following
 \begin{lem} \label{tilde_v_beschraenkt_in_rel_komp_mengen}
 As long as the flow stays in a precompact set $\Omega$ the quantity $\tilde v$ is uniformly bounded by a constant, which only depends on $\Omega$.
 \end{lem}
 
 Due to later demand our aim in the remainder of this section will be to prove an estimate for $\tilde v$ for the leaves of the IFCF on the maximal existence interval $[0, T^*)$, cf. Lemma \ref{tilde_v_uniformly_bounded} and to prove Lemma \ref{Schnittkruemmung_positiv}.
 
 To prove this we consider the flow to be embedded in $N$ with the conformal metric
 \beq \label{conformal_metric}
 \bar g_{\al \be}= e^{-2 \tilde \psi}\breve g_{\al \be} = -(dx^0)^2+ \sigma_{ij}(x^0, x)dx^idx^j.
 \eeq
 This point of view will be later on also a key ingredient in the proof of the convergence results for the flow. Though, formally we have a different ambient space we still denote it by the same symbol $N$ and distinguish only the metrics $\breve g_{\al \be}$ resp. $\bar g_{\al \be}$ and the corresponding quantities of the hypersurfaces $\breve h_{ij}$, $\breve g_{ij}$, $\breve \nu$ resp. $h_{ij}$, $g_{ij}$, $\nu$, etc., i.e., the standard notations now apply to the case when $N$ is equipped with the metric (\ref{conformal_metric}). 
 
 The second fundamental forms $\breve h^j_i$ and $h^j_i$ are related by 
 \beq
 e^{\tilde \psi}\breve h^j_i = h^j_i+ \tilde \psi_{\al}\nu^{\al}\delta^j_i=h^j_i - \tilde v  f^{'}{\de}^j_i+{\psi}_{\al}{\nu}^{\al} = _{def}\check h^j_i,
 \eeq
 cf. \cite[Proposition 1.1.11]{CP}.
 When we insert $\breve h^j_i$ into $F$ we will denote the result in accordance with our convention as $\breve F$. Due to a lack of convexity it would not make any sense to insert $h^j_i$ into the curvature function $F$, so that we stipulate that the symbol $F$ will stand for 
 \beq \label{zusammenhang_zwischen_F_und_breve_F}
 F = e^{\tilde \psi}\breve F =F(h^j_i - \tilde v  f^{'}{\de}^j_i+{\psi}_{\al}{\nu}^{\al}),
 \eeq  
 which will be useful, cf. (\ref{main_flow_equation}).

Quantities like $\tilde v$, that are not different if calculated with respect to $\breve g_{\al \be}$ or $\bar g_{\al \be}$ are denoted in the usual way.

 These notations introduced above will be used in the present section as well as from the beginning of Section \ref{Section_6} to the end of this paper. 
 
 Due to 
 \beq
 \breve \nu = e^{-\tilde \psi}\nu
 \eeq
 the evolution equation $\dot x = -\frac{1}{\breve F}\breve \nu$ can be written as
 \beq \label{main_flow_equation}
 \dot x = -\frac{1}{F}\nu.
 \eeq
 
 \begin{lem}\label{Evolution_equation_of_tilde_v}(Evolution of $\tilde v$)
 Consider the flow (\ref{main_flow_equation}). Then $\tilde v$ satisfies the evolution equation
\bea
\dot {\tilde v}   &-F^{-2} F^{ij} \tilde v_{ij}= -\Fz \Fij h_{kj}h^k_i \tilde v + \Fz \Fij \Riemann \nu^{\al}x^{\be}_ix^{\ga}_lx^{\de}_j u^l\\
& -\Fz \Fij h_{ij} \eta_{\al \be} \nu^{\al} \nu^{\be} -F^{-1} \eta_{\al \be} \nu^{\al} \nu^{\be}\\
& - \Fz (\Fij \eta_{\al \be \ga}\nu^{\al}x^{\be}_ix^{\ga}_j +  2\Fij \eta_{\al \be} x^{\al}_k x^{\be}_i h^k_j) \\
&  -\Fz(-\tilde v f^{''}\|Du\|^2 \Fij g_{ij}-\tilde v_ku^k f^{'} \Fij g_{ij}\\
&+\psi_{\al \be}\nu^{\al}x^{\be}_k u^k\Fij g_{ij}+\psi_{\al}x^{\al}_lh^l_k u^k \Fij g_{ij}) ,
\eea
where $\eta = (\eta_{\al})=(-1, 0, ..., 0)$ is a covariant unit vectorfield.
\end{lem}

\begin{proof}
We have
\beq
\tilde v = \eta_{\al}\nu^{\al}.
\eeq
Let $(\xi^i)$ be local coordinates for $M(t)$; differentiating $\tilde v$ covariantly yields
\beq
\tilde v_i = \eta_{\al\be}x^{\be}_i \nu^{\al} + \eta_{\al}\nu^{\al}_i
\eeq
and
\bea
\tilde v_{ij}=&\eta_{\al\be\ga}x^{\be}_ix^{\ga}_j\nu^{\al}+ \eta_{\al\be}\nu^{\al}_jx^{\be}_i+ \eta_{\al\be}\nu^{\al}\nu^{\be}h_{ij}+\eta_{\al}x^{\al}_{rj}h^r_i \\
 &+\eta_{\al}x^{\al}_rh^r_{i;j} + \eta_{\al\be}x^{\be}_jx^{\al}_rh^r_i.
\eea
As usual, cf. \cite[Lemma 2.3.2]{CP}, the evolution equation for the normal is
\beq
\dot \nu^{\al} =g^{ij}(-\frac{1}{F})_ix^{\al}_j =\frac{1}{F^2} g^{ij}F_ix^{\al}_j
\eeq
and for the time derivative of $\tilde v$ we get
\bea
\dot {\tilde v} =& \eta_{\al\be}\nu^{\al}\dot x^{\be}+\eta_{\al}\dot \nu^{\al} \\
=&-\frac{1}{F}\eta_{\al\be}\nu^{\al}\nu^{\be}-\frac{1}{F^2}g^{ij}F_iu_j.
\eea

Writing
\bea
F_k=&\Fij h_{ij;k} - \tilde v_k f^{'} \Fg - \tilde v f^{''}u_k \Fg \\ 
&+ \psi_{\al\be}\nu^{\al}x^{\be}_k\Fg+\psi_{\al}x^{\al}_rh^r_k\Fg
\eea
and using the Codazzi equation
\beq 
h_{ij;k}-h_{ik;j}= \Riemann \nu^{\al}x^{\be}_ix^{\ga}_jx^{\de}_k
\eeq
we deduce the desired evolution equation for $\tilde v$ by putting together the above equations.
\end{proof}

We now present some auxiliary estimates which will be needed in the following.
 
\begin{lem} \label{normabschaetzungen_mit_v_schlange}
Let $\norm{\cdot}$ denote the norm of a tensor with respect to the Riemannian metric $\stackrel {+}{\bar g}_{\al\be}$, cf Section \ref{Section_2}, then

(i)
\bea
| \eta_{\al \be}\nu^{\al}\nu^{\be}| \le & c \tilde v^2 \norm{\eta_{\al \be}}, \\
|\Fij \eta_{\al \be \ga}\nu^{\al}x^{\be}_i x^{\ga}_j| \le & c \tilde v^3 \norm{\eta_{\al \be \ga}}\Fg, \\
|\psi_{\al \be}\nu^{\al}x^{\be}_ku^k| \le &c \norm{\eta_{\al \be}}\tilde v^3. 
\eea

(ii) For any $\epsilon > 0$ we have
\beq
|\Fij\eta_{\al \be}x^{\al}_kx^{\be}_i h^k_j| \le c \epsilon  \tilde v \Fij h_{kj}h^k_i \norm{\eta_{\al \be}}+ c_{\epsilon} \tilde v^3 \Fg \norm{\eta_{\al \be}}.
\eeq

(iii)
\bea \label{most_complicated_FijRandsoon}
|\Fij \Riemann \nu^{\al}x^{\be}_i x^{\ga}_l x^{\de}_j u^l| \le & c\tilde v^3 \Fg.
\eea

(iv) Furthermore 
\beq \label{psiusw}
|\psi_{\al}x^{\al}_kh^k_iu^i| \le c \norm{D\psi} \tilde v^3
\eeq
in points where $\tilde v_i=0$.
\end{lem}
\begin{rem}
These are tensor estimates, i.e. not depending on the special local coordinates of the hypersurface and $S_0$. But to prove these estimates we sometimes choose special coordinates such that in a fixed point $g_{ij}=\de_{ij}, \stackrel {+}{ g}_{ij}= diagonal$.
\end{rem}

\begin{proof}[Proof of Lemma \ref{normabschaetzungen_mit_v_schlange}]
We have $\norm{\nu^{\al}} \le 2 \tilde v$,
\beq
\stackrel{+}{g}_{ij} \le 2 \sigma_{ij}  \le 2 \tilde v^2g_{ij} \quad \wedge \quad g^{ij}\le c\tilde v^2 \sigma^{ij} \quad \wedge \quad u^i = \tilde v^2 \check u^i
\eeq   
and $\|Du\|^2= \tilde v^2 |Du|^2$. 

{\bf Proof of (i):}
Using these properties together with Schwarz inequality proves (i).

{\bf Proof of (ii):}
\bea
\norm{\Fij &x^{\al}_kx^{\be}_i h^k_j }^2= \Fij F^{\bar i \bar j}h^k_jh^{\bar k}_{\bar j}\stackrel{+}{g}_{k \bar k}\stackrel{+}{g}_{i \bar i}, \quad g_{ij} = \de_{ij}, \stackrel{+}{g}_{ij}=diagonal \\
& \le c \tilde v^4\Fij F^{\bar i \bar j}h^k_jh^{\bar k}_{\bar j}g_{k \bar k}g_{i \bar i}, \quad g_{ij}= \de_{ij}, h_{ij}= \kappa_i\delta_{ij}, F^{ij}=diagonal\\
& \le c \tilde v^4 \sum_i(F^{ii})^2(h_{ii})^2 \\
& \le c \tilde v^4 (\sum_iF^{ii}|h_{ii}|)^2 \\
& \le c \tilde v^4 (\sum_i F^{ii}(\frac{\epsilon}{\tilde v}h_{ii}^2 + c_{\epsilon}\tilde v g_{ii}))^2,
\eea
taking the square root yields the result.

{\bf Proof of (iii):} The following proof can be found in \cite[Lemma 5.4.5]{CP}. Let $p \in M(t)$ be arbitrary. Let $(x^{\al})$ be the special Gaussian coordinate system of $N$ and $(\xi^i)$ local coordinates around $p$ such that 
\[
x^{\al}_i = \left\{ \begin{array}{r@{\quad,\quad}l}
u_i & \al=0\\
\de^k_i & \al=k.
\end{array}\right.
\]

All indices are raised with respect to $g^{ij}$ with exception of 
\beq
\check u^i = \sigma^{ij}u_j.
\eeq

We point out that
\bea
\|Du\|^2 &= g^{ij}u_iu_j = \tilde v^2\sigma^{ij}u_iu_j = \tilde v^2|Du|^2\\
(\nu^{\al})&=-\tilde v (1, \check u^i)
\eea
and
\beq \label{5.4.24}
 \eta_{\epsilon}x^{\epsilon}_lg^{kl} = -u^k.
\eeq 

We have
\bea
-\Fij \Riemann \nu^{\al}x^{\be}_i x^{\ga}_k x^{\de}_j u^k = \Fij \Riemann \nu^{\al}x^{\be}_i x^{\ga}_k x^{\de}_j \eta_{\epsilon}x^{\epsilon}_lg^{kl}.
\eea
Let
\beq \label{5.4.17}
a_{ij}= \Riemann \nu^{\al}x^{\be}_i x^{\ga}_k x^{\de}_j \eta_{\epsilon}x^{\epsilon}_lg^{kl}.
\eeq
We shall show that the symmetrization $\tilde a_{ij}=\frac{1}{2}(a_{ij}+a_{ji})$ of $a_{ij}$ satisfies 
\beq \label{5.4.18}
-c\tilde v^3g_{ij}\le \tilde a_{ij} \le c \tilde v^3g_{ij}
\eeq
with a uniform constant c. We have $F^{ij}\tilde a_{ij}= F^{ij}a_{ij}$, and assuming (\ref{5.4.18}) as true the claim then follows by chosing a coordinate system such that $g_{ij}=\delta_{ij}$ and $\tilde a_{ij}=diagonal$. 

Now we prove (\ref{5.4.18}). For this let $e_r$, $1\le r\le n$, be an orthonormal basis of $T_p(M(t))$ and let $\lambda^re_r$ be an arbitrary vector in $T_p(M(t))$ then we have with $e_r = (e_r^i)$ that
\beq
|\tilde a_{ij}\lambda^re^i_r\lambda^s e^j_s| \le n \max_{r,s}|\tilde a_{ij}e^i_re^j_s| \sum_r|\lambda^r|^2
\eeq
and
\beq
g_{ij}\lambda^re^i_r\lambda^s e^j_s  = \sum_r|\lambda^r|^2
\eeq
so that it will suffice to show that
\beq \label{last_suffice}
\max_{r,s}|\tilde a_{ij}e^i_re^j_s| \le c \tilde v^3
\eeq
for some special choice of orthonormal basis $e_r$.

To prove (\ref{last_suffice}) we may assume $Du\ne 0$ so that we can specialize our orthonormal basis by requiring that
\beq
e_1 = \frac{Du}{\|Du\|},
\eeq
here more precisely we had to write down the contravariant version of $Du$.

For $2\le k\le n$, the $e_k$ are also orthonormal with respect to the metric $\sigma_{ij}$ and it is also valid that
\beq
\sigma_{ij}\check u^ie^j_k = 0 \quad \forall 2 \le k \le n.
\eeq

In view of (\ref{5.4.24}) and the symmetry properties of the Riemann curvature tensor we have
\beq \label{5.4.25}
a_{ij}u^j=0.
\eeq 

Next we shall expand the right side of (\ref{5.4.17}) explicitly yielding
\bea \label{5.4.26}
a_{ij}=&\bar R_{0i0j}\tilde v \|Du\|^2 +\bar R_{0ik0}\tilde v u_ju^k + \bar R_{0ikj}\tilde vu^k \\
&+\bar R_{l0k0}\tilde vu^k\check u^lu_iu_j + \bar R_{l00j}\tilde v \check u^lu_i\|Du\|^2 \\
&+ \bar R_{lokj}\tilde v u^k\check u^lu_i  +\bar R_{li0j}\tilde v \check u^l\|Du\|^2 \\
&+\bar R_{lik0}\tilde vu^k\check u^lu_j +\bar R_{likj}\tilde v u^k\check u^l.
\eea

For $2\le r,s\le n$, we deduce from (\ref{5.4.26})
\bea
a_{ij}e^i_re^j_s =& \bar R_{0i0j}\tilde v\|Du\|^2e^i_re^j_s + \bar R_{0ikj}\tilde v u^ke^i_re^j_s\\
&+ \bar R_{li0j}\tilde v \check u^l\|Du\|^2e^i_re^j_s + \bar R_{likj}\tilde v u^k\check u^le^i_re^j_s
\eea
and hence
\beq
|a_{ij}e^i_re^j_s| \le c \tilde v^3 \quad \forall \ 2 \le r,s \le n.
\eeq
It remains to estimate $a_{ij}e^i_1e^j_r$ for $2\le r \le n$ because of (\ref{5.4.25}).

We deduce from (\ref{5.4.26})
\beq
a_{ij}e^i_1e^j_r = \bar R_{0i0j}\tilde v \|Du\|^2\tilde v^{-2}e^i_1e^j_r + \bar R_{0ikj}\tilde v^{-1}u^ke^i_1e^j_r,
\eeq
where we used the symmetry properties of the Riemann curvature tensor.

Hence, we conclude
\beq
|a_{ij}e^i_1e^j_r|\le c \tilde v^2 \quad \forall 2 \le r\le n,
\eeq
and the relation (\ref{last_suffice}) is proved.

{\bf Proof of (iv):} Differentiating the equation
\beq
\tilde v^2 = 1 + \|Du\|^2
\eeq
 with respect to $i$ yields
\beq
0=2 \tilde v \tilde v_i = 2 u_{ij}u^j
\eeq
which implies in view of 
\beq
\tilde v h_{ij}= -u_{ij}+ \bar h_{ij},
\eeq
cf. Section \ref{Section_2},
that 
\beq
h_{ij}u^j = \tilde v \bar h_{ij}\check u^j
\eeq
hence
\bea
\psi_{\al}x^{\al}_kh^k_iu^i = & \psi_{\al}x^{\al}_kg^{kl}h_{li}u^i =  \tilde v\psi_{\al}x^{\al}_kg^{kl}\bar h_{li}\check u^i \\
 = & \tilde v\psi_{\al}x^{\al}_k(\sigma^{kl}+ \tilde v^2 \check u^k\check u^l)\bar h_{li}\check u^i. %\\
 %= &  \tilde v(\psi_0 u_k+ \psi_k )(\sigma^{kl}+ \tilde v^2 \check u^k\check u^l)\bar h_{li}\check u^i 
\eea
Applying Schwarz inequality finishes the proof.
\end{proof}
 
\begin{lem} \label{tilde_v_uniformly_bounded}
$\tilde v$ is uniformly bounded on $[0, T^*)$ namely
\beq
\sup_{[0, T^*)}\tilde v \le c=c(\sup_{M_0}\tilde v, (N, \breve g_{\al \be})).
\eeq
\end{lem}
\begin{proof} We have (\ref{M_0_weit_genug_in_Zukunft}) in mind.
For $0<T<T^*$ assume that there are $0<t_0 \le T$ and $x_0\in S_0$ such that
\beq
\sup_{[0,T]} \sup_{M(t)} \tilde v = \tilde v (t_0, x_0) \ge 2.
 \eeq

In $(t_0, x_0)$ we have $\|Du\|^2 \ge \frac{1}{4}\tilde v^2$,
\beq
0 \le \dot {\tilde v} - \Fz\Fij\tilde v_{ij},
\eeq
and after multiplying this inequality by $F^2$ we get if $\epsilon>0$ sufficiently small that
\bea
0 \le & -\Fij h_{kj}h^k_i \tilde v + \Fij \bar R_{\al\be\ga\de} \nu^{\al}x^{\be}_i x^{\ga}_l x^{\de}_j u^l - \Fij h_{ij}\eta_{\al \be}\nu^{\al}\nu^{\be}\\
& - F \eta_{\al \be} \nu^{\al}\nu^{\be} - \Fij \eta_{\al \be \ga}\nu^{\al}x^{\be}_i x^{\ga}_j + 2 \Fij \eta_{\al \be}x^{\al}_k x^{\be}_i h^k_j \\
&+ \tilde v f^{''} \| Du\|^2 \Fg + \tilde v_k u^k f^{'}\Fg - \psi_{\al\be} \nu^{\al}x^{\be}_k u^k \Fg \\
&- \psi_{\al}x^{\al}_l h^l_k u^k \Fg \\
\le &  -\frac{1}{2}\Fij h_{kj}h^k_i \tilde v + c \tilde v^3 |f^{'}|\Fg  + \tilde v f^{''} \| Du\|^2 \Fg,
\eea
which is a contradiction if $\epsilon >0$ very small. 

Hence 
\beq
\tilde v(t_0, x_0) \le \max (\sup_{M_0}\tilde v, 2).
\eeq
\end{proof}

We prove a decay property of certain tensors.
\begin{lem} \label{arbitrary_decay_of_special_function}
(i) Let $\varphi\in C^{\infty}([a, 0))$, $a<0$, and assume 
\beq \label{converge_0}
\lim_{\tau \rightarrow 0}\varphi^{(k)}(\tau) =0 \quad \forall k \in \mathbb{N},
\eeq
then for every $k \in \mathbb{N}$ there exists a $c_k>0$ such that
\beq
|\varphi(\tau)| \le c_k |\tau|^k.
\eeq
(ii)  Let $T$ be a tensor such that for all $k \in \mathbb{N}$
\beq \label{tensor_uniform_to_0}
\norm{D^k T(x^0,x)} \longrightarrow 0 \quad \text{as}  \quad x^0 \longrightarrow 0 \quad  \text{uniformly in }x 
\eeq
then
\beq \label{spezielle_Gleichung_in_2}
\forall_{k \in \mathbb{N}} \quad \exists_{c_k >0} \quad \forall_{x\in S_0} \quad \norm{ T(x^0,x)} \le c_k |x^0|^k
\eeq
(iii) For $T=(\eta_{\al\be})$ the relation (\ref{spezielle_Gleichung_in_2}) is true, analogously for $\norm{ \eta_{\al \be \ga}}$, $\norm{D\psi}$, $\norm{\bar R_{\al \be\ga\de}\eta^{\al}}$, or more generally for any tensor that would vanish identically, if it would have been formed with respect to the product metric
\beq
-(dx^0)^2+\bar \sigma_{ij}dx^idx^j.
\eeq
\end{lem}
\begin{proof}
(i) From the assumptions it follows that
\beq \label{less_or_equ_k}
\sup_{[a,0)}|\varphi^{(k)}| \le c_k.
\eeq
From the mean value theorem we get
\bea
\sup_{[\tau, \tau_0]}|\varphi^{(k)}| \le |\varphi^{(k)}(\tau_0)| +|\tau|\sup_{[\tau, \tau_0]}|\varphi^{(k+1)}|
\eea
and therefore
\bea
\sup_{[\tau, \tau_0]}|\varphi| \le \sum_{l=0}^{k-1}|\tau|^l|\varphi^{(l)}(\tau_0)|+|\tau|^{k}\sup_{[\tau, \tau_0]}|\varphi^{(k)}|,
\eea
hence taking the limit $\tau_0 \rightarrow 0$ yields
\beq
|\varphi(\tau)| \le c_k |\tau|^k.
\eeq
(ii) For simplicity we only consider $T=(T^{\al})$. Choose $x\in S_0$ arbitrary and define
\beq
\varphi(\tau) = \norm{T(\tau,x)}^2 = T^{\al}T^{\be}\stackrel{+}{\bar g}_{\al\be}
\eeq
then we have
\beq
\varphi^{(1)}(\tau) = 2T^{\al}_{;\ga}\eta^{\ga}T^{\be}\stackrel{+}{\bar g}_{\al\be}+T^{\al}T^{\be}\stackrel{+}{\bar g}_{\al\be ; \de}\eta^{\de}
\eeq
so that one easily checks that $\varphi$ satisfies  (\ref{converge_0}) and (\ref{less_or_equ_k}) with $c_k$ not depending on $x$. The claim now follows by (i).

(iii) The tensor $T=\eta_{\al\be}$ is a covariant derivative of $\eta_{\al}$ with respect to the metric $\bar g_{\al\be}$. If we would have calculated this covariant derivative with respect to the limit metric
\beq
-(dx^0)^2 + \bar \sigma_{ij}(x)dx^idx^j
\eeq
then it would vanish identically, as well as all its derivatives of arbitrary order. From this together with the convergence properties of $\bar g_{\al\be}$ we deduce that $T$ satisfies the assumptions in (ii), so that the claim follows. The remaining estimates are similarly proved via (ii).
\end{proof}

Now we prove a result for general convex, spacelike graphs.

\begin{lem}\label{gerhardt_tipp_bei_C1}
Let $\epsilon>0$ be arbitrary, then there exists $\de=\de((N, \breve g_{\al\be}), \epsilon) >0$ such that for every closed, spacelike, convex hypersurface $M$ in the end $N_{\de}^+=\{x^0>-\de\}$ holds
\beq
\tilde v \le \epsilon |f^{'}|^{\frac{1}{\tilde \ga}}.
\eeq
%\beq
%\tilde v \le \begin{cases} \frac{|f^{'}|}{c_1}, & \tilde \ga \ge 1 \\ \frac{|f^{'}|^p}{c_1}, & \tilde \ga < 1. \end{cases}
%\eeq

\end{lem}
\begin{proof}
Let $p>\tilde \ga^{-1}$ and define
\beq \label{ansatz_bound_of_tilde_v}
w=\tilde v\{e^f+|u|^p\}
\eeq
and look at a point, where $w$ attains its maximum, and infer
\bea
0 = w_i &= \tilde v_i \{e^f+|u|^p\} + \tilde v \{e^ff^{'}-p|u|^{p-1}\}u_i \\
&= \{-h_{ik}u^k + \tilde v^{-1}\bar h_{ik}u^k\} \{e^f+|u|^p\} + \tilde v \{e^ff^{'}-p|u|^{p-1}\}u_i \\
&= \{-\breve h_{ik}u^ke^{-\tilde \psi}-\tilde v f^{'}u_i + \tilde \epsilon \tilde v u_i\} \{e^f+|u|^p\} + \tilde v \{e^ff^{'}-p |u|^{p-1}\}u_i, 
\eea
where
\beq
|\tilde \epsilon| \le c_m|u|^m \quad \forall \ m \in \mathbb{N}.
\eeq
Multiplying by $u^i$ and assuming $Du\ne 0$ we get the inequality 
\bea
0 &\le (-f^{'} + \tilde \epsilon) \{e^f+|u|^p\} +e^ff^{'}-p|u|^{p-1} \\
& = -f^{'}|u|^p + \tilde \epsilon \{e^f+|u|^p\} -p|u|^{p-1} <0,
\eea
if $\de>0$ small, since
\beq
f^{'}u \le \tilde \ga^{-1}+cu^2.
\eeq
This is a contradiction, hence $Du=0$.

Since
\beq
\varphi(\tau) = e^{f(\tau)}+|\tau|^p, \quad a \le \tau <0,
\eeq
is monotone decreasing we conclude
\beq
\tilde v \le \frac{e^{f(u_{\text{min}})}+|u_{\text{min}}|^p}{e^{f(u)}+|u|^p}  
\le (e^{f(u_{\text{min}})}+|u_{\text{min}}|^p) e^{-f(u)},
\eeq
where  $u_{\text{min}}=\inf u$. Choosing $\de$ appropiately small finishes the proof, where we used Lemma \ref{absolute_basic_konvergenzproperty_f} (ii).
\end{proof}

\begin{rem}
We also could have chosen
\beq
w=\tilde v\{|u|^{\frac{1}{\tilde \ga}}+|u|^p\}
\eeq
in (\ref{ansatz_bound_of_tilde_v}).
\end{rem}

\begin{cor} \label{corollary_eines_der_letzten}
Let $\de>0$ be small and $N_{\de}^+$ and $M$ be as in Lemma \ref{gerhardt_tipp_bei_C1}, then
\beq
\Fij \bar R_{\al\be\ga\de} \nu^{\al}x^{\be}_i \nu^{\ga} x^{\de}_j \ge -c\de \Fg,
\eeq
if the limit metric $\bar \sigma_{ij}$ has non-negative sectional curvature.
\end{cor}
\begin{proof} We define
\beq
\bar R_{\al\be\ga\de}(0, \cdot) = \lim_{\tau \uparrow 0}\bar R_{\al\be\ga\de}(\tau, \cdot)
\eeq
and have
\bea
\Fij \bar R_{\al\be\ga\de} &\nu^{\al}x^{\be}_i \nu^{\ga} x^{\de}_j \\
= &\Fij (\bar R_{\al\be\ga\de}(0, \cdot)+\bar R_{\al\be\ga\de}(u, \cdot)-\bar R_{\al\be\ga\de}(0, \cdot)) \nu^{\al}x^{\be}_i \nu^{\ga} x^{\de}_j \\
\ge & \Fij(\bar R_{\al\be\ga\de}(u, \cdot)-\bar R_{\al\be\ga\de}(0, \cdot)) \nu^{\al}x^{\be}_i \nu^{\ga} x^{\de}_j\\
%\ge & -\norm{\Fij \nu^{\al}x^{\be}_i \nu^{\ga} x^{\de}_j} \int_{u}^0\norm{\bar R_{\al\be\ga\de;0}} \\
\ge &-\norm{\Fij \nu^{\al}x^{\be}_i \nu^{\ga} x^{\de}_j} \cdot \norm{\bar R_{\al\be\ga\de}(u, \cdot)-\bar R_{\al\be\ga\de}(0, \cdot)}\\
\ge & -c_m|u|^m\Fg,
\eea
for arbitrary $m\in \mathbb{N}$ and suitable $c_m$. Note that we used for the last inequality that
\beq
\bar R_{\al\be\ga\de}(x^0, \cdot)-\bar R_{\al\be\ga\de}(0, \cdot)
\eeq
satisfies (\ref{tensor_uniform_to_0}).
\end{proof}

We want to formulate the relation of the curvature tensors for conformal metrics.

\begin{lem} \label{conformal_metric_curvature_tensor_relation}
The curvature tensors of the metrics $\breve g_{\al \be}, \bar g_{\al \be}$ are related by
\bea
e^{-2 \tilde \psi}\breve R_{\al \be \ga \de} = & \Riemann - \bar g_{\al\ga} \tilde \psi_{\be \de} -
\bar g_{\be\de}\tilde \psi_{\al\ga} + \bar g_{\al\de}\tilde \psi_{\be \ga}+ \bar g_{\be\ga}\tilde \psi_{\al \de} \\
& +\bar g_{\al \ga}\tp_{\be}\tp_{\de} + \bg_{\be\de}\tp_{\al}\tp_{\ga}-\bg_{\al\de}\tp_{\be}\tp_{\ga}-\bg_{\be\ga}\tp_{\al}\tp_{\de} \\
&+ \{\bg_{\al\de}\bg_{\be\ga}- \bg_{\al\ga}\bg_{\be\de}\}\|D\tp\|^2.
\eea
\end{lem}

Now we are able to prove the following lemma which is necessary for the $C^2$-estimates in the next section.

\begin{lem} \label{Schnittkruemmung_positiv}
There exists a constant $\tilde c>0$ such that we have for the leaves of the IFCF 
\beq \label{absolut_after_korrektur_ergaenzung}
\breve F^{ij}\breve R_{\al \be \ga \de}\breve\nu^{\al}x^{\be}_i\breve\nu^{\ga}x^{\de}_j\ge \tilde c |f^{'}|^2e^{-2 \tilde \psi}
\eeq
provided
\beq \label{M_0_is_sufficiently_far_in_the_future_of_N}
- \epsilon < \inf_{M_0} x^0 < 0,
\eeq
where $\epsilon =  \epsilon (N, \breve g_{\al\be})$. Here $\breve F^{ij}$ is evaluated at $\breve h^j_ i$. 
\end{lem}
\begin{proof}
In view of the homogeneity of $F$ we have  
\beq
F^i_j = \breve F^i_j,
\eeq
hence
\beq
F^{ij} = e^{2\tilde \psi}\breve F^{ij}.
\eeq
 We have due to Lemma \ref{conformal_metric_curvature_tensor_relation}
\bea \label{summand_cosmo}
e^{2 \tilde \psi} \breve F^{ij}\breve R_{\al \be \ga \de}&\breve\nu^{\al}x^{\be}_i\breve\nu^{\ga}x^{\de}_j  \\
=& F^{ij}\bar R_{\al \be \ga \de} \nu^{\al}x^{\be}_i\nu^{\ga}x^{\de}_j + \Fij x^{\be}_i x^{\de}_j\tp_{\be\de} - \Fij g_{ij}\tp_{\al\ga}\nu^{\al}\nu^{\ga}\\
& - \Fij x^{\be}_i x^{\de}_j \tp_{\be}\tp_{\de} + \Fij g_{ij} \tp_{\al}\tp_{\ga}\nu^{\al}\nu^{\ga} + \Fg \|D\tp\|^2.
\eea 

We have
\beq
\stackrel{+}{g}_{ij} \le 2 \sigma_{ij}  \le 2 \tilde v^2g_{ij}.
\eeq   

Now we estimate each summand in (\ref{summand_cosmo}) separately with the help of the Riemmanian background metric $\stackrel{+}{\bar g}_{\al \be}$, namely
\bea
 |F^{ij}\bar R_{\al \be \ga \de} \nu^{\al}x^{\be}_i\nu^{\ga}x^{\de}_j| & \le c \tilde v^2 (\Fij F^{\bar i \bar j}\stackrel{+}{g}_{i\bar i}\stackrel{+}{g}_{j\bar j})^{\frac{1}{2}} \le c \tilde v^2 \Fij \sigma_{ij} \le c \tilde v^4 \Fij g_{ij},
\eea
\bea
 \Fij x^{\be}_i x^{\de}_j\tp_{\be\de} = \Fij u_i u_j f^{''} +  \Fij x^{\be}_i x^{\de}_j\psi_{\be\de} \ge \Fij u_i u_j f^{''} - c \tilde v^2 \Fg,
\eea
\bea
- \Fij g_{ij}\tp_{\al\ga}\nu^{\al}\nu^{\ga} & = -\tilde v^2 \Fg f^{''}- \Fij g_{ij}\psi_{\al\ga}\nu^{\al}\nu^{\ga}\\
& \ge -\tilde v^2 \Fg f^{''} - c \tilde v^2 \Fg,
\eea
\bea
- \Fij x^{\be}_i x^{\de}_j \tp_{\be}\tp_{\de} &= -\Fij u_i u_j (\psi_0 + f^{'})^2 - \Fij \psi_i \psi_j - 2 \Fij u_j \psi_i (\psi_0 + f^{'}) \\
&\ge -\Fij u_i u_j (\psi_0 + f^{'})^2 -c (1+|f^{'}||Du|) F^{ij}\sigma_{ij}|D\psi | \\
& \ge -\Fij u_i u_j (\psi_0 + f^{'})^2 -c \tilde v^2(1+|f^{'}||Du|) \Fg |D\psi| \\
& \ge -\Fij u_i u_j (\psi_0 + f^{'})^2 -c |f^{'}| \tilde v^2 \Fg,
\eea
where $|D\psi|^2 = \sigma^{ij}\psi_i \psi_j$,
\bea
\Fij g_{ij} \tp_{\al}\tp_{\ga}\nu^{\al}\nu^{\ga} & \ge \tilde v^2 (\psi_0+f^{'})^2 \Fg - c \tilde v^2 |f^{'}| \Fg, 
\eea
\bea
\Fg \|D\tp\|^2 &= -(f^{'}+ \psi_0)^2 \Fg+ \sigma^{ij}\psi_i \psi_j\Fg \\
& \ge -(f^{'}+ \psi_0)^2 \Fg - c \Fg.
\eea

Thus we conclude (using $u_i u_j \le (\tilde v^2 -1)g_{ij}$)
\bea \label{nachherbenannt_schnitt}
e^{2 \tilde \psi} \breve F^{ij}\breve R_{\al \be \ga \de}\breve\nu^{\al}x^{\be}_i\breve\nu^{\ga}x^{\de}_j \ge
&-c\tilde v ^4 \Fg + F^{ij}u_iu_j f^{''}- \tilde v^2 f^{''}\Fg \\
& - c \tilde v^2 |f^{'}|\Fg \\
 &  +(\psi_0 + f^{'})^2F^{ij}(\tilde v^2g_{ij}-u_iu_j - g_{ij})\\
 \ge & -c\tilde v ^4 \Fg -\tilde v^2 f^{''}\Fg-c|f^{'}| \tilde v^2 \Fg.
\eea
Now, the claim follows with Lemma \ref{gerhardt_tipp_bei_C1} if $\tilde \ga \ge 1$, cf. (\ref{fzweistrich_wiefstrich_quadrat}). 

Let us now consider the case $\tilde \ga <1$. Due to assumption the limit metric $\bar \sigma_{ij}$ has non-negative sectional curvature. Now we use Corollary \ref{corollary_eines_der_letzten} to bound the first summand of the right side of (\ref{summand_cosmo}) from below by the term $-c\Fij g_{ij}$, one easily checks that this term replaces the summand with $\tilde v^4$ in (\ref{nachherbenannt_schnitt}) completing the proof.
\end{proof}

\begin{rem}
Lemma \ref{Schnittkruemmung_positiv} is also true for general convex, spacelike graphs over $S_0$ in a future end of $N$, we did not use in the proof that the hypersurfaces are flow hypersurfaces of the IFCF.
\end{rem}

Before we consider the $C^2$-estimates in the next section we show that $N$ satisfies the timelike convergence condition with respect to the future.

\begin{cor} \label{timelike_convergence}
Lemma \ref{Schnittkruemmung_positiv} remains valid, if we replace inequality (\ref{absolut_after_korrektur_ergaenzung}) by
\beq
\breve R_{\al \be}\breve \nu^{\al} \breve \nu^{\be} \ge \tilde c |f^{'}|^2e^{-2\tilde \psi}
\eeq
\end{cor}
\begin{proof}
We substitute $\breve F^{ij}$ by $\breve g^{ij}$ and $F^{ij}$ by $g^{ij}$ in the proof of Lemma \ref{Schnittkruemmung_positiv}. The proof even simplifies, since we have the estimate
\beq
|g^{ij} \Riemann \nu^{\al}x^{\be}_i\nu^{\ga}x^{\de}_j| = |\Ricci \nu^{\al}\nu^{\be}| \le c \tilde v^2,
\eeq
especially the assumption, that the limit metric $\bar \sigma_{ij}$ has non-negative sectional curvature in case $\tilde \ga < 1$, is not needed.
\end{proof}
%%%%%%%%%%%%%%%%%%%%%%%
\auskommentieren{
\begin{lem} \label{timelike_convergence_b}
In a future end of $N$, i.e. in points where $x^0$ is sufficiently close to 0, we have
\beq
\breve R_{\al\be}\breve \nu^{\al}\breve \nu^{\be}\ge 0 \quad \forall \braket{\breve \nu, \breve \nu} =-1
\eeq
\end{lem}
\begin{proof}
We have due to a well-known formular
\beq
\breve R_{\al\be}=\bar R_{\al\be}-(n-1)[\tilde \psi_{\al\be}-\tilde \psi_{\al}\tilde \psi_{\be}]-\bar g_{\al\be}[\Delta \tilde \psi+(n-1)\|D \tilde \psi\|^2]
\eeq
hence
\bea
e^{2 \tilde \psi}\breve R_{\al\be}\breve \nu^{\al} \breve \nu^{\be} =& \bar R_{\al \be}\nu^{\al}\nu^{\be} - (n-1)[\tilde \psi_{\al\be}-\tilde \psi_{\al}\tilde \psi_{\be}]\nu^{\al}\nu^{\be} \\
&-\bar g_{\al\be}\nu^{\al}\nu^{\be}[\Delta \tilde \psi + (n-1)\|D\tilde \psi\|^2]\\
\ge& -\frac{n}{2}f^{''}
\eea
in a future end of $N$.
\end{proof}
}
%%%%%%%%%%%%%%%%%%%%%%

\section{$C^2$-estimates\---Existence for all times} \label{Section_5}
In this section we consider $N$ to be equipped only with the metric $\breve g_{\al\be}$ and will\---for simplicity\---apply standard notation to this case, i.e. no $\breve \ $ is written down. In the next section we will go back to the notation of the previous section until the end oft this paper.

\begin{lem} \label{evol_equat_1_over_F_outer_metric}
The following evolution equation holds
\bea
 \frac{d}{dt}\left(\frac{1}{F}\right)- \frac{1}{F^2}F^{ij}\left(\frac{1}{F}\right)_{ij} = - \frac{1}{F^3} \Fij h_{ik}h^k_j-\frac{1}{F^3}\Fij \Riemann \nu^{\al}x^{\be}_i\nu^{\ga}x^{\de}_j.
\eea
\end{lem}
\begin{proof}
cf. \cite[Lemma 2.3.4]{CP}.
\end{proof}

\auskommentieren{
\begin{lem} \label{F_to_infty}
Assume (\ref{M_0_is_sufficiently_far_in_the_future_of_N}), then there exists a positive constant $c_0 = c_0(M_0)$ and $\epsilon >0$, such that the estimate
\beq
F \ge c_0 e^{\epsilon t}
\eeq
holds as long as the flow exists.
\end{lem}
\begin{proof}
Let $\varphi = F^{-1}e^{\epsilon t}$, $\epsilon > 0$ sufficiently small, then $\varphi$ satisfies the inequality
\beq
\dot \varphi - F^{-2}\Fij \varphi_{ij} \le - \frac{1}{F^2} \Fij h_{ik}h^k_j \varphi + \epsilon \varphi \le  (-\epsilon_0 +\epsilon) \varphi < 0,
\eeq
hence we conclude 
\beq
\varphi \le \sup_{M_0}\varphi = \sup_{M_0}F^{-1}.
\eeq
\end{proof}

Although not necessary for the following we show that a similar result as Lemma \ref{F_to_infty} can be proved if we assume only $F\in (K)$ instead of $F \in (K^{*})$.
}
\begin{lem} \label{F_to_infty}
Assume (\ref{M_0_is_sufficiently_far_in_the_future_of_N}), then 
\beq
F \ge \inf_{M_0}F
\eeq
as long as the flow exists.
If in addition the IFCF exists for all times, there even holds
\beq
F \ge c_0 e^{(\ga+\frac{1}{n}) t}
\eeq
with $c_0=c_0(M_0)>0$.
\end{lem}

\begin{proof}
We define 
\beq
\varphi (t) = \inf_{M(t)} F
\eeq
and infer from Lemma \ref{evol_equat_1_over_F_outer_metric} 
\bea
\frac{d}{dt}F-F^{-2}\Fij F_{ij} =& \frac{1}{F}\Fij h_{ik}h^k_j+\frac{1}{F} \Fij \Riemann \nu^{\al} x^{\be}_i \nu^{\ga}x^{\de}_j\\
&-\frac{2}{F^3} \Fij F_iF_j, 
\eea
hence using Lemma \ref{Schnittkruemmung_positiv} we deduce 
\beq \label{drei}
\dot \varphi (t) \ge \tilde c \frac{|f^{'}|^2}{F}e^{-2f}, 
\eeq
especially $\dot \varphi (t)\ge 0$ for a.e. $0<t<T^{*}$.

If the flow exists for all times, we know from Remark \ref{Fluss_laeuft_in_Zukunft_wenn_fuer_alle_Zeiten_existent} that the flow runs into the future singularity
\beq
\lim_{t\rightarrow \infty}\inf u(t, \cdot)=0.
\eeq
A careful view of the proofs of Lemma \ref{optimal_exponential_decay_for_u_help_lemma} and Theorem \ref{optimal_exponential_decay_for_u} shows that everything needed there is available at this point, so that we infer from (\ref{drei})
\beq
\dot \varphi(t) \ge \tilde c \frac{|f^{'}|^2}{\varphi}e^{-2f}
\eeq
and
\beq
\frac{d}{dt}(\varphi^2)\ge c e^{2 (\ga+\frac{1}{n}) t}
\eeq
for a.e. $t>0$ and a positive constant $c>0$. This implies
\beq
\varphi(t)^2\ge \varphi(0)^2+ \frac{c}{2 (\ga+\frac{1}{n})}(e^{2 (\ga+\frac{1}{n}) t}-1)
\eeq
for all $t>0$.
\end{proof}

\begin{rem} \label{chi}
Due to \cite[Lemma 1.8.3]{CP},  and the remark at the beginning of Section \ref{Section_3}, especially inequality (\ref{convex_coordinate_slice}), for every relative compact subset $\Omega$ of $N$ lying sufficiently far in the future of $N$, i.e. $|\inf_{\Omega}x^0|$
close to 0,
there exists a strictly convex function $\chi \in C^2(\bar \Omega)$, this means
\beq
\chi_{\al \be}\ge c_0 \bar g_{\al \be}
\eeq
with a constant $c_0>0$.
\end{rem}
\begin{lem}
The following evolution equation holds
\bea \label{chi_evolution}
\dot \chi - \frac{1}{F^2}\Fij \chi_{ij} = - \frac{2}{F}\chi_{\al}\nu^{\al}-\frac{1}{F^2}\Fij \chi_{\al\be}x^{\al}_i x^{\be}_j
\eea
\end{lem}
\begin{proof}
Direct calculation.
\end{proof}
\begin{lem}
The following evolution equation holds
\bea \label{logF_evolution}
(\log F)^{'} - \frac{1}{F^2}F^{ij}(\log F)_{ij} =& \frac{1}{F^2}F^{ij}h_{ik}h^k_j + \frac{1}{F^2} \Fij \bar R_{\al \be \ga \de} {\nu}^{\al}x^{\be}_i{\nu}^{\ga}x^{\de}_j\\
&-\frac{1}{F^4}\Fij F_i F_j
\eea
\end{lem}
\begin{proof}
Use Lemma \ref{evol_equat_1_over_F_outer_metric}.
\end{proof}
\begin{lem}
The following evolution equation holds
\bea \label{v_evolution}
\dot {\tilde v} - \frac{1}{F^2}\Fij \tilde v_{ij}  = & - \frac{1}{F^2}\Fij h_{ik}h^k_j \tilde v - \frac{2}{F}\eta_{\al \be}\nu^{\al}\nu^{\be} - \frac{2}{F^2}\Fij h^k_j x^{\al}_ix^{\be}_k \eta_{\al \be} \\
& - \frac{1}{F^2}\Fij \eta_{\al \be \ga}x^{\be}_ix^{\ga}_j\nu^{\al} - \frac{1}{F^2}\Fij \bar R_{\al \be \ga \de}\nu^{\al}x^{\be}_ix^{\ga}_k x^{\de}_j\eta_{\epsilon}x^{\epsilon}_l g^{kl},
\eea
where $(\eta_{\al})=e^{\tilde \psi}(-1, 0, ..., 0)$.
\end{lem}
\begin{proof}
cf. \cite[Lemma 2.4.4]{CP}.
\end{proof}

\begin{lem}
Let $\Omega \subset N$ be precompact and assume that the flow stays in $\Omega$ for $0\le t \le T < T^*$, then the $F$-curvature of the flow hypersurfaces is bounded from above,
\bea
0<F< c(\Omega).
\eea
\end{lem}

\begin{proof}
Consider the function 
\beq
w = \log F + \lambda \tilde v + \mu \chi,
\eeq
where $\lambda, \mu > 0$ will be chosen later appropiately. Assume
\beq
w(t_0, x_0) = \sup_{[0,T]} \sup_{M(t)} w 
\eeq
with $0 < t_0 \le T$, then we have in $(t_0, x_0)$
\bea
0 \le & \dot w - \frac{1}{F^2}\Fij w_{ij}  \\
= & \frac{1}{F^2}F^{ij}h_{ik}h^k_j + \frac{1}{F^2} \Fij \bar R_{\al \be \ga \de} {\nu}^{\al}x^{\be}_i{\nu}^{\ga}x^{\de}_j-\frac{1}{F^4}\Fij F_i F_j \\
& - \frac{\lambda}{F^2}\Fij h_{ik}h^k_j \tilde v - \frac{2\lambda}{F}\eta_{\al \be}\nu^{\al}\nu^{\be} - \frac{2\lambda}{F^2}\Fij h^k_j x^{\al}_ix^{\be}_k \eta_{\al \be} \\
& - \frac{\lambda}{F^2}\Fij \eta_{\al \be \ga}x^{\be}_ix^{\ga}_j\nu^{\al} - \frac{\lambda}{F^2}\Fij \bar R_{\al \be \ga \de}\nu^{\al}x^{\be}_ix^{\ga}_k x^{\de}_j\eta_{\epsilon}x^{\epsilon}_l g^{kl} \\
 & - \frac{2\mu}{F}\chi_{\al}\nu^{\al}-\frac{\mu}{F^2}\Fij \chi_{\al\be}x^{\al}_i x^{\be}_j \\
 \le & - \epsilon_0 (\frac{\lambda}{2} -1) \tilde v + \frac{c\lambda }{F^2} \Fg + c (\mu + \lambda) \frac{1}{F}- c_0 \frac{\mu}{F^2}\Fg.
\eea
Now we choose $\lambda > 2$ arbitrary and $\mu>> 1$ large and we deduce that $F$ is a priori bounded from above in $(t_0, x_0)$ from which we conclude the Lemma.
\end{proof}

Let $\Omega \subset N$ be precompact and assume that the flow stays in $\Omega$ for $0\le t \le T < T^*$, then there exist\---as we have just proved\---constants $0<c_1(\Omega)<c_2(\Omega)$ such that 
\beq \label{F_bound}
c_1(\Omega) < F < c_2(\Omega)
\eeq
(concerning the lower bound we proved even more, cf. Lemma \ref{F_to_infty}).
It remains to prove that there also holds an estimate for the principal curvatures from above
\beq
\kappa_i \le c_3(\Omega),
\eeq 
yielding 
\beq
0 < c_4(\Omega) \le \kappa_i \le c_3(\Omega)
\eeq
due to the convexity of the flow hypersurfaces and (\ref{F_bound}).

\begin{lem}
The mixed tensor $h^j_i$ satisfies the parabolic equation
\bea \label{h_nn_evolution}
\dot h^j_i - \frac{1}{F^2}&F^{kl}h^j_{i;kl}  = -\Fz F^{kl}h_{rk}h^r_lh^j_i + \frac{1}{F}h_{ri}h^{rj} + \frac{1}{F}h^k_ih^j_k  \\
& + \frac{1}{F^2}F^{kl, rs}h_{kl;i}h_{rs}^{;j}  -\frac{2}{F^3}F_iF^j + \frac{2}{F^2}F^{kl}\bar R_{\al \be \ga \de}x^{\al}_mx^{\be}_i x^{\ga}_kx^{\de}_rh^m_lg^{rj} \\
& - \frac{1}{F^2}F^{kl}\bar R_{\al \be \ga \de} x^{\al}_mx^{\be}_kx^{\ga}_rx^{\de}_lh^m_ig^{rj}
- \frac{1}{F^2}F^{kl}\bar R_{\al\be\ga\de}x^{\al}_m x^{\be}_kx^{\ga}_ix^{\de}_lh^{mj} \\
& - \frac{1}{F^2}F^{kl}\bar R_{\al \be \ga \de}\nu^{\al}x^{\be}_k\nu^{\ga}x^{\de}_lh^j_i + \frac{2}{F} \bar R_{\al \be\ga\de}\nu^{\al}x^{\be}_i\nu^{\ga}x^{\de}_mg^{mj} \\
&+ \frac{1}{F^2}F^{kl}\bar R_{\al \be\ga\de; \epsilon}\left\{\nu^{\al}x^{\be}_k x^{\ga}_lx^{\de}_ix^{\epsilon}_mg^{mj} + \nu^{\al}x^{\be}_ix^{\ga}_kx^{\de}_mx^{\epsilon}_lg^{mj}\right\}.
\eea
\end{lem}

\begin{proof}
cf. \cite[Lemma 2.4.1]{CP}.
\end{proof}
\begin{lem}
Let $\Omega \subset N$ be precompact and assume that the flow stays in $\Omega$ for $0\le t  < T^*$, then there exists $c_3(\Omega)$ such that
\beq
\kappa_i \le c_3(\Omega).
\eeq
\end{lem}
\begin{proof}
Let $\varphi$ and $w$ be defined respectively by
\bea \label{ansatz_hnn_nach_oben}
\varphi & = \sup\{h_{ij}\eta^i\eta^j : \|\eta\| = 1\}, \\
w & = \log \varphi + \lambda \tilde v + \mu \chi,
\eea
where $\lambda, \mu$ are large positive parameters to be specified later. We claim that $w$ is bounded for a suitable choice of $\lambda, \mu$.

Let $0<T<T^*$, and $x_0 = x_0(t_0)$, with $0<t_0 \le T$, be a point in $M(t_0)$ such that
\beq
\sup_{M_0}w < \sup \{\sup_{M(t)}w: 0<t \le T\} = w(x_0).
\eeq
We then introduce a Riemannian normal coordinate system $(\xi^i)$ at $x_0 \in M(t_0)$ such that at $x_0=x(t_0, \xi_0)$ we have
\beq
g_{ij}= \delta_{ij} \quad and \quad \varphi = h^n_n.
\eeq

Let $\tilde \eta = (\tilde \eta^i)$ be the contravariant vector field defined by
\beq
\tilde \eta = (0, ..., 0, 1),
\eeq
and set
\beq
\tilde \varphi = \frac{h_{ij}\tilde \eta^i \tilde \eta^j}{g_{ij}\tilde \eta^i \tilde \eta^j}.
\eeq
$\tilde {\varphi}$ is well defined in a neighbourhood of $(t_0, \xi_0)$.

Now, define $\tilde {w}$ by replacing $\varphi$ by $\tilde \varphi$ in (\ref{ansatz_hnn_nach_oben}); then $\tilde w$ assumes its maximum at $(t_0, \xi_0)$. Moreover, at $(t_0, \xi_0)$ we have
\beq
\dot {\tilde \varphi}= \dot h^n_n,
\eeq
and the spatial derivatives do also coincide; in short, at $(t_0, \xi_0)$ $\tilde \varphi$ satisfies the same differential equation (\ref{h_nn_evolution}) as $h^n_n$. For the sake of greater clarity, let us therefore treat $h^n_n$ like a scalar and pretend that $w$ is defined by
\beq
w = \log h^n_n + \lambda \tilde v + \mu \chi.
\eeq
W.l.o.g. we assume that $\mu$, $\lambda$ and $h^n_n$ are larger than 1.

At $(t_0, \xi_0)$ we have $\dot w \ge 0$ and in view of the maximum principle, we deduce from $(\ref{h_nn_evolution}), (\ref{v_evolution})$, $(\ref{chi_evolution})$ and (\ref{F_bound})
\bea \label{last_two_summands}
0  \le& ch^n_n + c \lambda \Fg - \frac{\lambda}{2}\epsilon_0 \tilde v \frac{H}{F} + \mu c - c_0 \frac{\mu}{F^2}\Fg \\
&+ \frac{1}{F^2}F^{ij}(\log h^n_n)_i(\log h^n_n)_j - \frac{2}{h^n_nF^3}F^nF_n + \frac{1}{h^n_nF^2} F^{kl,rs}h_{kl;n}h_{rs;i}g^{ni}.
\eea

Because of \cite[Lemma 2.2.6]{CP} we have 
\beq
F^{kl,rs}h_{kl;n}h_{rs;n} \le F^{-1}(\Fij h_{ij;n})^2 - \frac{1}{h^n_n}\Fij h_{in;n}h_{jn;n}
\eeq
so that we can estimate the last two summands of (\ref{last_two_summands}) from above by
\beq
-\frac{1}{(h^n_n)^2}\frac{1}{F^2}\Fij (h^n_{n;i} + \bar R_i)(h^n_{n;j}+\bar R_j);
\eeq
here 
\beq
\bar R_i = \Riemann \nu^{\al}x^{\be}_nx^{\ga}_ix^{\de}_n = h_{in;n}-h_{nn;i}
\eeq
denotes the correction term which comes from the Codazzi equation when changing the indices from $h_{in;n}$ to $h_{nn;i}$.

Thus the terms in (\ref{last_two_summands}) containing derivatives of $h^n_n$ are estimated from above by
\beq
-2 \frac{1}{(h^n_n)^2F^2}\Fij h^n_{n;i} \bar R_j = -2 \frac{1}{h^n_nF^2}\Fij (\log h^n_n)_i \bar R_j.
\eeq
Moreover $Dw$ vanishes at $\xi_0$, i.e., 
\bea
(\log h^n_n)_i & = -\lambda  \tilde v_i - \mu \chi_i \\
&= -\lambda \eta_{\al\be}x^{\be}_i\nu^{\al} -\lambda \eta_{\al}x^{\al}_kh^k_i - \mu \chi_{\al}x^{\al}_i.
\eea
Hence we conclude from (\ref{last_two_summands}) that
\bea \label{rest_hnn_from_above}
0 \le&  ch^n_n + c \lambda \Fg - \frac{\lambda}{2}\epsilon_0 \tilde v \frac{H}{F} + \mu c + \mu \frac{c}{h^n_n} \Fg- c_0 \frac{\mu}{F^2}\Fg \\
\le &  c_1h^n_n + c_2 \lambda \Fg - \lambda c_3 h^n_n + \mu c_4 + \mu \frac{c_5}{h^n_n} \Fg- c_0 \mu\Fg,
\eea
where $c_i$, $i=0, ..., 5$, are positive constants and the value of $c_0$ changed.
We note that we used the estimate
\beq
\Fij \bar R_j \eta_{\al}x^{\al}_k h^k_i \le c F, 
\eeq
which can be immediately proved.
 
Now suppose $h^n_n$ to be so large that
\beq
\frac{c_5}{h^n_n}< \frac{1}{2}c_0,
\eeq
and choose $\lambda, \mu$ such that
\beq
\frac{\lambda}{2} c_3 >c_1  \quad and \quad \frac{1}{4}c_0 \mu > c_2 \lambda
\eeq
yielding that estimating the right side of (\ref{rest_hnn_from_above}) yields
\beq
0 \le -\frac{\lambda}{2} c_3 h^n_n - \frac{c_0}{4}\mu \Fg + \mu c_4,
\eeq
hence $h^n_n$ is apriori bounded at $(t_0, \xi_0)$.
\end{proof}

\begin{rem} \label{flow_runs_into_future_singularity}
Now all neccessary apriori estimates are proved so that we can deduce existence of the flow for all times in the usual way. In view of Remark \ref{Fluss_laeuft_in_Zukunft_wenn_fuer_alle_Zeiten_existent} the flow runs into the future singularity.

The latter property can also be proved as follows. Using Lemma \ref{F_to_infty} and $F\le H$ we infer 
\beq
\infty \longleftarrow \inf_{M(t)}F \le \inf_{M(t)}H \quad as \quad t \longrightarrow \infty.
\eeq
The timelike convergence condition with respect to the future, cf. Corollary \ref{timelike_convergence}, together with 
\beq
\lim_{t \rightarrow \infty} \inf_{M(t)}H = \infty 
\eeq
implies that the flow runs into the future singularity. To see this we argue as in the proof of \cite[Lemma 4.2]{HK}.
\end{rem}

\section{$C^0$-estimates\---Asymptotic behaviour of the flow}\label{Section_6}
From now on until the end of this paper we go back to the notation introduced in Section \ref{Section_4} and consider the flow as embedded in $(N, \bar g_{\al\be})$, i.e. standard notations apply to this case.

%\subsection{Lower order estimates}
We prove that the flow runs exponentially fast into the future singularity, which means more precisely that there are constants $c_1, c_2>0$ such that
\beq
-c_1e^{-\ga t} < u< -c_2e^{-\ga t}.
\eeq
The first step for this will be the following Lemma.
\begin{lem} \label{optimal_exponential_decay_for_u_help_lemma}
Let $u$ be the scalar solution of the inverse F-curvature flow, then for every $0<\lambda < \ga$ there is $c(\lambda) >0$ such that 
\begin{equation}
| u e^{\lambda t}| \le c(\lambda).
\end{equation}
\end{lem}
\begin{proof}
Define
\begin{equation}
\varphi(t) = \inf_{x \in S_0}u(t,x)
\end{equation}
and 
\begin{equation}
w = \log(-\varphi) + \la t.
\end{equation}
In $x_t$ we have, we remind that $h_{ij} = -u_{ij} - \frac{1}{2}\dot \sigma_{ij}$, 
\bea
F =& F(h_{ij}-\tilde v f^{'}g_{ij}+\psi_{\al}\nu^{\al}g_{ij}) \\
\le & F(cg_{ij}-f^{'}g_{ij}) \quad (\text{where } c >0)\\
=& (c-f^{'})F(g_{ij}) \\
=&n(c-f^{'})
\eea
and
\bea
\dot w =& \frac{\dot \varphi}{\varphi}+\la = \frac{\frac{\partial u}{\partial t}}{u} + \la = \frac{1}{Fu}+\la \\
& \le \frac{1}{nu(c-f^{'})} + \la \quad a.e.,
\eea
cf. (\ref{phi_dot_gleich_u_punkt}).
Now we observe that the argument of $f^{'}$ is $u$ and 
\beq
\lim_{t\rightarrow \infty}\inf_{x\in S_0}u(t,x) = 0
\eeq
because of Remark \ref{flow_runs_into_future_singularity}. On the other hand
\beq
\lim_{t \rightarrow \infty}f^{'}u = \tilde \gamma^{-1}=\frac{1}{n\ga},
\eeq
in view of (\ref{asymptotic_relation_for_f_strich}), and we infer
\beq
\frac{1}{nu(c-f^{'})} \rightarrow -\ga,
\eeq
hence $\dot w(t) \le 0$ for a.e. $t \ge t_{\la}$, $t_{\lambda}>0$ suitable.

Therefore, we deduce
\beq
w \le w(t_{\lambda}) \quad \forall t \ge t_{\lambda},
\eeq
i.e.
\beq
-u e^{\lambda t} \le c(\lambda) \quad \forall t\in \mathbb{R}_+.
\eeq
\end{proof}

We are now able to prove the exact exponential velocity.
\begin{theorem} \label{optimal_exponential_decay_for_u}
There are constants $c_1, c_2 > 0$ such that
\beq
-c_1 \le \tilde u = ue^{\ga t} \le - c_2 < 0.
\eeq
\end{theorem}
\begin{proof}
(i) We prove the estimate from above. Define
\beq
\varphi(t) = \sup_{x \in S_0}u(t,x)
\eeq
and 
\bea \label{7.3.30}
w = \log(-\varphi) + \ga t.
\eea
Reasoning similar as in the proof of the previous lemma, we obtain for a.e. $t\ge t_0$, $t_0$ sufficiently large,
\bea \label{7.3.31}
\dot w \ge&\frac{1}{nu(-c-f^{'})} + \ga \quad (\text{where } c>0) \\
=& u \frac{\frac{1-\tilde \ga u f^{'}}{u}-cn\ga}{nu(-c-f^{'})} \\
\ge& \tilde c u,
\eea
where $\tilde c$ is a positive upper bound for the fraction; note that this fraction converges due to the assumptions, cf. (\ref{asymptotic_relation_for_f_strich}).
  
The previous lemma now yields
\beq \label{12345678932}
\dot w \ge \tilde cu \ge -\tilde cc_{\lambda}e^{-\lambda t} \quad a.e. \text{\ }Êt \ge t_{\lambda}
\eeq
for any $0<\lambda < \ga$. Hence $w$ is bounded from below, or equivalently,
\beq
\tilde u \le -c_2 < 0.
\eeq
(ii) Now, we prove the estimate from below. Define 
\beq
\varphi(t) = \inf_{x\in S_0)}u(t,x)
\eeq
and $w$ as in (\ref{7.3.30}), then we obtain analogously that
%%%%%%%%%%%%%%%%%%%%%%% 
\auskommentieren{
inequality (\ref{12345678932}) holds the other way round where $\tilde c$ is now a negative constant and finally}
%%%%%%%%%%%%%%%%%%%%%%%
\beq
-c_1 \le \tilde u.
\eeq 
\end{proof}

\begin{lem}
For any $k \in \mathbb{N}^{*}$ there exists $c_k>0$ such that 
\beq
|f^{(k)}|\le c_k e^{k \ga t},
\eeq
where $f^{(k)}$ is evaluated at $u$.
\end{lem}
\begin{proof}
In view of the assumption (\ref{abschaetzung_der_ableitungen_von_f_nach_vor}) there holds
\beq
|f^{(k)}| \le c_k |f^{'}|^k = c_k|f^{'}|^ku^k \tilde u^{-k}e^{k\ga t}.
\eeq
Then use (\ref{asymptotic_relation_for_f_strich}) and the preceding theorem.
\end{proof}

\section{$C^1$-estimates\---Asymptotic behaviour of the flow}\label{Section_7}

In Section \ref{Section_4} we proved that $\tilde v$ is uniformly bounded for all times, cf. Lemma \ref{tilde_v_uniformly_bounded}.
We recall that
\beq
\tilde u = u e^{\ga t}.
\eeq

Our final goal is to show that $\|D\tilde u\|^2$ is uniformly bounded, but this estimate has to be deferred to Section \ref{Section_8}. At the moment we only prove an exponential decay for any $0<\lambda<\gamma$, i.e., we shall estimate $\|Du\| e^{\lambda t}$.

We remember that we have
\begin{equation} 
F=F(\check {h}^j_i)=F(e^{\tilde \psi}\breve {h}^j_i) = F(h^j_i - \tilde v  f^{'}{\de}^j_i+{\psi}_{\al}{\nu}^{\al}{\de}^j_i).
\end{equation} 

We need in the following a slightly different estimate from the one in (\ref{most_complicated_FijRandsoon}).
\begin{lem} \label{Riemann_neu_abgeschaetzt}
\bea
\Fij \Riemann \nu^{\al}x^{\be}_i x^{\ga}_l x^{\de}_j u^l=& -\tilde v\Fij  \Riemann \eta^{\al}x^{\be}_i x^{\ga}_l x^{\de}_j u^l \\
&-\tilde v\Fij\bar R_{r\be\ga\de}\check u^rx^{\be}_i x^{\ga}_l x^{\de}_j u^l.
\eea
\end{lem}

With the help of the boundedness of $\tilde v$, cf. Lemma \ref{tilde_v_uniformly_bounded}, we prove the following estimate.
\begin{lem} \label{v_tilde_decay_epsilon}
There exists $\epsilon >0$ and a constant $c_{\epsilon}$ such that
\beq
\|Du\| e^{\epsilon t} \le c_{\epsilon}.
\eeq
\end{lem}

\begin{proof}
We have
\beq 
\tilde v^2 = 1 + \|Du\|^2.
\eeq
Taking the log yields since $\tilde v$ is bounded
\beq
\|Du\|^2(1-c_1\|Du\|^2)\le 2 \log \tilde v = \log(1+\|Du\|^2)\le \|Du\|^2(1+c_1\|Du\|^2), 
\eeq
where $c_1$ is a positive constant,
i.e., it is sufficient to prove that $\log \tilde v e^{2 \epsilon t}$ is uniformly bounded.

Let $\epsilon>0$ be small and set 
\beq
\varphi = \log \tilde v e^{2 \epsilon t}, 
\eeq
then $\varphi$ satisfies
\bea
\dot \varphi -  \Fz \Fij \varphi_{ij} =    \frac{1}{\tilde v}(\dot{\tilde v} - \Fz\Fij\tilde v_{ij}) e^{2 \epsilon t} + \Fz \frac{1}{\tilde v^2}\Fij \tilde v_i  \tilde v_j e^{2 \epsilon t} + 2 \epsilon \varphi
\eea
hence (cf. Lemma \ref{Evolution_equation_of_tilde_v} )
\bea 
F^2e^{-2\epsilon t}(\dot \varphi -  \Fz \Fij \varphi_{ij})  = 
& - \Fij h_{kj}h^k_i 
+ \frac{1}{\tilde v} \Fij \bar R_{\al \be \ga \de}{\nu}^{\al}x^{\be}_ix^{\ga}_lx^{\de}_ju^l\\
&- \frac{1}{\tilde v}  F^{ij}h_{ij}\eta_{\al \be}\nu^{\al}\nu^{\be} 
-\frac{1}{\tilde v} F\eta_{\al \be}\nu^{\al}\nu^{\be} \\
&-  \frac{1}{\tilde v}  F^{ij} \eta_{\al \be \ga}\nu^{\al}x^{\be}_ix^{\ga}_j 
- \frac{2}{\tilde v} F^{ij}\eta_{\al\be}x^{\al}_kx^{\be}_ih^k_j \\
& +  f^{''}\|Du\|^2\Fij g_{ij}
+ \frac{1}{\tilde v} \tilde v_k u^k f^{'}\Fij g_{ij}  \\
&-  \psi_{\al \be}\nu^{\al}x^{\be}_ku^k \frac{1}{\tilde v}\Fij g_{ij} 
- \frac{1}{\tilde v} \psi_{\al}x^{\al}_l h^l_k u^k \Fij g_{ij} \\
& +  \Fij \tilde v_i \tilde v_j \frac{1}{\tilde v^2} +2 \epsilon F^2 \log \tilde v.
\eea
For $T$, $0<T<\infty$, assume that
\beq
\sup_{[0,T]} \sup_{M(t)} \varphi= \varphi(t_0, x_0),
\eeq
where $0<t_0\le T$ large, $x_0\in S_0$.

Applying the maximum principle we deduce in $(t_0, x_0)$ using Lemma \ref{normabschaetzungen_mit_v_schlange}, Lemma \ref{arbitrary_decay_of_special_function} and Lemma \ref{Riemann_neu_abgeschaetzt} that (note that  $\tilde u = ue^{\ga t}$ is bounded) for $t_0$ large and $\epsilon >0$ small.
\bea \label{c1_widerspruch_bei_konvergenz}
0 \le & - \frac{1}{2} \Fij h_{kj}h^k_i  +  c u^2 \Fg  + c |u|\|Du\|\Fg \\
&+ c\|Du\|^2\Fg +  f^{''}\|Du\|^2\Fij g_{ij} + c\epsilon |f^{'}|^2 \log \tilde v \Fg,
\eea
here we used that we have
\beq
F^2 \le c(\Fij h_{ik}h^k_j + |f^{'}|^2 \Fg)
\eeq
due to $\epsilon_0 F^2 \le \Fij \check h_{kj}\check h^k_i $, cf. Definition \ref{Kstern}.

The $\log \tilde v$ in (\ref{c1_widerspruch_bei_konvergenz}) can be estimated by $c\|Du\|^2$
yielding
\bea 
0 \le & - \frac{1}{2} \Fij h_{kj}h^k_i  +  c u^2 \Fg  + \frac{1}{2} f^{''}\|Du\|^2\Fij g_{ij}, 
\eea
where we have chosen  $\epsilon >0$ small and assumed that $t_0>0$ large.

Hence in $(t_0, x_0)$ 
\beq
\varphi = \log \tilde v e^{2\epsilon t} \le c \|Du\|^2e^{2\epsilon t}\le \frac{cu^2}{|f^{''}|}e^{2\epsilon t}\le c.
\eeq
\end{proof}

\begin{lem} \label{lemma_evolution_u}(Evolution of $u$)
\begin{equation} \label{evolution_u}
\begin{aligned} 
\dot u   &-F^{-2} F^{ij}u_{ij} = 2F^{-1} \tilde v + \Fz \tilde v^2 f^{'} \Fg - \Fz \tilde v {\psi}_{\al}{\nu}^{\al} \Fg -\Fz  \Fij \bar h_{ij} 
\end{aligned} \nn
\end{equation}
\end{lem}
\begin{proof}
The claim follows from the three identities  
\bea \label{u_punkt_gleich_tilde_v_durch_F}
\dot u &= \frac{\tilde v}{F} \\
u_{ij}&=-\tilde v h_{ij}+\bar h_{ij} \\ 
-\Fij h_{ij}&=-F-\tilde vf^{'}\Fg+\psi_{\al}\nu^{\al}\Fg.
\eea

\end{proof}

\begin{lem} \label{suboptimal_decay_of_Du}
For any $0<\lambda < \gamma$, there exists $c_{\lambda}$ such that
\beq
\| Du\|e^{\lambda t} \le c_{\lambda}.
\eeq
\end{lem}
\begin{proof}
Define
\beq
\varphi = \log \tilde v - \frac{\mu}{2}|u|^{2-\epsilon},  
\eeq
with $0<\epsilon <1$ arbitrary and $\mu >> 1$ chosen appropriately later. The interesting case is, when $\epsilon$ is close to 0.

Then $\varphi$ satisfies the following evolution equation, cf. Lemma \ref{Evolution_equation_of_tilde_v} and Lemma \ref{lemma_evolution_u}, 
\bea \label{eigen_logtildev}
\dot \varphi -  \Fz \Fij \varphi_{ij}  = &   \frac{1}{\tilde v}(\dot {\tilde v} - \Fz\Fij\tilde v_{ij}) +\frac{2-\epsilon}{2}\mu |u|^{1-\epsilon} (\dot u - \Fz \Fij u_{ij}) \\
& + \Fz \frac{1}{\tilde v^2}\Fij \tilde v_i  \tilde v_j 
+ \frac{2-\epsilon}{2}\mu (1-\epsilon)|u|^{-\epsilon}\Fz\Fij u_iu_j\\
=& - \Fz \Fij h_{kj}h^k_i 
+ \frac{1}{\tilde v}\Fz \Fij \bar R_{\al \be \ga \de}{\nu}^{\al}x^{\be}_ix^{\ga}_lx^{\de}_ju^l\\
&- \frac{1}{\tilde v} \Fz F^{ij}h_{ij}\eta_{\al \be}\nu^{\al}\nu^{\be} 
-\frac{1}{\tilde v} F^{-1}\eta_{\al \be}\nu^{\al}\nu^{\be} \\
&-  \frac{1}{\tilde v} \Fz F^{ij} \eta_{\al \be \ga}\nu^{\al}x^{\be}_ix^{\ga}_j 
+ \frac{2}{\tilde v}\Fz F^{ij}\eta_{\al\be}x^{\al}_kx^{\be}_ih^k_j \\
& + \Fz f^{''}\|Du\|^2\Fij g_{ij}
+ \frac{1}{\tilde v}\Fz \tilde v_k u^k f^{'}\Fij g_{ij} \\
& - \Fz \psi_{\al \be}\nu^{\al}x^{\be}_ku^k \frac{1}{\tilde v}\Fij g_{ij} 
- \frac{1}{\tilde v}\Fz \psi_{\al}x^{\al}_l h^l_k u^k \Fij g_{ij}�\\
& + (2-\epsilon)\mu |u|^{(1-\epsilon)}\frac{\tilde v}{F}
 +(1-\frac{\epsilon}{2}) \mu |u|^{1-\epsilon}\Fz  \tilde v^2 f^{'}\Fij g_{ij} \\
 &-(1-\frac{\epsilon}{2}) \mu |u|^{1-\epsilon} \Fz \tilde v \psi_{\al}\nu^{\al}\Fg \\
 &-(1-\frac{\epsilon}{2}) \mu |u|^{1-\epsilon} \Fz \Fij \bar h_{ij}\\
& + \Fz \Fij \tilde v_i \tilde v_j \frac{1}{\tilde v^2} 
+  (1-\frac{\epsilon}{2}) (1-\epsilon) \mu |u|^{-\epsilon}\Fz \Fij u_i u_j \\
= & \quad RHS.
\eea  
\end{proof}

We will show 
\beq
\varphi < 0 \quad \forall t \ge 0.
\eeq

Assume that this is not the case. Let $t_0>0$ be minimal such that
\beq
\sup_{S_0} \varphi(t_0, \cdot) =0
\eeq
and $x_0\in S_0$ such that
\beq
\varphi(t_0,x_0) = 0, 
\eeq
which implies that in $(t_0, x_0)$ the RHS in (\ref{eigen_logtildev}) is $\ge 0$, 
\beq
\frac{1}{2}\|Du\|^2 \ge \log \tilde v = \frac{\mu}{2}|u|^{2-\epsilon}
\eeq
for $\mu>0$ large (which implies $t_0$ large) and
\beq
\tilde v_i = -(1-\frac{\epsilon}{2})\mu \tilde v |u|^{1-\epsilon}u_i.
\eeq
%as well as
%\beq
%u_k h^k_i = \eta_{\al\be}\nu^{\al}x^{\be}_i-\tilde v_i
%\eeq

We now show that RHS in (\ref{eigen_logtildev}) is negative, if $t_0$ is sufficiently large, which can be guaranteed by increasing $\mu$ accordingly. 

We use
\beq \label{eigene_F_estimate}
F \le |u|^{1-\be}\de\Fij h_{ik}h^k_j + \frac{c(\de)}{|u|^{1-\be}}\Fg+ \tilde v |f^{'}|\Fg
\eeq
where $\be>0$ is chosen according to Lemma \ref{v_tilde_decay_epsilon} such that
\beq
\log \tilde v \le c |u|^{\be}
\eeq
and $\de >0$ is small, $c(\de)$ also depends on the upper bound of $|\psi_{\al}\nu^{\al}|$.

We find
\bea
0  \le  & F^2 (RHS) \\
 \le &- \frac{1}{2}\Fij h_{kj}h^k_i 
 + c \| Du\|^2 \Fij g_{ij} 
 + c \mu|u|\Fij g_{ij} 
 + f^{''}\mu |u|^{2-\epsilon} \Fg \\
& + (1-\frac{\epsilon}{2}) \mu |u|^{-\epsilon} \| Du\|^2 (\frac{1}{\tilde \gamma}+ c u^2) \Fg  \\
& + (2-\epsilon) \mu |u|^{1-\epsilon}\tilde v (|u|^{1-\be}\de \Fij h_{ik}h^k_j + \frac{c(\de)}{|u|^{1-\be}} \Fg + \tilde v |f^{'}| \Fg) \\
& + (1-\frac{\epsilon}{2}) \mu |u|^{1-\epsilon} \tilde v^2 f^{'}\Fg + 2(1-\frac{\epsilon}{2})^2 \mu \log \tilde v |u|^{-\epsilon}|Du|^2\Fg\\
&  + (1-\frac{\epsilon}{2})\mu (1-\epsilon)|u|^{-\epsilon}|Du|^2\Fg\\
< & \frac{\mu}{\tilde \ga}|u|^{-\epsilon}(-1+ c \|Du\|+(1-\frac{\epsilon}{2})\tilde v^2+ c \tilde \ga |u|^{\epsilon}+c|u|^{\be})\Fg \\
<& 0, 
\eea
where we have chosen $\mu$ large ($\Rightarrow$ $t_0$ large).
%if $t_0$ is sufficiently large, which can be reached by chosing $\mu$ very large. 
Here we used 
\bea
|f^{'}u - \frac{1}{\tilde \ga}| \le c u^2, \quad \big||f^{''}|u^2 - \frac{1}{\tilde \ga}\big| \le c u^2
\eea
and
\beq
\mu = \frac{2\log \tilde v}{|u|^{2-\epsilon}} \le 2 |u|^{\be+\epsilon-2}.
\eeq

\section{$C^2$-estimates\---Asymptotic behaviour of the flow} \label{Section_8}

$F$ grows exponentially fast in time, more precisely we have the following
\begin{theorem} \label{F_grows_expontially_fast_in_part_2}
The estimate
\beq
F \ge c e^{\ga t}
\eeq
is valid, where $c>0$ depends on $M_0$.
\end{theorem}

\begin{proof}
Use Lemma \ref{F_to_infty} (note that we used a different notation there) and (\ref{zusammenhang_zwischen_F_und_breve_F}). 
\end{proof}

For later purposes we obtain an evolution equation for $F$. 

As usual we have (we remark that in our case the evolution equations are the same as in \cite[Lemma 2.3.2, Lemma 2.3.3]{CP}, see also (\ref{u_punkt_gleich_tilde_v_durch_F}))
\bea \label{hij_in_konformer_metrik_einfach_version}
\dot h^j_i = &(-\frac{1}{F})^j_i+\frac{1}{F}h^k_ih^j_k+ \frac{1}{F}\Riemann \nu^{\al}x^{\be}_i\nu^{\ga}x^{\de}_kg^{kj}\\
\dot \nu^{\al}=&g^{ij}\frac{F_i}{F^2}x^{\al}_j \\
\dot{\tilde v}=&-\frac{1}{F}\eta_{\al\be}\nu^{\al}\nu^{\be}-g^{ij}\frac{F_i}{F^2}u_j \\
 \dot u =&\frac{\tilde v}{F} \\
 \dot g_{ij}=&-\frac{2}{F}h_{ij}
\eea
and, furthermore, since 
\beq
F = F(\check {h}^j_i)= F(h^j_i-\tilde v f^{'} \delta^j_i + \psi_{\al}\nu^{\al}\de^j_i)
\eeq
we infer
\beq
\dot F= F^i_j \dot {\check h}^j_i, 
\eeq
and finally
\begin{lem} \label{Lemma_Evolution_von_F_Teil_2}
\bea
\dot F -\frac{1}{F^2}\Fij F_{ij} =& - \frac{2}{F^3}\Fij F_i F_j + \frac{1}{F}\Fij h^k_ih_{kj}+\frac{1}{F}\Fij 
\Riemann \nu^{\al}x^{\be}_i\nu^{\ga}x^{\de}_j \\
+& \frac{1}{F}\eta_{\al\be}\nu^{\al}\nu^{\be}f^{'}\Fg - \frac{1}{F}\tilde v^2 f^{''}\Fg + \frac{1}{F^2} f^{'}F_ku^k\Fg \\
-&\frac{1}{F}\psi_{\al\be}\nu^{\al}\nu^{\be}\Fg + \frac{1}{F^2} \psi_{\al}x^{\al}_kF^k\Fg.
\eea
\end{lem}

In the following lemma we prove the important evolution equation for the second fundamental form $(h^k_l)$.
\begin{lem} \label{evol_von_h_k_l_in_konformer_metrik}
\begin{equation} \label{evol_hkl}
\begin{aligned} 
&\dot {h}^k_l   -F^{-2} F^{ij} h^k_{l;ij}= -2 F^{-3}F^kF_l
+ F^{-1}h^{kr}h_{rl}
+ F^{-1}\Riemann {\nu}^{\al}x^{\be}_l{\nu}^{\ga}x^{\de}_rg^{rk}\\
& -\Fz \Fij h_{aj}h^a_ih^k_l
+\Fz \Fij h_{ij}h_{al}h^{ak}
+2\Fz g^{pk}\Fij \Riemann x^{\al}_rx^{\be}_px^{\ga}_ix^{\de}_lh^r_j\\
&-\Fz \Fij \Riemann x^{\al}_ax^{\be}_ix^{\ga}_lx^{\de}_jh^{ak}
- \Fz g^{pk}\Fij\Riemann x^{\al}_rx^{\be}_ix^{\ga}_px^{\de}_jh^r_l \\
&-\Fz \Fij \Riemann {\nu}^{\al}x^{\be}_i{\nu}^{\ga}x^{\de}_jh^k_l
+\Fz g^{pk}\Fij \Riemann {\nu}^{\al}x^{\be}_p{\nu}^{\ga}x^{\de}_lh_{ij}\\
&+\Fz g^{pk}\Fij \bar R_{\al\be\ga\de;\epsilon}{\nu}^{\al}x^{\be}_px^{\ga}_ix^{\de}_lx^{\epsilon}_j
+\Fz g^{pk}\Fij \bar R_{\al\be\ga\de;\epsilon}{\nu}^{\al}x^{\be}_ix^{\ga}_jx^{\de}_px^{\epsilon}_l\\
&+\Fz g^{pk}F^{ij,rs}\check h_{ij;p}\check h_{rs;l} \\
&+\Fz \Fg (-u_l u^k \tilde v f^{'''}
+ g^{pk}{\psi}_{\al\be\ga}{\nu}^{\al}x^{\be}_px^{\ga}_l
+ {\psi}_{\al\be}{\nu}^{\al}{\nu}^{\be}h^k_l\\
&\quad + g^{pk}{\psi}_{\al\be}x^{\al}_rx^{\be}_ph^r_l
+ \pab x^{\al}_rx^{\be}_lh^{rk}
+ {\psi}_{\al}{\nu}^{\al}h_{lr}h^{rk}
+ {\psi}_{\al}x^{\al}_rh^{rk}_{\ \ ;l})\\
& + \Fz \Fg (-g^{pk}f^{'}{\eta}_{\al\be\ga}{\nu}^{\al}x^{\be}_px^{\ga}_l
-g^{pk}f^{'}{\eta}_{\al\be}x^{\al}_rx^{\be}_ph^r_l
-f^{'}{\eta}_{\al\be}{\nu}^{\al}{\nu}^{\be}h^k_l \\
& \quad - f^{'}{\eta}_{\al\be} x^{\al}_rx^{\be}_lh^{rk}
- f^{'}h_{rl}h^{rk}\tilde v + f^{'} u^rh^k_{l;r} 
+ f^{'}u^r g^{kp} \Riemann {\nu}^{\al}x^{\be}_px^{\ga}_rx^{\de}_l \\
& \quad -f^{''}(\tilde v^ku_l+\tilde v_lu^k)
+ f^{''} \tilde v^2 h^k_l  
+ f^{''} \tilde v{\eta}_{\al\be}x^{\al}_lx^{\be}_rg^{rk}).
\end{aligned} 
\end{equation}
\end{lem}
\begin{proof}
The starting point of the proof is the equation for $\dot h^j_i$ given in (\ref{hij_in_konformer_metrik_einfach_version}), which contains the summand 
\beq
(-\frac{1}{F})^j_i = \frac{1}{F^2}F^j_i-\frac{2}{F^3}F_iF^j.
\eeq
To finish the proof, we only have to calculate the covariant derivative $F^j_i$ in detail. Deriving the purely covariant version of this tensor we first get
\beq
F_{kl}= \Fij \check h_{ij;kl}+F^{ij,rs}\check h_{ij;k}\check h_{rs;l},
\eeq
then $\check h_{ij;kl}$ will be expressed as 
\beq
\check h_{ij;kl}= h_{ij;kl}+ additional\ terms
\eeq
and interchanging indices in the usual way (which is technical using the Codazzi equations and the Ricci identities, cf. the proof of \cite[Lemma 2.4.1]{CP}) leads to the representation
\beq
\check h_{ij;kl}=h_{kl;ij}+ additional\ terms,
\eeq
with different \textit{additional terms}.
\end{proof}

We already know the estimate
\beq 
-c |f^{'}| \le \kappa_i,
\eeq
$c>0$, because of the fact that the $\check \kappa_i$ are positive, remember $\check \kappa_i = \kappa_i-\tilde v f^{'}+\psi_{\al}\nu^{\al}$.
Now we prove an estimate from above.

\begin{theorem}
We have
\beq
\kappa_i \le c . 
\eeq
\end{theorem}

\begin{proof}
Let $\varphi$ be defined by
\beq
\varphi = \sup \{h_{ij}\eta^i\eta^j: \|\eta\|=1\}.
\eeq
We shall prove that
\beq
w = \log \varphi + \lambda \tilde v
\eeq
is uniformly bounded from above, if $\lambda$ is large enough.

The proof is devided into two steps:

(i) There is a $\mu>0$ such that if a maximum of $w|_{[0,T]}$ (where $0<T<\infty$ arbitrary but fixed) is attained in $(t_0, x_0)$, $0<t_0\le T $, $x_0 \in S_0$, then there holds in $(t_0, x_0)$
\beq \label{h_nn_le_c_fstrich}
 h^n_n \le \mu |f^{'}|
\eeq
($h^n_n$ denotes as usual the largest principal curvature).

(ii) Secondly we prove that
\beq
h^n_n \le c
\eeq
in $(t_0,x_0)$, where, without loss of generality,  we  may assume that $t_0$ is large.

Now we prove (i) by contradiction.  Introducing Riemannian normal coordinates  around $(t_0,x_0)$ and arguing as usual, i.e. second derivatives of $\varphi$ with respect to space and the first derivative with respect to time coincide with the corresponding ones of $h^n_n$, furthermore $g_{ij}=\de_{ij}$ and $h^j_i$ is diagonal, we may assume that $w$ is defined by
\beq
w = \log h^n_n + \la \tilde v.
\eeq
Moreover, we assume $h^n_n>\mu|f^{'}|$ in $(t_0, x_0)$, where $\mu$ is large and will be chosen later. Applying the maximum principle we obtain 
\beq \label{heiko}
0 \le \dot w - \frac{1}{F^2}\Fij w_{ij}.
\eeq
in $(t_0, x_0)$.

Using $F=F^{ij}\check h_{ij}$ and $F\in (K^{*})$ we have, cf. Definition \ref{Kstern},
\bea \label{kritische_stelle_nach_korrektur_ausfuehrlicher}
\epsilon_0 F \check H & \le \Fij\check h^k_i \check h_{kj} 
 = F^{ii}(\check h_{ii})^2
\le F^{ii}(h_{ii}+\tilde v |f^{'}|g_{ii}+\psi_{\al}\nu^{\al}g_{ii})^2\\
& \le (1+\epsilon) F^{ii}h_{ii}^2 + 2 \tilde v |f^{'}|F^{ij}h_{ij}+\tilde v^2 |f^{'}|^2\Fg + c_{\epsilon} |u|\Fg \\
& \le (1+\epsilon) F^{ii}h_{ii}^2 + 2 \tilde v |f^{'}|F, 
\eea
where $\epsilon >0$. In view of 
\beq
\check h^n_n = h^n_n - \tilde v f^{'}+\psi_{\al} \nu^{\al}.
\eeq
we infer
\bea
-(1+\epsilon) F^{ii}h_{ii}^2 &\le - \epsilon_0 F \check H + 2 \tilde v |f^{'}|F \\
&\le - \frac{\epsilon_0}{2}F h^n_n - \frac{\epsilon_0}{2}F h^n_n + 2 \tilde v |f^{'}|F \\
& \le - \frac{\epsilon_0}{2}F h^n_n -  \frac{\epsilon_0}{2}\mu F |f^{'}| +2 \tilde v |f^{'}|F \\
& \le - \frac{\epsilon_0}{2}F h^n_n, 
\eea
where we assume that $\mu$ is large; hence there is $\de_0 >0$ such that
\beq
-\Fij h_{ik}h^k_j \le -\de_0 F h^n_n
\eeq
in $(t_0, x_0)$.

In $(t_0, x_0)$ we have
\beq
h^n_{n;i} = -\la \tilde v_i h^n_n
\eeq 
and in view of (\ref{heiko})
\bea \label{Svenja}
0 & \le \frac{1}{h^n_n}(\dot h^n_n - \Fz F^{ij}h^n_{n;ij}) + \lambda (\dot {\tilde v} - \Fz F^{ij}\tilde v_{ij}) + \frac{\lambda^2}{F^2} F^{ij}\tilde v_i \tilde v_j.
\eea
Multiplying this inequality by $F^2$, inserting the evolution equations for $h^n_n$ and $\tilde v$, cf. Lemma \ref{evol_von_h_k_l_in_konformer_metrik} and Lemma \ref{Evolution_equation_of_tilde_v}, as well as some trivial estimates yield (no summation with respect to $n$)
\bea
0 \le  & -2 \frac{1}{h^n_n}F^{-1}F^nF_n + 2F h^n_n  + \frac{c}{h^n_n}F+\frac{1}{h^n_n}F^{ij,rs}\check h_{ij;n}\check h_{rs;n} \\
& + c|f^{'}|^{\frac{3}{2}}\Fg + \tilde v^2f^{''}\Fg +  \lambda |u|\Fg \\
&- \frac{\lambda}{2} \tilde v F^{ii}h_{ii}^2 +\lambda^2 F^{ij}\tilde v_i \tilde v_j.
\eea

We remark that we have estimated the term arising from the second term in the second line of equation (\ref{evol_hkl}) together with two other terms arising from (\ref{evol_hkl}) by employing the homogeneity of $F$, namely,  $F=\Fij h_{ij}- \tilde v f^{'}\Fg+\psi_{\al}\nu^{\al}\Fg$.

Terms arising from the two terms in (\ref{evol_hkl}) depending linearly on the derivatives of the second fundamental form are first rewritten with the help of the Codazzi equation (the correction terms can be estimated very easily) such that we obtain the derivative of $h^n_n$. The resulting terms can be estimated as follows:  
%(assume $\mu>2\tilde v$ so that $|\kappa_i|\le h^n_n$, use $F\ge \Fij h_{ij}+\frac{1}{2}\tilde v |f^{'}|\Fg$)
\bea
\frac{1}{h^n_n}\psi_{\al}x^{\al}_rh^{n;r}_{n} \Fg & \le \la c|u|\Fg + \la \psi_{\al}x^{\al}_r u_sh^{sr}\Fg
\eea
for the first term, where we used
\beq
\tilde v_i = \eta_{\al \be}\nu^{\al}x^{\be}_i - u_r h^r_i,
\eeq 
with $(\eta_{\al})$ as in Lemma \ref{Evolution_equation_of_tilde_v}, cf. also Lemma \ref{arbitrary_decay_of_special_function}, and
\beq
\frac{1}{h^n_n}f^{'} u^r h^{n}_{n;r}\Fg = - \la f^{'}u^r\tilde v_r \Fg
\eeq
for the second one. Both last summands in the previous inequalities appear among the terms coming from the evolution equation of $\tilde v$ with opposite sign. 

Since $F \in (K)$ and homogenous of degree 1 we deduce from \cite[Lemma 2.2.14]{CP} that $F$ is concave, hence \cite[Proposition 2.1.23]{CP} implies 
\beq
F^{ij,rs}(\check h_{ij})\check h_{ij;n}\check h_{rs;n}\le 0.
\eeq
Together with 
\beq
F^{ij}\tilde v_i \tilde v_j \le c|u|\Fg + c\|Du\|^2 \Fij h_{ik}h^k_j
\eeq 
(which follows by using $\| x-y\|^2\le 2 \|x\|^2+2 \|y\|^2$, here $\|\cdot \|$ is the norm induced by the quadratic form $\Fij$)
we conclude
\bea
0 \le  &  2F h^n_n + \frac{c}{h^n_n}F + c|f^{'}|^{\frac{3}{2}}\Fg + \tilde v^2f^{''}\Fg + c\lambda^2 |u|\Fg \\
&- \frac{\lambda}{4} \tilde v\de_0 Fh^n_n.
\eea
For $\la>0$ large we get a contradiction, which finishes the proof of (i).

\medskip

We now prove (ii). From (i) we deduce that the largest principal curvature of $M(t)$ is bounded by $ce^{\ga t}$ for all $t>0$. Combining this with Lemma \ref{F_grows_expontially_fast_in_part_2}, namely,
\beq
0<c_0 \le F(e^{-\ga t}\check h^i_j)=F(e^{-\ga t}(h^i_j-\tilde v f^{'}\de^i_j+\psi_{\al}\nu^{\al}\de^i_j)),
\eeq
we infer that $e^{-\ga t}\check h^i_j$ lies in a compact subset of $\Gamma_+$ for all $t>0$.
Hence we have constants $c, \tilde c_1, \tilde c_2, \bar c_1, \bar c_2>0$ (not depending on $t_0$ or $T$), such that for all times and especially in $(t_0, x_0)$
\beq
-c e^{\ga t} \le \kappa_i \le c e^{\ga t} \quad \wedge \quad \tilde c_1 e^{\ga t} \le F \le \tilde c_2 e^{\ga t}
\quad \wedge \quad  0<\bar c_1 g_{ij}\le F^{ij}\le \bar c_2g_{ij}.
\eeq 
We again look at (\ref{Svenja}) multiplied by $F^2$ in $(t_0, x_0)$. 
We assume that $h^n_n$ is large and will show that it is a priori bounded.
We have
\bea
-F^{ii}h_{ii}^2 \le -\bar c_1 (h^n_n)^2 
\eea
and furthermore
\bea
0 &\le 2 F h^n_n +\frac{c}{h^n_n}F+ \la |f^{'}|^{\frac{3}{2}}\Fg + f^{''}\tilde v^2 \Fg-\frac{\la}{2}\Fij h_{ik}h^k_j \tilde v \\
&\le \epsilon F^2 + c_{\epsilon}(h^n_n)^2 + \la |f^{'}|^{\frac{3}{2}} \Fg + f^{''}\tilde v^2 \Fg
-\frac{\la}{2}\bar c_1 \tilde v (h^n_n)^2,
\eea
$\epsilon >0$ small; we remember
\beq
f^{''}\le -c e^{2 \ga t}.
\eeq

If $\lambda$ is sufficiently large and $t_0$ sufficiently large we get a contradiction.
\end{proof}

%\begin{rem}
%We want to point out that the proof of the previous Lemma shows that $0<c_1<e^{-\ga t}\check \kappa_i<c_2$ for some constants $c_2, c_2$, hence the first derivative $F^i$ ($=$ partial derivative of $F$ considered as a function defined in the positive cone in $\mathbb{R}^n$) of $F$ are also uniformly bounded by positive constants. 
%\end{rem}

\begin{lem} \label{eigenes_schoenstes_Lemma}
\begin{equation} \label{eigene_Behauptung}
\sup_{M(t)} \max_i |\kappa_i u| \rightarrow 0 \quad t\rightarrow \infty.
\end{equation}
\end{lem}
\begin{proof}
We remember that 
\begin{equation} \label{eigen_remember}
F = F(\check {h}^j_i)= F(h^j_i-\tilde v f^{'} \delta^j_i + \psi_{\al}\nu^{\al}\de^j_i)= F(\kappa_i-\tilde v f^{'} + \psi_{\al}\nu^{\al}),
\end{equation}
where the $\kappa_i$ are the eigenvalues of $h^j_i$, now numbered such that $\kappa_n$ is the smallest one.

The function 
$\varphi = -uF$
satisfies the following parabolic equation, cf. Lemma \ref{Lemma_Evolution_von_F_Teil_2} and Lemma \ref{lemma_evolution_u}, 
\begin{equation} \label{eigene_Loesung}
\begin{aligned}
\dot \varphi-&F^{-2}\Fij \varphi_{ij} = 
\frac{2u}{F^3}\Fij F_iF_j 
-\frac{u}{F}\Fij h^k_i h_{kj}
-\frac{u}{F} \Fij \Riemann {\nu}^{\al}x^{\be}_i {\nu}^{\ga}x^{\de}_j\\
&\ -\frac{u}{F} f^{'}  {\eta}_{\al \be} {\nu}^{\al}{\nu}^{\be}\Fij g_{ij} 
-\frac{u}{F^2} f^{'} F^k u_k \Fij g_{ij} 
+\frac{u}{F}\tilde v^2 f^{''}\Fij g_{ij}\\
&\ +\frac{u}{F}\psi_{\al \be}\nu^{\al}\nu^{\be}\Fij g_{ij}
-\frac{u}{F^2} F^k \psi_{\al}x^{\al}_k \Fij g_{ij} 
-2 \tilde v \\
&\ -\frac{\tilde v^2}{F} f^{'}\Fij g_{ij}
+ \frac{\tilde v}{F}\psi_{\al}\nu^{\al} \Fij g_{ij}
+ \frac{1}{F}\Fij \bar h_{ij} 
+ \frac{2}{F^2}\Fij u_i F_j
.
\end{aligned}
\end{equation}

For $t>0$ we define $\tilde \varphi (t) = \inf_{S_0}\varphi(t, \cdot)$  and choose $x_t \in S_0$ such that
\begin{equation}
\tilde \varphi (t) = \varphi(t, x_t),
\end{equation}
then $\tilde \varphi$ is differentiable a.e. and we have
\begin{equation} \label{eigene_Loesung2}
\dot {\tilde {\varphi}} (t) = \dot \varphi(t, x_t)
\end{equation}
for a.e. $t>0$.

Let $t_0> 0$ be sufficiently large, then combining (\ref{eigene_Loesung}) and (\ref{eigene_Loesung2}) and using $\varphi_i=0$ yields
\begin{equation} \label{eigene_Loesung3}
\begin{aligned}
\dot {\tilde {\varphi}} (t) \ge &\ -\frac{u}{F}\Fij h^k_i h_{kj} + u \frac{\tilde v^2}{F}f^{''}\Fij g_{ij} 
-\frac{\tilde v^2}{F}f^{'}\Fij g_{ij} \\
&\ -2 \tilde v -\frac{c_0}{F}\Fij g_{ij}
\end{aligned}
\end{equation}
for a.e. $t>t_0$, where $c_0=c_0(t_0)$ and the right side is evaluated at $(t, x_t)$. Due to the assumptions on $f$ we may furthermore assume that for all $t>t_0$ the following inequality holds in $(t, x_t)$
\begin{equation}
\frac{u}{F}\tilde v^2 f^{''}\Fij g_{ij} - \frac{\tilde v^2}{F}f^{'} \Fij g_{ij} \ge 2\tilde v - \frac{c_0}{F}\Fij g_{ij},
\end{equation}
which leads to 
\begin{equation} \label{eigene_Loesung4}
\begin{aligned}
\dot {\tilde {\varphi}} (t) \ge &\ -\frac{u}{F}\Fij h^k_i h_{kj} -2 \frac{c_0}{F}\Fij g_{ij}
\end{aligned}
\end{equation}
for a.e. $t>t_0$ in view of (\ref{eigene_Loesung3}); again the right side is evaluated at $(t, x_t)$.

We assume that (\ref{eigene_Behauptung}) is not true, then there are sequences $0<t_k\rightarrow \infty $, $x_k \in S_0$ and a constant $c_1 > 0$ such that 
\begin{equation}
\sup_{M(t_k)} \max_i \kappa_i u = \kappa_n u |_{(t_k, x_k)}  \rightarrow c_1, 
\end{equation}
which implies
\bea \label{eigener_Widerspruch}
\limsup_{k \rightarrow \infty}\tilde \varphi(t_k) &< F(-\frac{c_1}{2}+\tilde \gamma^{-1}, \tilde \gamma^{-1}, ..., \tilde \gamma^{-1}) \\
&< F(\tilde \gamma^{-1}-r, ..., \tilde \gamma^{-1} -r) \\
&=: c(r),
\eea
for $r>0$ sufficiently small and fixed from now on.

Next, we will show that, after increasing $t_0$ if necessary, there exists $\de > 0$ such that the following implication holds for a.e. $t>t_0$

\begin{equation} \label{eigener_Widerspruch2}
\tilde \varphi(t) \le c(r) \Rightarrow \dot {\tilde \varphi}(t) \ge \de
\end{equation}
in contradiction to (\ref{eigener_Widerspruch}).

For that purpose assume $t_0$ to be sufficiently large. Let $t>t_0$ be such that $\tilde \varphi$ is differentiable in $t$ and $\tilde \varphi(t) \le c(r)$, then it follows from (\ref{eigen_remember}) that we have in $(t, x_t)$
\begin{equation}
|u|\kappa_n + \tilde v |f^{'}u| + |u|\psi_{\al}\nu^{\al} \le -r+ \tilde \gamma^{-1},
\end{equation}
i.e.
\begin{equation}
\kappa_n \le - \frac{r}{2|u|}.
\end{equation}
Hence, we infer from (\ref{eigene_Loesung4}) 
\begin{equation}
\dot {\tilde \varphi}(t) \ge \frac{r^2}{4F|u|}-2 \frac{c_0}{F}\Fij g_{ij}.
\end{equation}
After a possibly further enlargement of $t_0$ we get a positive lower bound for the right side of the last inequality that does not depend on $t$, thus the desired $\de >0$, which completes the proof.
\end{proof}

Now we are able to prove a decay of $\|A\|$.

\begin{lem} \label{suboptimal_decay_of_A}
For any $0<\lambda < \ga$ there exists $c_{\lambda}> 0$ such that
\begin{equation}
\|A\| e^{\lambda t} \le c_{\la}.
\end{equation}

\end{lem}
\begin{proof}
Define $\varphi=\frac{1}{2} {\|A\|}^2 e^{2\la t}$ with $0<\lambda<\ga$, then
\begin{equation} \label{7.5.32_eigen}
e^{-2\la t}(\dot \varphi -\frac{1}{F^2} \Fij {\varphi}_{ij}) = -\frac{1}{F^2}F^{kl}h^i_{j;k}h^j_{i;l} +(\dot h^i_j-\frac{1}{F^2}F^{kl}h^i_{j;kl})h^j_i + \la \|A\|^2.
\end{equation}

 Let $0<T<\infty$ be large, and $x_0=x_0(t_0)$, with $0<t_0 \le T$, be a point in $M(t_0)$ such that
\begin{equation}
\sup_{M_0}\varphi < \sup\{\sup_{M(t)} \varphi: 0<t\le T\} = \varphi(x_0).
\end{equation} 
From Lemma \ref{eigenes_schoenstes_Lemma} we know that
\beq
\sup_{M(t)}\|A\||u| \longrightarrow 0 \quad as \quad t \longrightarrow \infty
\eeq
so that especially in view of the homogeneity of $F$
\beq \label{second_der_of_F_exponential}
0 < c_1 < F^{i}(\check \kappa_i) \le c_2 \quad \wedge \quad |F^{ij}(\check \kappa_i)| \le c e^{-\ga t}
\eeq
(first and second derivatives of $F$ considered as a function on $\Gamma_+$).
In $x_0$ we have due to (\ref{7.5.32_eigen}) and Lemma \ref{evol_von_h_k_l_in_konformer_metrik}, after multiplication by $F^2$ and some straight-forward estimates,
\bea \label{zentrale_Ungleichung_A_decay}
0 \le &-F^{kl}h^i_{j;k}h^j_{i;l} -2 F^{-1}F^iF_jh^j_i + 2F h^{ir}h_{rj}h^j_i  + F^{ij,rs}\check h_{ij;p}\check h_{rs;p}h^{pp} \\
&+ c|f^{'}|^{1 + \epsilon} \|A\| +c |f^{'}|^{\frac{1}{2}}\|A\|^2  + \tilde v^2f^{''}\|A\|^2 \Fg+\la F^2 \|A\|^2 \\
\le & - \frac{1}{2}F^{kl}h^i_{j;k}h^j_{i;l} + \tilde v^2 f^{''}\|A\|^2 \Fg + c F \|A\|^2\|A\| + c|f^{'}|^{1+\epsilon} \|A\| \\
&+c |f^{'}|^{\frac{1}{2}}\|A\|^2
+ \la F^2 \|A\|^2.
\eea
For the last inequality we used that in local coordinates (such that $g_{ij}$= $\de_{ij}$, $h_{ij}$ diagonal and $\Fij$ diagonal)
\beq
|F_iF_j| \le c\sum_{i,k,l}|h_{kl;i}|^2 +c \|A\|^2|f^{'}|^{\frac{1}{2}} + c |f^{'}|^{2+\epsilon},
\eeq
where we used Lemma \ref{suboptimal_decay_of_Du}, and where $0<\epsilon <1$ is arbitrary but fixed, so that 
\beq
F^{-1}F^iF_jh^j_i \le c\frac{\|A\|}{F}\sum_{i,k,l}|h_{kl;i}|^2 + c\|A\|^2|f^{'}|^{-\frac{1}{2}}\|A\|+c |f^{'}|^{1+\epsilon}\|A\|,
\eeq
where $\frac{\|A\|}{F} \rightarrow 0$ because of Lemma \ref{eigenes_schoenstes_Lemma}. \\
To estimate $F^{ij,rs}\check h_{ij;p}\check h_{rs;p}h^p_p$ we used \cite[inequality (2.1.73)]{CP} and (\ref{second_der_of_F_exponential}).

Now we have
\bea
F &= |f^{'}|F(\frac{\check \kappa_i}{|f^{'}|}) = |f^{'}|F(1, ..., 1) + |f^{'}|(F(\frac{\check \kappa_i}{|f^{'}|})-F(1, ..., 1)) \\
& \le n|f^{'}| +|f^{'}|c(t),
\eea
where $0<c(t)\rightarrow 0$, hence
\bea
\tilde v^2 f^{''}\|A\|^2 \Fg + \la F^2 \|A\|^2 \le& c\|A\|^2 - (\ga -\la)n^2 |f^{'}|^2 \|A\|^2 \\
&+\la cc(t)|f^{'}|^2\|A\|^2.
\eea

Together with (\ref{zentrale_Ungleichung_A_decay}) we deduce that $\varphi$ is a priori bounded from above.
\end{proof}

In the next two theorems we prove the optimal decay of $\|Du\|$ and $\|A\|$ which finishes the $C^2$-estimates.

\begin{theorem} \label{optimal_decay_of _Du}
Let $\tilde u = u e^{\ga t}$, then $\|D \tilde u\|$ is uniformly bounded during the evolution.
\end{theorem}
\begin{proof}
Let $\vp=\vp(t)$ be defined by
\begin{equation}
\vp = \sup_{M(t)} \log \tilde v e^{2 \ga t}.
\end{equation}
Then, in view of the maximum principle, we deduce from the evolution equation of $\tilde v$, cf. Lemma (\ref{Evolution_equation_of_tilde_v}),
\bea \label{7.5.40}
\dot \vp &\le ce^{-\epsilon t}+ F^{-2}(f^{''}\|D\tilde u\|^2\Fij g_{ij} + 2 \ga F^2 \vp) \\
& \le c e^{-\epsilon t} + 2F^{-2}(f^{''}\Fij g_{ij} +  \ga F^2) \vp \\
& \le  c e^{-\epsilon t}(1 + \vp),
\eea
where $\epsilon >0$ small, 
i.e., $\vp$ is uniformly bounded.
\end{proof}

\begin{theorem} \label{optimal_decay_A}
The quantity $w=\frac{1}{2}\| A\|^2e^{2\ga t}$ is uniformly bounded during the evolution.
\end{theorem}
\begin{proof}
Define $\vp=\vp(t)$ by
\begin{equation}
\vp = \sup_{M(t)}w.
\end{equation}
We deduce from Lemma \ref{evol_von_h_k_l_in_konformer_metrik} that for a.e. $t\ge t_0$, $t_0 >0$ large,
\bea \label{diese_ungleichung_brauchen_wir_gleich}
\dot \vp = &\frac{1}{F^2} \Fij \vp_{ij}  -\frac{1}{F^2}F^{kl}h^i_{j;k}h^j_{i;l}e^{2\ga t} +(\dot h^i_j-\frac{1}{F^2}F^{kl}h^i_{j;kl})h^j_ie^{2\ga t} \\
&+ \ga \|A\|^2e^{2\ga t} \\
 \le& -\frac{1}{2F^2}F^{kl}h^i_{j;k}h^j_{i;l}e^{2\ga t}  + F^{-3} (-2h^{ij} F_i F_j e^{2\ga t} - F f^{'''} h^{ij} \tilde u_i \tilde u_j \Fij g_{ij} \tilde v) \\
 &+2 F^{-2}(n f^{''} \tilde v^2 \vp + \ga F^2 \vp) + ce^{-\epsilon t}(1+\vp) \\  
  &+F^{-1} \Riemann {\nu}^{\al} x^{\be}_i {\nu}^{\ga} x^{\delta}_j h^{ij} e^{2 \ga t},
\eea
where $\epsilon > 0$ is small.

For the last inequality we estimated the crucial term
\beq \label{crucial_term_for_optimal_decay_of_A}
F^{ij, rs}\check h_{ij;p}\check h_{rs;q}h^{pq}
\eeq
in the following way. 

%For the last inequality we estimated the crucial term $\sum_p F^{ij,rs}(\check h_{kl})(f^{''}u_p)^2g_{ij}g_{rs}$ in the following way. 

%First we remark that the usual estimate
%\beq
%\|D^kF(\check h_{kl})\| \le |u|^{k-1}\|D^kF(|u|\check h_{kl})\| \le c|u|^{k-1}
%\eeq
%is not good enough. 

Since $F^{ij}g_{ij}\ge F(1, ..., 1)$ and 
 \begin{equation}
 F^{ij}(g_{kl})g_{ij} = F(1, ..., 1)
 \end{equation}
 we deduce that the derivative vanishes in $\check h_{kl}=g_{kl}$
 \begin{equation} \label{Fijrs0}
F^{ij,rs}(g_{kl})g_{ij} = 0.
 \end{equation}
Hence
\bea \label{Baum1}
F^{ij,rs} (\check h_{kl}) g_{ij} =  |u| (F^{ij,rs}(|u|\check h_{kl})-F^{ij,rs}(\frac{1}{\tilde \ga}g_{kl}))g_{ij}
\eea
which means by mean value theorem
\beq
\|F^{ij,rs} (\check h_{kl})g_{ij}\| \le c|u|^2. 
\eeq
Although the last inequality is good enough, we mention that its right side could be improved to $c|u|^{3-\epsilon}$, $\epsilon >0$ arbitrary, cf. Lemma \ref{suboptimal_decay_of_A}. 

Furthermore, to estimate (\ref{crucial_term_for_optimal_decay_of_A}) we use
\beq
\check h_{ij;p} = h_{ij;p}-\tilde v_pf^{'}g_{ij}-\tilde vf^{''}u_pg_{ij}+\psi_{\al\be}\nu^{\al}x^{\be}_pg_{ij}+\psi_{\al}x^{\al}_r h^r_pg_{ij}
\eeq
and
\beq
\tilde v_p = \eta_{\al\be}\nu^{\al}x^{\be}_p-u_rh^r_p
\eeq

%%%%%%%%%%%%%%%%%%Auskommentieren: Beginn%%%%%%%%%%%%%%%%%%
\auskommentieren{
\begin{equation}
\begin{aligned}
\sum_p F^{ij,rs}& (\check h_{ij}) (f^{''}u_p)^2 g_{ij} g_{rs} \\ & =
\sum_p |u| F^{ij,rs} (|u|\check h_{ij}) (f^{''}u_p)^2 g_{ij} g_{rs} \\
& =|u|  \sum_{p,i,j} \frac{\partial^2F(|u|\check \kappa_i) }{\partial \kappa_i \partial \kappa_j} (f^{''}u_p)^2
\end{aligned}
\end{equation}
Hierbei ist 
\begin{equation}
\sum_{i,j}\frac{\partial^2F(\tilde v |u||f^{'}|g_{ij}) }{\partial \kappa_i \partial \kappa_j} =0
\end{equation}
und daher
\begin{equation}
\begin{aligned}
 \left|\sum_{i,j}\frac{\partial^2F(|u|\check \kappa_i) }{\partial \kappa_i \partial \kappa_j} \right| & \le \sum_{i,j}  \left|\frac{\partial^2F(|u|\check \kappa_i) }{\partial \kappa_i \partial \kappa_j}-\frac{\partial^2F(\tilde v|u||f^{'}|g_{ij}) }{\partial \kappa_i \partial \kappa_j} \right| \\
 & \le \sum_{i,j}D^3F(x_{i,j}) |\check \kappa_i |u|- \tilde v|u||f^{'}|| \\
 & \le c u
\end{aligned}
\end{equation}
$x_{i,j} \in \Omega \subset \subset R^n_+$, hence
}
%%%%%%%%%%%%%%%%%%%Auskommentieren: Ende%%%%%%%%%%%%%%%%
So in view of Lemma \ref{suboptimal_decay_of_A} and Theorem \ref{optimal_decay_of _Du} we have choosing coordinates such that $(h_{ij})$ diagonal and $g_{ij}=\de_{ij}$

\bea
|F^{ij, rs}&\check h_{ij;p} \check  h_{rs;q}h^{pq}|
 \le  |F^{ij, rs} h_{ij;p} h_{rs;q}h^{pq}|\\
 &+  2|F^{ij, rs}h_{rs;p}g_{ij}(-\tilde v_qf^{'}-\tilde vf^{''}u_q+\psi_{\al\be}\nu^{\al}x^{\be}_q+\psi_{\al}x^{\al}_r h^r_q)h^{pq}| \\
 &+ |F^{ij, rs}g_{rs}g_{ij}\sum_p(-\tilde v_pf^{'}-\tilde vf^{''}u_p+\psi_{\al\be}\nu^{\al}x^{\be}_p+\psi_{\al}x^{\al}_r h^r_p)^2h^p_p| \\
 \le & %\frac{1}{F}\sum_p(F^{ij}h_{ij;p})^2\|A\| 
 c|u|\|DA\|^2\|A\|+ c \|A\||u| (\|DA\|+1).
\eea

The second term of the right side of inequality (\ref{diese_ungleichung_brauchen_wir_gleich}) can be estimated as follows
\begin{equation}\label{7.5.45}
\begin{aligned}
F^{-3} &\ (-2h^{ij} F_i F_j e^{2\ga t} - F f^{'''} h^{ij} \tilde u_i \tilde u_j \Fij g_{ij} \tilde v)  \le \\
&\ F^{-3}(-2 |f^{''}|^2+f^{'}f^{'''}) h^{ij}\tilde u_i \tilde u_j \tilde v^2 n^2 + cF^{-3}\|DA\|^2e^{2 \ga t} \\
&\ + ce^{-\epsilon t} (1 + \vp).
\end{aligned}
\end{equation}
Now, we observe that
\begin{equation}
(f^{''}+ \tilde \ga |f^{'}|^2)^{'} = f^{'''} + 2 \tilde \ga f^{'}f^{''} = C f^{'},
\end{equation}
where $C$ is a bounded function in view of (\ref{assumption_derivative_of_f}). Hence
\begin{equation}
2 |f^{''}|^2 - f^{'}f^{'''} = 2 |f^{''}|^2 + 2 \tilde \ga |f^{'}|^2 f^{''} - C|f^{'}|^2, 
\end{equation}
i.e.,
\begin{equation}
|2 |f^{''}|^2 - f^{'}f^{'''}| \le c |f^{'}|^2
\end{equation}
because of (\ref{fzweistrich_wiefstrich_quadrat}) and we conclude that the left-hand side of (\ref{7.5.45}) can be estimated from above by
\begin{equation}
ce^{-\epsilon t} (1 + \vp)+cF^{-2}\|DA\|^2e^{\ga t}.
\end{equation}
Next, we estimate
\begin{equation}
F^{-2}(n f^{''} \tilde v^2 + \ga F^2) \vp \le c e^{-\epsilon t} \vp
\end{equation}
and finally 
\begin{equation}
F^{-1} \Riemann {\nu}^{\al} x^{\be}_i {\nu}^{\ga} x^{\delta}_j h^{ij} e^{2 \ga t} \le c e^{-\epsilon t}(1+\vp) + F^{-1}\bar R_{0i0j}h^{ij}e^{2 \ga t}\tilde v^2,
\end{equation}
but 
\begin{equation}
\bar R_{0i0j} \le c|u|,
\end{equation}
cf. proof of Lemma \ref{arbitrary_decay_of_special_function}(iii).

Hence we deduce 
\begin{equation}
\dot \vp \le c e^{-\epsilon t} (1+\vp)
\end{equation}
for some positive $\epsilon$ and for a.e. $t\ge t_0$, i.e. $\vp$ is bounded.
\end{proof}

\section{Higher order estimates\---Asymptotic behaviour of the flow} \label{Section_9}
In this section and the following two sections many proofs are identical to the proofs in \cite{ARW}. For reasons of completeness and convenience for the reader we present them here.

Let us now introduce the following abbreviations

\begin{defi}
(i) For arbitrary tensors $S$, $T$ denote by $S*T$ any linear combination of contractions of $S\otimes T$. The result can be a tensor or a function. Note that we do not distinguish between $S*T$ and $cS*T$, where $c$ is a constant.

(ii) The symbol $A$ represents the second fundamental of the hypersurfaces $M(t)$ in $N$, $\tilde A = Ae^{\ga t}$ is the scaled version, and $D^mA$ resp. $D^m\tilde A$ represent the covariant derivative of order $m$.

(iii) For $m \in \mathbb{N}$ denote by $\tilde O_m$ a tensor expression defined on $M(t)$ that satisfies the pointwise estimate
\begin{equation}
\| \tilde O_m\| \le c_m(1+ \| \tilde A\|_m)^{p_m}
\end{equation}
and
\begin{equation}
\| D\tilde O_m\| \le c_m(1+ \| \tilde A\|_m)^{p_m}(1+\| D^{m+1}\tilde A\|),
\end{equation}
where $c_m, p_m > 0$ are constants and
\begin{equation}
\|\tilde A\|_m = \sum_{|\al| \le m} \| D^{\al}\tilde A\|.
\end{equation} 

(iv) For arbitrary $m\in \mathbb{N}$ denote by $O_m$ a tensor expression defined on $M(t)$ that satisfies
\begin{equation}
D^kO_m = \tilde O_{m+k} \quad \forall k \in \mathbb{N}.
\end{equation}

(v) By the symbol  $O$ we denote a tensor expression such that $DO=O_0$. 
\end{defi}

\begin{rem}
\beq
D^kO_m = O_{m+k} \quad \forall (k,m) \in \mathbb{N}\times \mathbb{N}.
\eeq
\end{rem}
\begin{lem} \label{m_te_Ableitung_von_u_mal_f_strich}
We have
\beq
D(uf^{'})=e^{-2 \ga t}O
\eeq
especially
\beq
D^m(uf^{'}) = e^{-2\ga t} O_{m-2},\quad m \ge 2.
\eeq
\end{lem}
\begin{proof}
Differentiating and adding a zero yields
\beq
D_i(uf^{'}) = u_i f^{'}(1-\tilde \ga f^{'}u) + uu_i (\tilde \ga |f^{'}|^2 + f^{''})
\eeq
from which we deduce the claim in view of (\ref{fzweistrich_wiefstrich_quadrat}), (\ref{assumption_derivative_of_f}) and (\ref{abschaetzung_der_ableitungen_von_f_nach_vor}).

\end{proof}
\begin{lem}
We have
\beq
D(u  \check h_{kl}) =e^{-2\ga t}O_0 + e^{-2\ga t}D\tilde A O.
\eeq
\end{lem}

\begin{proof}
Differentiating yields ($g_{ij}=\de_{ij}$, $h_{ij}=$ diagonal)
\bea
D_i(u \check h_{kl})= &u_i \check h_{kl} + u (h_{kl;i}-\eta_{\al\be}\nu^{\al}x^{\be}_if^{'}g_{kl} \\
&+ (\psi_{\al}\nu^{\al})_ig_{kl})+uu_i \kappa_i f^{'}g_{kl}-u\tilde v f^{''}u_ig_{kl}
\eea
and now we focus on the last term and write there
%\beq
%uf^{'}= \frac{1}{\tilde \ga}+ (uf^{'}-\frac{1}{\tilde \ga})
%\eeq
%and
\beq
f^{''}=(f^{''}+ \tilde \ga |f^{'}|^2) - \tilde \ga |f^{'}|^2. 
\eeq
Then all terms can be estimated obviously except for
\beq
u_i \check h_{kl}+\tilde \ga |f^{'}|^2u \tilde v u_i g_{kl},
\eeq
for which we use (\ref{asymptotic_relation_for_f_strich}). 
\end{proof}

\begin{cor}
We have 
\beq \label{Abkuerzung_Ordnung_uhij}
D^m(u  \check h_{kl}) = e^{-2\ga t}O_{m-1} + e^{-2\ga t}D^m \tilde A* O.
\eeq
\end{cor}

\begin{defi}
We denote by $\mathcal{D}^mF$ the derivatives of order $m$ of $F$ with respect to $\check h^i_j$.
\end{defi}

\begin{lem}
We have
\beq \label{Abkuerzung_Ordnung_Fij}
D^m \mathcal{D}F = e^{-2\ga t}O_{m-1}+ e^{-2 \ga t}D^m \tilde A *\mathcal{D}^2F(|u|\check h_{kl})*O,
\eeq
\beq \label{unterschied_Fij_g_ij_und_n}
|\Fij (\check h_{kl})g_{ij}-\Fij(g_{kl})g_{ij})| \le c e^{-2\ga t},
\eeq
\beq \label{Abkuerzung_Ordnung_F_}
DF = \mathcal{D}F *DA + e^{-\ga t}\mathcal{D}F *O_{0}+ e^{\ga t}\mathcal{D}F *O,
\eeq
and 
\beq \label{Abkuerzung_Ordnung_F}
D^mF = \mathcal{D}F *D^mA + e^{-\ga t} O_{m-1}+ e^{\ga t} O_{m-2},
\eeq
for $m\ge 2$.
\end{lem}
\begin{proof}
To prove  (\ref{Abkuerzung_Ordnung_Fij}) we write 
\beq
\D F(\check h_{kl}) = \D F(|u|\check h_{kl})
\eeq
and infer
\beq
D\D F(\check h_{kl})=\D^2 F(|u|\check h_{kl})D(|u|\check h_{kl}),
\eeq
hence the desired result follows in view of (\ref{Abkuerzung_Ordnung_uhij}) and the fact that 
\beq
\|\D^mF(|u|\check h_{kl})\|  
\eeq
is bounded for all $m \in \mathbb{N}$.

(\ref{unterschied_Fij_g_ij_und_n}) is proved by applying the mean value theorem.

(\ref{Abkuerzung_Ordnung_F_}) follows by
\beq
F_k = \Fij h_{ij;k} + \Fg (-\tilde v_k f^{'}-\tilde v f^{''}u_k+\psi_{\al\be}\nu^{\al}x^{\be}_k + \psi_{\al}x^{\al}_rh^r_k).
\eeq

To prove (\ref{Abkuerzung_Ordnung_F}) we differentiate (\ref{Abkuerzung_Ordnung_F_}) and get
\bea
D^2F = &D\D F*DA + \D F*D^2A + e^{-\ga t}D\D F*O_0\\
&+e^{-\ga t}\D F*O_1 + e^{\ga t}D\D F*O+e^{\ga t}\D F*O_0 \\
=& \D F *D^2A+e^{-\ga t}O_1+e^{\ga t}O_0,
\eea
from which the claim follows easily.
\end{proof}

Now we want to write the evolution equation for $\tilde h^k_l$ in the form
\bea \label{zusammengefasst_evol_hkl}
\dot{\tilde h}^j_i - F^{-2}F^{kl}\tilde h^j_{i;kl} =& F^{-3}D\tilde A *DA *O_0 +F^{-2}D\tilde A *O_0 + F^{-1}O_0 \\
& + F^{-3}D \tilde A *DA *O.
\eea

To check this we consider all the terms in (\ref{evol_hkl}) separately and start with
\beq
 \boxed{(-2F^{-3}F^kF_l- F^{-2}\Fg f^{'''}u^ku_l\tilde v)e^{\ga t}}.
 \eeq
We have
\bea
F_k=&F^{rs}h_{rs;k} -\eta_{\al\be}\nu^{\al}x^{\be}_kf^{'}F^{rs}g_{rs}-\eta_{\al}x^{\al}_ih^i_kf^{'}F^{rs}g_{rs}-\tilde v f^{''}u_kF^{rs}g_{rs}\\
&+\psi_{\al\be}\nu^{\al}x^{\be}_kF^{rs}g_{rs} + \psi_{\al}x^{\al}_rh^r_kF^{rs}g_{rs}\\
=& A_1+A_2+A_3+A_4+A_5+A_6,
\eea
hence
\bea
(-2 F^{-3}A_4A_4 -& F^{-2}\Fg f^{'''}u^ku_l\tilde v)e^{\ga t} \\
= &F^{-3}(-2 |f^{''}|^2 + f^{'}f^{'''})\tilde v^2 (\Fg)^2u_lu^k e^{\ga t}\\
&-F^{-2}F^{rs}h_{rs}\Fg f^{'''}u^ku_l\tilde v e^{\ga t}\\
&-F^{-2}(\Fg)^2\psi_{\al}\nu^{\al} f^{'''}u^ku_l\tilde v e^{\ga t}\\
=& F^{-1}O_0,
\eea
where we observed that
\beq
\vp = -2 |f^{''}|^2 + f^{'}f^{'''} = (f^{''}+\tilde \ga |f^{'}|^2)^{'}f^{'}-2f^{''}(f^{''}+\tilde \ga |f^{'}|^2).
\eeq
In view of the assumptions on $f$ the spatial derivatives of $\varphi$ can be estimated by
\beq
\|D^m\vp\| \le c_m (1+\|\tilde u\|_{m-1})^{p_{m-1}}(1+\|D^m\tilde u\|)e^{2\ga t} \quad \forall m \in \mathbb{N}^{*}
\eeq
for some suitable $p_{m-1}\in \mathbb{N}$.
Furthermore, we have
\beq
-2F^{-3}A_1A_1e^{\ga t} = F^{-3}O_0*D\tilde A* DA.
\eeq
All remaining terms are estimated as follows
\bea
-2F^{-3}e^{\ga t}\sum_{(i,j) \notin \{(1,1), (4,4)\}}A_iA_j = F^{-2}D\tilde A*O_0 +F^{-2}O_0 \eea
hence 
\bea
(-2F^{-3}F^kF_l- & F^{-2}\Fg f^{'''}u^ku_l\tilde v)e^{\ga t} =\\Ê
 &F^{-3}D\tilde A*DA* O_0 + F^{-2}D\tilde A*O_0 +F^{-1}O_0.
\eea
Now, there are some quiet easy estimates, namely
\bea
\boxed{2F^{-1}h^{kr}h_{rl}e^{\ga t}}&=F^{-2}O_0 \\
\boxed{-F^{-2}e^{\ga t}\Fij h_{aj}h^a_ih^k_l}&=F^{-2}O_0\\
\boxed{2F^{-2}g^{pk}\Fij\Riemann x^{\al}_rx^{\be}_px^{\ga}_ix^{\de}_lh^r_je^{\ga t}}&=F^{-2}O_0\\
\boxed{F^{-2}g^{pk}\Fij\bar R_{\al\be\ga\de;\epsilon} \nu^{\al}x^{\be}_px^{\ga}_ix^{\de}_lx^{\epsilon}_je^{\ga t}}&=F^{-1}O_0.
\eea
Furthermore, we have
\bea
&\boxed{
F^{-1}e^{\ga t}\Riemann \nu^{\al}x^{\be}_l\nu^{\ga}x^{\de}_rg^{rk}}=-F^{-1}\Riemann \eta^{\al}\tilde vx^{\be}_l\nu^{\ga}x^{\de}_re^{\ga t} \\
&\quad \quad-F^{-1}\bar R_{i\be\ga\de}u^ix^{\be}_l\eta^{\ga}x^{\de}_rg^{rk}e^{\ga t} + F^{-1}\bar R_{i\be j \de}\tilde v^{-2}u^ix^{\be}_lu^jx^{\de}_rg^{rk}e^{\ga t} \\
& \quad \quad = F^{-2}O,
\eea
%$\underline{F^{-2}e^{\ga t}g^{pk}F^{ij,rs}\check h_{ij;p}\check h_{rs;l}}$
\bea
&\boxed{F^{-2}e^{\ga t}g^{pk}F^{ij,rs}\check h_{ij;p}\check h_{rs;l}} = F^{-2}e^{\ga t}g^{pk}\check h_{ij;p}D_l\Fij \\
&\quad \quad = F^{-2}D\tilde A* O + F^{-3}D\tilde A *DA *O + F^{-3}D\tilde A *O_0 + F^{-2}O_0
\eea
and
\beq
\boxed{F^{-2}e^{\ga t}\Fg(\psi_{\al}x^{\al}_rh^{rk}_{\ \ ;l}+f^{'}u^rh^k_{l;r})} = F^{-2}D\tilde A* O_0,
\eeq
so that only the following term is left
%$\underline{F^{-2}e^{\ga t}\Fg f^{''}\tilde v^2h^k_l + \ga h^k_le^{\ga t}}$
\beq
\boxed{F^{-2}e^{\ga t}\Fg f^{''}\tilde v^2h^k_l + \ga h^k_le^{\ga t}} = F^{-2}e^{\ga t}(\Fg f^{''}\tilde v^2 h^k_l + \ga F^2 h^k_l).
\eeq
There holds
\bea
F^2 =& (\Fij h_{ij})^2-2 \Fij h_{ij}\tilde v f^{'}F^{rs}g_{rs} + 2 \Fij h_{ij}\psi_{\al}\nu^{\al}F^{rs}g_{rs} \\
& +\tilde v^2 |f^{'}|^2 (\Fg)^2-2 \psi_{\al}\nu^{\al}\tilde v f^{'}(\Fg)^2 + (\psi_{\al}\nu^{\al})^2(\Fg)^2
\eea
and
\bea
f^{''}\tilde v^2 h^k_l + \ga h^k_l \tilde v^2|f^{'}|^2 \Fg =\ &  \tilde v^2 h^k_l (f^{''}+\ga |f^{'}|^2n) \\
&+ \tilde v^2 h^k_l\ga |f^{'}|^2 (\Fg-n),
\eea
so that we infer 
\bea
F^{-2}e^{\ga t}\Fg f^{''}\tilde v^2h^k_l + \ga h^k_le^{\ga t}= F^{-2}O_0.
\eea

Using the fact that
\beq
\dot g_{ij} = -2 F^{-1}h_{ij}= F^{-2}O_0
\eeq
(\ref{zusammengefasst_evol_hkl}) is proved.
%\bea \label{7.6.18}
%\dot {\tilde A} - F^{-2}F^{kl}\tilde A_{kl} =& F^{-3}D\tilde A DA O_0 +F^{-2}D\tilde A O_0 + F^{-1}O_0 \\
%& + F^{-3}D \tilde A DA O
%\eea 

\medskip
 
Differentiating (\ref{zusammengefasst_evol_hkl}) covariantly with respect to a spatial variable we deduce
\begin{equation} \label{7.6.19}
\begin{aligned}
\frac{D}{dt}&(D\tilde A) -F^{-2}F^{ij} (D\tilde A)_{ij} =  F^{-1}O_0+F^{-3}D^2\tilde A * DAO_0 \\
&+ F^{-2}O_0*D^2\tilde A + F^{-4}D\tilde A * DA*DA*O_0+F^{-3}O_0*D\tilde A*DA \\
& + F^{-2}D\tilde A*O_0 + F^{-3}D\tilde A*DA*DO_0+F^{-2}D\tilde A*DO_0+F^{-1}DO_0.
\end{aligned} 
\end{equation}
And using induction we conclude for $m \in \mathbb{N^{*}}$
\bea \label{7.6.20}
\frac{D}{dt}(D^{m+1}\tilde A)&-F^{-2}\Fij (D^{m+1}\tilde A)_{ij} = F^{-1}O_m \\
&+ \Theta F^{-3}D^{m+1}\tilde A*D^{m+1}A *O_0 +F^{-3}D^{m+2}\tilde A*DA*O_0 \\
&+ F^{-2}D^{m+1}\tilde A *O_m+F^{-2}D^{m+2}\tilde A *O_0 \\
&+ F^{-2}DO_m, 
\eea
where $\Theta =1$ if $m=1$ and  $\Theta =0$ else.

We are now going to prove uniform bounds for $\frac{1}{2}\|D^{m+1}\tilde A\|$ for all $m \in \mathbb{N}$. First we observe that
\bea
\frac{D}{dt}&(\frac{1}{2}\|D\tilde A\|^2)-F^{-2}\Fij\frac{1}{2}(\|D\tilde A\|^2)_{ij} = -F^{-2}\Fij(D\tilde A)_i(D\tilde A)_j \\
&+ F^{-1}O_0*D\tilde A +F^{-3}D^2\tilde A * DA*O_0*D\tilde A + F^{-2}O_0*D^2\tilde A*D\tilde A \\
&+F^{-4}D\tilde A * DA*DA*O_0*D\tilde A +F^{-3}O_0*D\tilde A*DA*D\tilde A \\
&+ F^{-2}D\tilde A*O_0*D\tilde A +F^{-3}D\tilde A*DA*DO_0*D\tilde A \\
& +F^{-2}D\tilde A*DO_0*D\tilde A+F^{-1}DO_0*D\tilde A.
\eea

Furthermore we have for $m \in \mathbb{N^{*}}$
\bea \label{gleichung_fuer_hoehere_Ableitung_von_A_zum_Quadrat}
\frac{D}{dt}&(\frac{1}{2}\|D^{m+1}\tilde A\|^2)-F^{-2}\Fij \frac{1}{2}(\|D^{m+1}\tilde A\|^2)_{ij} =\\
& -F^{-2}\Fij(D^{m+1}\tilde A)_i(D^{m+1}\tilde A)_j +F^{-1}O_m*D^{m+1}\tilde A \\
&+\Theta F^{-3}D^{m+1}\tilde A*D^{m+1}A *O_0*D^{m+1}\tilde A \\
&+F^{-3}D^{m+2}\tilde A*DA*O_0*D^{m+1}\tilde A +F^{-2}D^{m+1}\tilde A* O_m*D^{m+1}\tilde A\\
&+F^{-2}D^{m+2}\tilde A* O_0*D^{m+1}\tilde A + F^{-2}DO_m*D^{m+1}\tilde A.
\eea

\begin{theorem} \label{7.6.3}
The quantities $\frac{1}{2}\|D^m\tilde A\|^2$ are uniformly bounded during the evolution for all $m\in \mathbb{N}^{*}$
\end{theorem}
\begin{proof}
We prove the theorem recursively by estimating
\begin{equation}
\varphi = \log{\frac{1}{2}\|D^{m+1}\tilde A\|^2} + \mu \frac{1}{2}\|D^m\tilde A\|^2 + \lambda e^{-\ga t},
\end{equation}
where $\mu$ is a small positive constant
and $\lambda >> 1$ large.

We shall only treat the case $m=0$.
%, since then the structure of the right side is worst, at least formally, cf. (\ref{gleichung_fuer_hoehere_Ableitung_von_A_zum_Quadrat}).

Fix $0<T<\infty$, $T$ very large, and suppose that 
\begin{equation}
0 < \sup_{[0,T]}\sup_{M(t)}\varphi = \varphi(t_0, \xi_0)
\end{equation}
for large $0<t_0 \le T$.

Applying the maximum principle we deduce
\bea
0 \le& \frac{1}{\|D\tilde A\|^2}(\frac{D}{dt}\|D\tilde A\|^2-F^{-2}\Fij\|D\tilde A\|^2_{ij}) + \mu \tilde A(\dot{\tilde A}-F^{-2}\Fij \tilde A_{ij}) \\
&+F^{-2}\Fij \tilde A_i \tilde A_j(-\mu+\mu^2\tilde A^2) - \la \ga e^{-\ga t} \\
\le& -\frac{1}{2} \frac{1}{\|D\tilde A\|^2}F^{-2}\Fij(D\tilde A)_i(D\tilde A)_j-\frac{\la}{2} \ga e^{-\ga t} + c F^{-4}\|D\tilde A\|^2 \\
 &+c F^{-2}\|D\tilde A\| + F^{-2}\Fij \tilde A_i \tilde A_j(-\mu+\mu^2\tilde A^2) \\
 <&0,
\eea
here we assumed that $\|D\tilde A\|$ is larger than a sufficiently large positive constant that does not depend on $t_0, T$.

Thus $\varphi$ is a priori bounded.

The proof for $m\ge 1$ is similar.
\end{proof}

\auskommentieren{
\subsection{Anhang}
\begin{equation}
\begin{aligned}
\sum_p e^{\ga t}F^{-2} F^{ij,rs}& (\check h_{ij}) (\tilde v f^{'})_p g_{ij}(\tilde v f^{'})_p g_{rs} \\ & =
\sum_p e^{\ga t}F^{-2}|u| F^{ij,rs} (|u|\check h_{ij}) (\tilde v f^{'})_p g_{ij}(\tilde v f^{'})_p g_{rs} \\
& = e^{\ga t}F^{-2}|u|  \sum_{p,i,j} \frac{\partial^2F(|u|\check h_{ij}) }{\partial \kappa_i \partial \kappa_j} (\tilde v f^{'})_p (\tilde v f^{'})_p \\
& \le e^{\ga t}F^{-2}|u| c u^2 c(|f^{'}|^2 + A^2) \\
& = F^{-2}O_0
\end{aligned}
\end{equation}
Hierbei ist 
\begin{equation}
\sum_{i,j}\frac{\partial^2F(\tilde v |u|f^{'}g_{ij}) }{\partial \kappa_i \partial \kappa_j} =0
\end{equation}
und daher
\begin{equation}
\begin{aligned}
 \left|\sum_{i,j}\frac{\partial^2F(|u|\check h_{ij}) }{\partial \kappa_i \partial \kappa_j} \right| & \le \sum_{i,j}  \left|\frac{\partial^2F(|u|\check h_{ij}) }{\partial \kappa_i \partial \kappa_j}-\frac{\partial^2F(\tilde v|u|f^{'}g_{ij}) }{\partial \kappa_i \partial \kappa_j} \right| \\
 & \le \sum_{i,j}D^3F(x_{i,j}) |\check h_{ij} |u|- \tilde v|u|f^{'}g_{ij}| \\
 & \le c u^2
\end{aligned}
\end{equation}
$x_{i,j} \in \Omega \subset \subset R^n_+$
}

\section{Convergence of $\tilde u$ and the behaviour of derivatives in $t$} \label{Section_10}
\begin{lem} \label{tilde_u_konvergiert_in_Cm}
$\tilde u$ converges in $C^m(S_0)$ for any $m \in \mathbb{N}$, if $t$ tends to infinity, and hence $D^m\tilde A$ converges.
\end{lem}
\begin{proof}
$\tilde u$ satisfies the evolution equation
\beq
\dot {\tilde u} = \frac{\tilde v e^{\ga t}}{F} (1- \ga u f^{'}\Fg + \frac{\ga u}{\tilde v} \Fij h_{ij} + \frac{\ga u }{\tilde v} \psi_{\al}\nu^{\al}\Fg).
\eeq
Using (\ref{unterschied_Fij_g_ij_und_n}) and the already known exponential decays we deduce
\beq
|\dot {\tilde u} | \le ce^{-2 \ga t},
\eeq
hence $\tilde u$ converges uniformly. Due to Theorem \ref{7.6.3} $D^m \tilde u$ is uniformly bounded, hence $\tilde u$ converges in $C^m(S_0)$.

The convergence of $D^m\tilde A$ follows from Theorem \ref{7.6.3} and the convergence of $\tilde h_{ij}$, which in turn can be deduced from 
\beq
h_{ij}\tilde v = -u_{ij}+\bar h_{ij}.
\eeq
\end{proof}

Combining the equations (\ref{zusammengefasst_evol_hkl}), (\ref{7.6.19}) and (\ref{7.6.20}) we immediately conclude
\begin{lem} \label{7.7.2}
$ \|\frac{D}{dt}D^m \tilde A\|$ and $ \|\frac{D}{dt}D^m A\|$ decay by the order $e^{-\ga t}$ for any $m \in \mathbb{N}$.
\end{lem}

\begin{cor} \label{7.7.3}
$\frac{D}{dt}D^mAe^{\ga t}$ converges, if $t$ tends to infinity.
\end{cor}
\begin{proof}
Applying the product rule we obtain
\beq
\frac{D}{dt}D^m\tilde A = \frac{D}{dt}D^m A e^{\ga t} + \ga D^m \tilde A,
\eeq
hence the result, since the left-hand side converges to zero and $D^m\tilde A$ converges.
\end{proof}

\begin{cor} \label{7.7.5t}
We have
\beq \label{7.7.7}
\|D^m F^{-1}\| \le c_m F^{-1} \quad \forall m \in \mathbb{N}.
\eeq
\end{cor}
\begin{proof}
Use (\ref{Abkuerzung_Ordnung_F}).
\end{proof}

In the next Lemmas we prove some auxiliary estimates.

\begin{lem} \label{7.7.6}
The following estimates are valid
\beq \label{7.7.11}
\|D\dot u\| \le c e^{-\ga t},
\eeq
\beq \label{7.7.12}
\|\frac{d}{dt}F^{-1}\|\le cF^{-1},
\eeq
and
\beq \label{7.7.13}
|\dot {\tilde v}| \le c e^{-2\ga t}.
\eeq
\end{lem}
\begin{proof}
 "(\ref{7.7.11})" The estimate follows immediately from
 \beq
 \dot u = \frac{\tilde v}{F},
 \eeq
 in view of Corollary \ref{7.7.5t}.
 
  "(\ref{7.7.12})" Differentiating with respect to $t$ we obtain
  \bea
  \frac{d}{dt}F^{-1}=& -F^{-2}(\Fij \dot h_{ij} - \dot {\tilde v}f^{'}\Fg - \tilde v f^{''}\dot u \Fg + \frac{d}{dt}(\psi_{\al}\nu^{\al})\Fg)\\
  &+ \Fij (-\tilde v f^{'}+\psi_{\al}\nu^{\al})\dot g_{ij}
  \eea
  and the result follows from (\ref{7.7.13}) and the known estimates for $|\dot u|$ and $F$.
  
  "(\ref{7.7.13})" We differentiate the relation $\eta_{\al}\nu^{\al}$ to get
  \bea \label{7.7.16}
  \dot {\tilde v} &= \eta_{\al \be}\nu^{\al}\dot x^{\be} + \eta_{\al}\dot \nu^{\al} \\
  & = -\eta_{\al \be}\nu^{\al}\nu^{\be}F^{-1}+(F^{-1})_ku^k,
  \eea	
  cf. (\ref{hij_in_konformer_metrik_einfach_version}),
  yielding the estimate for $|\dot{\tilde v}|$, in view of Corollary \ref{7.7.5t}. 
\end{proof}

\begin{lem} \label{haeuserbau}
\beq \label{haus1}
\|F^{ij,kl}(\check h_{rs})g_{ij}\| \le c e^{-3\ga t},
\eeq
\beq \label{haus2}
\|F^{ij,kl}(|u|\check h_{rs})g_{ij}\| \le c e^{-2\ga t},
\eeq
\beq \label{haus3}
\|\frac{D}{dt}(u \check h_{kl})\|\le ce^{-2 \ga t},
\eeq
\beq \label{haus4}
\|\frac{D}{dt}F^{ij,kl}(\check h_{rs})g_{ij}\| \le c e^{-3\ga t},
\eeq
\beq \label{haus4_}
\|\frac{D}{dt}F^{ij,kl}(\check h_{rs})\| \le c e^{-\ga t}
\eeq
\beq \label{haus5}
\|\frac{D}{dt}\Fij\| \le ce^{-2 \ga t},
\eeq
\beq \label{haus6}
\|\frac{D}{dt}D \check h_{kl} \|+\|\frac{D}{dt}DF\| \le ce^{\ga t}.
\eeq
\end{lem}
\begin{proof}
"(\ref{haus1})" Use Theorem \ref{optimal_decay_A}, (\ref{Fijrs0}) and (\ref{Baum1}).

"(\ref{haus2})"  Obvious.

"(\ref{haus3})" We have
\bea
\frac{D}{dt}(u\check h_{kl}) &= \dot u \check h_{kl}+u \dot {\check h}_{kl} \\
=&\frac{\tilde v}{F}\check h_{kl}+u \dot h_{kl}-u \dot{\tilde v}f^{'}g_{kl}-u \tilde v f^{''}\frac{\tilde v}{F}g_{kl}-u \tilde v f^{'}\dot g_{kl}\\
&+u \frac{D}{dt}(\psi_{\al}\nu^{\al}g_{kl}),
\eea
hence in view of (\ref{7.7.13})
\beq
\|\frac{D}{dt}(u\check h_{kl})\| \le ce^{-2\ga t} + n|\frac{\tilde v^2}{F} (-f^{'}-uf^{''})| \le ce^{-2\ga t}.
\eeq
Here, concerning the summand
\beq
\frac{\tilde v}{F}\check h_{kl}-u \tilde v f^{''}\frac{\tilde v}{F}g_{kl},
\eeq
we use
\beq
|- f^{'}-uf^{''}| \le \big|-f^{'}+u\tilde \ga |f^{'}|^2\big| + c |u|,
\eeq
which follows from (\ref{fzweistrich_wiefstrich_quadrat}), and then (\ref{asymptotic_relation_for_f_strich}).

"(\ref{haus4}), (\ref{haus4_})" We have
\bea
\frac{D}{dt}F^{ij,kl}(\check h_{rs})= \frac{D}{dt}(|u|F^{ij,kl}(|u|\check h_{rs}))
\eea
which implies the claim together with (\ref{haus2}) und (\ref{haus3}).

"(\ref{haus5})" Use  (\ref{haus3}) and $\Fij(\check h_{rs})=\Fij(|u|\check h_{rs})$.

"(\ref{haus6})" Obvious in view of (\ref{7.7.11}) and (\ref{7.7.13}).
\end{proof}

\begin{lem} 
We have
\beq \label{7.7.13_}
|  \ddot {\tilde v}|+\|D\dot{\tilde v}\| \le c e^{-2\ga t}
\eeq
and $\dot{\tilde v}e^{2 \ga t}$ and $\ddot{\tilde v}e^{2 \ga t}$ converge, if $t$ goes to infinity.
\end{lem}
\begin{proof}
Differentiating (\ref{7.7.16}) covariantly with respect to $x$ we infer the estimate for $\|D\dot{\tilde v}\|$. 
A direct computation and easy check of each of the (many) appearing terms yield the convergence of $\dot{\tilde v}e^{2 \ga t}$ and $\ddot{\tilde v}e^{2 \ga t}$, especially the lemma is proved.
\end{proof}

Finally let us estimate ${\ddot h}_i^j$ and $\ddot{\tilde h}^j_i$. 

%We have
%\beq
%\|\frac{D}{dt}\check h_{kl}\|+\|\frac{D}{dt}D\check h_{kl}\|+|\frac{D}{dt}F|+\|\frac{D}{dt}DF\|\le c e^{\ga t},
%\eeq
%\beq
%\|\frac{D}{dt}\Fij\| \le c
%\eeq
%and 
%\beq
%\|\frac{D}{dt}F^{ij,rs}\|\le c e^{-\ga t}.
%\eeq

\begin{lem} \label{7.7.7}
${\ddot h}_i^j$ and $\ddot{\tilde h}^j_i$ decay like $e^{-\ga t}$. 
\end{lem}
\begin{proof}
The estimate for $\ddot{h}_i^j$ follows immediately by differentiating equation (\ref{evol_hkl}) covariantly with respect to $t$ and by applying the above lemmata as well as Theorem \ref{7.6.3}.

Now we estimate $\ddot{\tilde h}^j_i$. We have
\beq \label{abschaetzung_hkl_tt}
\ddot{\tilde h}^k_l =  e^{\ga t}\ddot{ h}^k_l + 2\ga e^{\ga t}\dot{h}^k_l+\ga^2e^{\ga t}h^k_l.
\eeq

Now we insert (\ref{evol_hkl}) and the equation which results from (\ref{evol_hkl}) after covariant differentiation with respect to $t$ into (\ref{abschaetzung_hkl_tt}).

Then many of the appearing terms decay like $e^{-\ga t}$ obviously. To see the decay of the remaining terms, namely
\bea
\frac{D}{dt}&(F^{-3}F^kF_l)e^{\ga t}+\frac{D}{dt}(F^{-2}g^{pk}F^{ij,rs}\check h_{ij;p}\check h_{rs;l})e^{\ga t}  \\
&-\frac{D}{dt}(F^{-2}\Fij g_{ij}u_lu^k\tilde vf^{'''})e^{\ga t}
+\frac{D}{dt}(F^{-2}f^{''}\tilde v^2 h^k_l\Fij g_{ij})e^{\ga t}\\
&-4\ga F^{-3} F^kF_le^{\ga t}
+2\ga F^{-2}g^{pk}F^{ij,rs}\check h_{ij;p}\check h_{rs;l}e^{\ga t} \\
&- 2\ga F^{-2}u_lu^k\tilde v f^{'''}\Fg e^{\ga t}\\
&+ 2 \ga e^{\ga t}F^{-2}f^{''}\tilde v^2h^k_l\Fg +\ga ^2e^{\ga t}h^k_l \\
= & S_1 + ... + S_9,
\eea
we use the technique developed in (\ref{zusammengefasst_evol_hkl}) et seq., confer also the proof of Theorem \ref{optimal_decay_A}, to rearrange terms. In this way we see the claimed decay of $S_5+S_7$ and $\frac{1}{2}S_8+S_9$. The summand $S_1+S_3$ can be handled similar. The summand $S_2$ decays as it should due to Lemma \ref{haeuserbau}. $S_6$ is obvious. To estimate $S_4+\frac{1}{2}S_8$ we differentiate in $S_4$ by product rule and use (\ref{evol_hkl}) to substitute $\dot h^k_l$. Then a little bit rearranging terms leads to the desired estimate.
\end{proof}

From Corollary \ref{7.7.3}, Lemma \ref{7.7.7} and (\ref{abschaetzung_hkl_tt}) we infer
\begin{cor} \label{7.7.8}
The tensor $\ddot h_i^je^{\ga t}$ converges, if $t$ tends to infinity.
\end{cor}

The claims in Theorem \ref{main_theorem} are now almost all proved with the exception of two. In order to prove the remaining claims we need:

\begin{lem}
The function $\varphi = e^{\tilde \ga f}u^{-1}$ converges to $-\tilde \ga \sqrt{m}$ in $C^{\infty}(S_0)$, if $t$ tends to infinity.
\end{lem}
\begin{proof}
$\varphi$ converges to $-\tilde \ga \sqrt {m}$ in view of (\ref{limes_ist_gleich_der_masse}). Hence, we only have to show that 
\beq
\|D^m\varphi \| \le c_m \quad \forall m \in \mathbb{N}^*,
\eeq
which will be achieved by induction.

We have 
\beq
\varphi_i = \tilde \ga e^{\tilde \ga f}f^{'}u_i u^{-1}- e^{\tilde \ga f} u^{-2}u_i = \varphi (\tilde \ga f^{'}u-1)u^{-1}u_i.
\eeq
Now, we observe that
\beq
u^{-1}u_i = \tilde u^{-1} \tilde u_i
\eeq
and $f^{'}u$ have uniformly bounded $C^m$-norms in view of Lemma \ref{tilde_u_konvergiert_in_Cm} and Lemma \ref{m_te_Ableitung_von_u_mal_f_strich}.

The proof of the lemma is then completed by a simple induction argument.
\end{proof}

When we formulated Theorem \ref{main_theorem} (iii) and (iv) we did not use the current notation where we distinguish quantities related to $\breve g_{\al\be}$ in contrast to those related to $\bar g_{\al \be}$ by the superscript 
$\breve\ $. 

In the following two lemmas we reformulate Theorem \ref{main_theorem} (iii) and (iv) using the current notation.
\begin{lem}
Let $(\breve g_{ij})$ be the induced metric of the leaves $M(t)$ of the IFCF, then the rescaled metric 
\beq
e^{\frac{2}{n}t}\breve g_{ij}
\eeq
converges in $C^{\infty}(S_0)$ to 
\beq
(\tilde \ga^2 m)^{\frac{1}{\tilde \ga}}(-\tilde u)^{\frac{2}{\tilde \ga}}\bar \sigma_{ij},
\eeq
where we are slightly ambiguous by using the same symbol to denote $\tilde u (t, \cdot)$ and $\lim \tilde u(t, \cdot)$.
\end{lem}
\begin{proof}
There holds
\beq
\breve g_{ij} = e^{2f}e^{2 \psi}(-u_iu_j + \sigma_{ij}(u,x)).
\eeq
Thus, it suffices to prove that
\beq \label{(7.7.24)}
e^{2f}e^{\frac{2}{n}t}\rightarrow (\tilde \ga^2 m)^{\frac{1}{\tilde \ga}}(-\tilde u)^{\frac{2}{\tilde \ga}}
\eeq
in $C^{\infty}(S_0)$. But this is evident in view of the preceding lemma, since
\beq
e^{2f}e^{\frac{2}{n}t}=(-e^{\tilde \ga f}u^{-1})^{\frac{2}{\tilde \ga}}(-\tilde u)^{\frac{2}{\tilde \ga}}.
\eeq
\end{proof}

\begin{lem}
The leaves $M(t)$ of the IFCF get more umbilical, if $t$ tends to infinity, namely
\beq
 \breve F^{-1}|\breve h^j_i-\frac{1}{n}\breve H\delta^j_i | \le ce^{-2\ga t}.
\eeq
In case $n+\omega - 4 > 0$, we even get a better estimate, namely
\beq
|\breve h^j_i-\frac{1}{n}\breve H\delta^j_i| \le c e^{-\frac{1}{2n}(n+\omega -4)t}.
\eeq
\end{lem}
\begin{proof}
Denote by $\breve h_{ij}, \breve \nu$, etc., the geometric quantities of the hypersurfaces $M(t)$ with respect to the original metric $(\breve g_{\al \be})$ in $N$, then
\beq
e^{\tilde \psi} \breve h_i^j = h^j_i + \tilde \psi_{\al}\nu^{\al}\delta^j_i, \quad \breve F = e^{-\tilde \psi}F
\eeq
and hence,
\beq
\breve F^{-1}|\breve h^j_i-\frac{1}{n}\breve H \delta^j_i|= F^{-1}|h^j_i-\frac{1}{n}H\delta^j_i | \le ce^{-2\ga t}.
\eeq
In case $n+\omega - 4 > 0$, we even get a better estimate, namely
\beq
|\breve h^j_i - \frac{1}{n}\breve H \delta^j_i|= e^{-\psi}e^{-f}e^{-\frac{1}{n}t}|h^j_i-\frac{1}{n}H\delta^j_i|e^{\ga t}e^{(\frac{1}{n}-\ga)t} \le c e^{-\frac{1}{2n}(n+\omega -4)t}
\eeq
in view of (\ref{(7.7.24)}).
\end{proof}

\section{Transition from big crunch to big bang} \label{transition_from_big_crunch_to_big_bang} \label{Section_11}
We shall define a new spacetime $\hat N$ by reflection and time reversal such that the IFCF in the old spacetime transforms to an IFCF in the new one.

By switching the light cone we obtain a new spacetime $\hat N$. If we extend $F$, which is defined in the positive cone $\Gamma_+\subset \mathbb{R}^n$, to $\Gamma_+ \cup (-\Gamma_+)$ by
\beq
F(\kappa_i)=-F(-\kappa_i)
\eeq 
for $(\kappa_i) \in -\Gamma_+$ 
the flow equation in $N$ is independent of the time orientation, and we can write it as
\beq
\dot x = -\breve F^{-1}\breve \nu = -(-\breve F)^{-1}(-\breve \nu) =: - \hat F^{-1}\hat \nu,
\eeq
where the normal vector $\hat \nu=-\check \nu$ is past directed in $\hat N$ and the curvature $\hat F = -\breve F$ negative.

Introducing a new time function $\hat x^0=-x^0$ and formally new coordinates $(\hat x^{\al})$ by setting
\beq
\hat x^0 = -x^0, \quad \hat x^i=x^i,
\eeq
we define a spacetime $\hat N$ having the same metric as $N$--only expressed in the new coordinate system--such that the flow equation has the form
\beq \label{flow_equation_neues_universum}
\dot{\hat x}=-{\hat F}^{-1}\hat \nu,
\eeq
where $M(t)=\graph \hat u(t)$, $\hat u = -u$, and 
\beq
(\hat \nu^{\al}) = -\tilde v e^{-\tilde \psi}(1, \hat u^i)
\eeq
in the new coordinates, since
\beq
\hat \nu^{0} = -\breve \nu^{0}\frac{\partial \hat x^0}{\partial x^0} = \breve \nu^{0}
\eeq
and
\beq
\hat \nu^i = -\breve \nu^i.
\eeq
The singularity in $\hat x^0=0$ is now a past singularity, and can be referred to as a big bang singularity.

The union $N\cup\hat N$ is a smooth manifold, topologically a product $(-a, a)\times S_0$--we are well aware that formally the singularity $\{0\}\times S_0$ is not part of the union; equipped with the respective metrics and time orientations it is a spacetime which has a (metric) singularity in $x^0=0$. The time function
\beq \label{new_time_in_gesamt_universum}
\hat x^0=\begin{cases}
  x^0,  & \text{in }N\\
  -x^0, & \text{in }\hat N
\end{cases}
\eeq
is smooth across the singularity and future directed.

$N \cup \hat N$ can be regarded as a cyclic universe with a contracting part $N=\{\hat x^0<0\}$ and an expanding part $\hat N=\{\hat x^0>0\}$ which are joined at the singularity $\{\hat x^0=0\}$.

We shall show that the inverse $F$-curvature flow, properly rescaled, defines a natural $C^3$-diffeomorphism across the singularity and with respect to this diffeomorphism we speak of a transition from big crunch to big bang.

The inverse $F$-curvature flows in $N$ and $\hat N$ can be uniformly expressed in the form
\beq \label{flow_equation_gesamt_universum}
\dot{\hat x}=-\hat F^{-1}\hat \nu,
\eeq
where (\ref{flow_equation_gesamt_universum}) represents the original flow in $N$, if $\hat x^0<0$, and the flow in (\ref{flow_equation_neues_universum}), if $\hat x^0>0$.

Let us now introduce a new flow parameter
\beq \label{7.8.9}
s = \begin{cases}
-\ga^{-1}e^{-\ga t}, & \text{for the flow in } N\\
\ga^{-1}e^{-\ga t}, &\text{for the flow in } \hat N
\end{cases}
\eeq
and define the flow $y=y(s)$ by $y(s)=\hat x(t)$. $y=y(s, \xi)$ is then defined in $[-\ga^{-1}, \ga^{-1}]\times S_0$, smooth in $\{s\neq 0\}$, and satisfies the evolution equation
\beq \label{evol_equation_in_mirrored_universe}
y^{'}:= \frac{d}{ds}y=
\begin{cases}
-\hat F^{-1}\hat \nu e^{\ga t}, \quad &s < 0 \\
\hat F^{-1}\hat \nu e^{\ga t}, \quad &s >0.
\end{cases}
\eeq 
\begin{theorem}
The flow $y=y(s, \xi)$ is of class $C^3$ in $(-\ga^{-1}, \ga^{-1})\times S_0$ and defines a natural diffeomorphism across the singularity. The flow parameter $s$ can be used as a new time function.
\end{theorem}
The flow $y$ is certainly continuous across the singularity, and also future directed, i.e., it runs into the singularity, if $s<0$, and moves away from it, if $s>0$.
The continuous differentiability of $y=y(s, \xi)$ with respect to $s$ and $\xi$ up to order three will be proved in a series of lemmata.

As in the previous sections we again view the hypersurfaces as embeddings with respect to the ambient metric
\beq
d \bar s^2 = -(dx^0)^2+\sigma_{ij}(x^0, x)dx^idx^j.
\eeq
The flow equation for $s<0$ can therefore be written as
\beq \label{evol_equation_in_ausgangsuniversum}
y^{'}=-F^{-1}\nu e^{\ga t}.
\eeq

To prove that $y$ is of class $C^3$ in $(-\ga ^{-1}, \ga^{-1})\times S_0$ we must show that $y^{'}$, $y_i$, $y^{'}_i$, $y^{''}$, $y_{ij}$, $y^{'''}$, $y^{'}_{ij}$, $y^{''}_i$, $y_{ijk}$ (and derivatives obtained by commuting the order of differentiation) are continuous in $\{0\}\times S_0$, which means that we must show that for each of these derivatives the limits $\lim_{s \uparrow 0}$, $\lim_{s\downarrow 0}$ (uniformly with respect to the space variables $\xi^i$) exist and are the same. 

Due to
\beq
y^0(s)=x^0(t), \quad y^{i}(s)=x^{i}(t) \quad \forall s <0,
\eeq
and
\beq
y^0(s)=-x^0(t), \quad y^i(s) = x^i(t) \quad \forall s>0
\eeq
we will consider the 0-component and the $i$-component of each of the above derivatives separately and calculate their limits  as $s\uparrow 0$ and $s \downarrow 0$. Since in each case the limit $s\uparrow 0$ has the same value or the same value up to a sign as the limit $s \downarrow 0$ (provided one of them exists) it is sufficient to have a look at the limit $s\uparrow 0$ and prove its existence or that it is in addition zero respectively.

\begin{lem}
$y$ is of class $C^1$ in $(-\ga^{-1}, \ga^{-1})\times S_0$.
\end{lem}
\begin{proof}
$y^{'}$ is continuous across the singularity if
\beq \label{dds_y0}
\lim_{s \uparrow 0}\frac{d}{ds}y^0,  \lim_{s \uparrow 0}y^j_i \quad \text{exist},
\eeq
and if
\beq
 \lim_{s \uparrow 0}\frac{d}{ds}y^i =\lim_{s \uparrow 0}y^0_i =0.
\eeq
Only the limit  $\lim_{s \uparrow 0}y^j_i$ is not obvious, but one easily checks that $x^j_i$ is a 'Cauchy sequence' as $t \rightarrow \infty$ since its derivative with respect to $t$ can be estimated by $ce^{-\ga t}$, hence $\lim_{s \uparrow 0}y^j_i$ exists as well.

\begin{rem}
The limit relations for $\braket{D^my, \frac{\partial}{\partial x^0}}$ and $\braket{D^my, \frac{\partial}{\partial x^i}}$, where $D^my$ stands for covariant derivatives of order $m$ of $y$ with respect to $s$ or $\xi^i$ are identical to those for $\braket{D^my, -\nu}$ and $\braket{D^my, x_i}$ because $\nu$ converges to $-\frac{\partial}{\partial x^0}$, if $s \uparrow 0$. We want to point out that we have chosen local coordinates in $S_0$ which are given by the limit of the embedding vector $x$ so that we also have $x_i\rightarrow \frac{\partial}{\partial x^i}$.
\end{rem}
\end{proof}
Let us examine the second derivatives
\begin{lem}
$y$ is of class $C^2$ in $(-\ga^{-1}, \ga^{-1})\times S_0$.
\end{lem}
\begin{proof}
"$y^{'}_i$":  The normal component of $y^{'}_i$ has to converge and the tangential components have to converge to zero as $s \uparrow 0$. For $s<0$ we have
\beq
y^{'}=-F^{-1}e^{\ga t}\nu
\eeq
and
\beq	
y^{'}_i=F^{-2}F_ie^{\ga t}\nu-F^{-1}e^{\ga t} \nu_i.
\eeq
The normal component is therefore equal to
\beq
-F^{-2}e^{\ga t}(F^{kl}h_{kl;i}-F^{kl}g_{kl}\tilde v_if^{'}-F^{kl}g_{kl}\tilde v f^{''}u_i + F^{kl}g_{kl}\psi_{\al\be}x^{\be}_i\nu^{\al}+F^{kl}g_{kl}\psi_{\al}x^{\al}_rh^r_i)
\eeq
which converges to
\beq
\lim n(Fu)^{-2}f^{''}u^2\tilde u_i. %= \frac{\ga}{n} \lim \tilde u_i.
\eeq
The tangential components are equal to
\beq
-F^{-1}e^{\ga t}h_{ik}
\eeq
which converge to zero.

"$y_{ij}$": The Gau\ss formula yields
\beq
y_{ij}=h_{ij}\nu
\eeq
which converges to zero as it should.

"$y^{''}$": Here, the normal component has to converge to zero, while the tangential ones have to converge.

We get for $s<0$
\bea \label{(7.8.25)}
y^{''}&= -\frac{D}{dt}(F^{-1}\nu)e^{2\ga t}-F^{-1}\nu\ga e^{2\ga t} \\
&=-F^{-1}\dot \nu e^{2 \ga t} + F^{-2}\nu\dot Fe^{2 \ga t}-F^{-1}\nu\ga e^{2 \ga t}.
\eea

The normal component is equal to 
\bea
-F^{-2}e^{2 \ga t}&(\Fij \dot h_{ij}-\dot{\tilde v}f^{'}\Fg-\tilde v f^{''}\dot u \Fg\\
&+\psi_{\al\be}\nu^{\al}\dot x^{\be}\Fg+\psi_{\al}\dot\nu^{\al}\Fg-\ga F) \\
&-F^{-2}e^{2 \ga t}(-\tilde v f^{'}+\psi_{\al}\nu^{\al})\Fij \dot g_{ij}.
\eea

$F^{-2}e^{2\ga t}$ converges, all terms converge to zero with the possible exception of 
\beq
-\Fg\tilde vf^{''}\dot u - \ga F = -F^{-1}(\Fg\tilde v^2f^{''}+\ga F^2),
\eeq
which however converges to zero, too.

The tangential components are equal to
\bea
F^{-1}D_k(F^{-1})e^{2\ga t} = &-F^{-3}e^{2 \ga t}(\Fij h_{ij;k}-\tilde v_kf^{'}\Fg \\
&-\tilde v f^{''}u_k \Fg + \psi_{\al\be}\nu^{\al}x^{\be}_k\Fg + \psi_{\al}x^{\al}_rh^r_k \Fg),
\eea
which converge to
\beq
\lim -\tilde \ga n(Fu)^{-3} (f^{'}u)^2\tilde u \tilde u_k.
\eeq
\end{proof}
\begin{lem}
$y$ is of class $C^3$ in $(-\ga^{-1}, \ga^{-1})\times S_0$.
\end{lem}
\begin{proof}
"$y_{ijk}$": Now, the normal component has to converge to zero, while the tangential ones should converge. Again we look at $s<0$ and get
\beq \label{(7.8.30)}
y_{ij}=h_{ij}\nu,
\eeq
\beq
y_{ijk}=h_{ijk}\nu+h_{ij}\nu_k.
\eeq
Hence, $y_{ijk}$ converges to zero.

"$y^{'}_{ij}$": The normal component has to converge, while the tangential ones should converge to zero. 

Using the Ricci identities and Lemma \ref{arbitrary_decay_of_special_function} (iii) it can be easily checked that, instead of $y^{'}_{ij}$, we may look at $\frac{D}{ds}(y_{ij})$. 

From (\ref{(7.8.30)}) we deduce
\beq
\frac{D}{ds}y_{ij}=\dot h_{ij}\nu e^{\ga t}+h_{ij}\dot \nu e^{\ga t},
\eeq
and conclude further that the normal component converges in view of Corollary \ref{7.7.3} and the tangential ones converge to zero, since $\dot \nu$ vanishes in the limit.

"$y^{''}_i$": The normal component has to converge to zero and the tangential ones have to converge.

From (\ref{(7.8.25)}) we infer
\bea \label{(7.8.34)}
y^{''} =&-F^{-3}e^{2 \ga t}\Fij(h_{ij;}^{\;\;\;k}-\tilde v^kf^{'}g_{ij} - \tilde v f^{''}u^kg_{ij}+(\psi_{\al}\nu^{\al})^kg_{ij})x_k \\
&+ F^{-2}e^{2 \ga t}(\Fij\dot h_{ij}-\dot{\tilde v} f^{'}\Fg+\frac{D}{dt}(\psi_{\al}\nu^{\al})\Fg)\nu \\
&+F^{-3}e^{2 \ga t}(-\tilde v^2 \Fg[f^{''}+\ga \Fg |f^{'}|^2]-\ga [(\Fij h_{ij})^2\\
&+(\psi_{\al}\nu^{\al})^2(\Fg)^2-2 \tilde v f^{'}\Fij h_{ij}\Fg \\
&+2 \psi_{\al}\nu^{\al}\Fij h_{ij}\Fg-2 \tilde v f^{'}\psi_{\al}\nu^{\al}(\Fg)^2])\nu \\
&+2F^{-3}e^{2 \ga t}(\tilde v f^{'}\Fij h_{ij}-\psi_{\al}\nu^{\al}\Fij h_{ij})\nu
\eea
and thus
\bea
y^{''}_l=&-(F^{-3}e^{2 \ga t}\Fij(h_{ij;}^{\;\;\;k}-\tilde v^kf^{'}g_{ij} - \tilde v f^{''}u^kg_{ij}+(\psi_{\al}\nu^{\al})^k)g_{ij})_lx_k \\
&-F^{-3}e^{2 \ga t}\Fij(h_{ij;}^{\;\;\;k}-\tilde v^kf^{'}g_{ij} - \tilde v f^{''}u^kg_{ij}+(\psi_{\al}\nu^{\al})^kg_{ij})h_{kl}\nu \\
&+( F^{-2}e^{2 \ga t}(\Fij\dot h_{ij}-\dot{\tilde v} f^{'}\Fg+\frac{D}{dt}(\psi_{\al}\nu^{\al})\Fg))_l\nu \\
&+ F^{-2}e^{2 \ga t}(\Fij\dot h_{ij}-\dot{\tilde v} f^{'}\Fg+\frac{D}{dt}(\psi_{\al}\nu^{\al})\Fg)\nu_l \\
&+(F^{-3}e^{2 \ga t}(-\tilde v^2 \Fg[f^{''}+\ga \Fg |f^{'}|^2]-\ga [(\Fij h_{ij})^2 \\
&+(\psi_{\al}\nu^{\al})^2(\Fg)^2 -2 \tilde v f^{'}\Fij h_{ij}\Fg +2 \psi_{\al}\nu^{\al}\Fij h_{ij}\Fg\\
&-2 \tilde v f^{'}\psi_{\al}\nu^{\al}(\Fg)^2]))_l\nu +F^{-3}e^{2 \ga t}(-\tilde v^2 \Fg[f^{''}+\ga \Fg |f^{'}|^2]\\
&-\ga [(\Fij h_{ij})^2+(\psi_{\al}\nu^{\al})^2(\Fg)^2-2 \tilde v f^{'}\Fij h_{ij}\Fg \\
&+2 \psi_{\al}\nu^{\al}\Fij h_{ij}\Fg-2 \tilde v f^{'}\psi_{\al}\nu^{\al}(\Fg)^2])\nu_l \\
&+(2F^{-3}e^{2 \ga t}(\tilde v f^{'}\Fij h_{ij}-\psi_{\al}\nu^{\al}\Fij h_{ij}))_l\nu \\
&+2F^{-3}e^{2 \ga t}(\tilde v f^{'}\Fij h_{ij}-\psi_{\al}\nu^{\al}\Fij h_{ij})\nu_l.
\eea
Therefore, the normal component converges to zero, while the tangential ones converge.

"$y^{'''}$": Differentiating the equation (\ref{(7.8.34)}) we get
\bea
y^{'''} = &\ 3F^{-4}e^{3 \ga t}\dot F\Fij(h_{ij;}^{\;\;\;k}-\tilde v^kf^{'}g_{ij} - \tilde v f^{''}u^kg_{ij}+(\psi_{\al}\nu^{\al})^kg_{ij})x_k \\
&-2 \ga F^{-3}e^{3 \ga t}\Fij(h_{ij;}^{\;\;\;k}-\tilde v^kf^{'}g_{ij} - \tilde v f^{''}u^kg_{ij}+(\psi_{\al}\nu^{\al})^kg_{ij})x_k \\
&-F^{-3}e^{3 \ga t}\frac{D}{dt}(\Fij(h_{ij;}^{\;\;\;k}-\tilde v^kf^{'}g_{ij} - \tilde v f^{''}u^kg_{ij}+(\psi_{\al}\nu^{\al})^kg_{ij}))x_k \\
&-F^{-3}e^{3 \ga t}\Fij(h_{ij;}^{\;\;\;k}-\tilde v^kf^{'}g_{ij} - \tilde v f^{''}u^kg_{ij}+(\psi_{\al}\nu^{\al})^kg_{ij})\dot x_k \\
&-2 F^{-3}e^{3 \ga t}(\Fij\dot h_{ij}-\dot{\tilde v} f^{'}\Fg+\frac{D}{dt}(\psi_{\al}\nu^{\al})\Fg)\nu \\
&+2\ga F^{-2}e^{3 \ga t}(\Fij\dot h_{ij}-\dot{\tilde v} f^{'}\Fg+\frac{D}{dt}(\psi_{\al}\nu^{\al})\Fg)\nu \\
&+ F^{-2}e^{3 \ga t}\frac{D}{dt}(\Fij\dot h_{ij}-\dot{\tilde v} f^{'}\Fg+\frac{D}{dt}(\psi_{\al}\nu^{\al})\Fg)\nu \\
&+ F^{-2}e^{3 \ga t}(\Fij\dot h_{ij}-\dot{\tilde v} f^{'}\Fg+\frac{D}{dt}(\psi_{\al}\nu^{\al})\Fg)\dot \nu \\
&-3F^{-4}e^{3 \ga t}(-\tilde v^2 \Fg[f^{''}+\ga \Fg |f^{'}|^2]-\ga [(\Fij h_{ij})^2\\
&+(\psi_{\al}\nu^{\al})^2(\Fg)^2 -2 \tilde v f^{'}\Fij h_{ij}\Fg \\
&+2 \psi_{\al}\nu^{\al}\Fij h_{ij}\Fg-2 \tilde v f^{'}\psi_{\al}\nu^{\al}(\Fg)^2])\nu \\
&+2 \ga F^{-3}e^{3 \ga t}(-\tilde v^2 \Fg[f^{''}+\ga \Fg |f^{'}|^2]-\ga [(\Fij h_{ij})^2\\
&+(\psi_{\al}\nu^{\al})^2(\Fg)^2 -2 \tilde v f^{'}\Fij h_{ij}\Fg \\
&+2 \psi_{\al}\nu^{\al}\Fij h_{ij}\Fg-2 \tilde v f^{'}\psi_{\al}\nu^{\al}(\Fg)^2])\nu \\
&+F^{-3}e^{3 \ga t}\frac{D}{dt}(-\tilde v^2 \Fg[f^{''}+\ga \Fg |f^{'}|^2]-\ga [(\Fij h_{ij})^2\\
&+(\psi_{\al}\nu^{\al})^2(\Fg)^2 -2 \tilde v f^{'}\Fij h_{ij}\Fg \\
&+2 \psi_{\al}\nu^{\al}\Fij h_{ij}\Fg-2 \tilde v f^{'}\psi_{\al}\nu^{\al}(\Fg)^2])\nu \\
&+F^{-3}e^{3 \ga t}(-\tilde v^2 \Fg[f^{''}+\ga \Fg |f^{'}|^2]-\ga [(\Fij h_{ij})^2\\
&+(\psi_{\al}\nu^{\al})^2(\Fg)^2 -2 \tilde v f^{'}\Fij h_{ij}\Fg \\
&+2 \psi_{\al}\nu^{\al}\Fij h_{ij}\Fg-2 \tilde v f^{'}\psi_{\al}\nu^{\al}(\Fg)^2])\dot \nu \\
&-6F^{-4}e^{3 \ga t}(\tilde v f^{'}\Fij h_{ij}-\psi_{\al}\nu^{\al}\Fij h_{ij})\nu\\
&+4\ga F^{-3}e^{3 \ga t}(\tilde v f^{'}\Fij h_{ij}-\psi_{\al}\nu^{\al}\Fij h_{ij})\nu\\
&+2F^{-3}e^{3 \ga t}\frac{D}{dt}(\tilde v f^{'}\Fij h_{ij}-\psi_{\al}\nu^{\al}\Fij h_{ij})\nu\\
&+2F^{-3}e^{3 \ga t}(\tilde v f^{'}\Fij h_{ij}-\psi_{\al}\nu^{\al}\Fij h_{ij})\dot \nu.
\eea
We remark that 
\beq
\dot x_k = F^{-2}F_k\nu-F^{-1}\nu_k
\eeq
and
\beq
\dot u_k = F^{-1}\tilde v_k - F^{-2}\tilde v F_k
\eeq
and that in the following especially the results of Lemma \ref{7.7.6}, Lemma \ref{7.7.7} and Corollary \ref{7.7.8} will be used.

Let us consider the normal component of $y^{'''}$ first, which has to converge. We will present here only how to handle the following term, the other terms are easier.

\bea \label{anwendung_zusatzbedingung_fuer_c3_regularitaet}
\frac{D}{dt}[f^{''}+\ga \Fg|f^{'}|^2] = &\frac{D}{dt}[f^{''}+\tilde \ga |f^{'}|^2] + \frac{D}{dt}(\Fg-n)\ga |f^{'}|^2 \\
&+2(\Fg-n)\ga f^{'}f^{''}\dot u \\
\equiv& I_1 + I_2 + I_3.
\eea
$I_1$ converges due to assumption (\ref{Zusatzbedingung_heiko}), and the convergence of $I_3$ is obvious. For $I_2$ we use
\bea
\frac{D}{dt}\Fij g_{ij} =& F^{ij,rs}(|u|\check h_{ij})g_{ij}\frac{D}{dt}(|u|\check h_{rs})+\Fij \dot g_{ij}
\eea
together with (\ref{hij_in_konformer_metrik_einfach_version}), (\ref{haus2}) and (\ref{haus3}).
%\bea
%F^{ij,rs}(|u|\check h_{ij})g_{ij}= (F^{ij,rs}(|u|\check h_{ij}) - F^{ij,rs}(\tilde \ga^{-1}g_{ij}))g_{ij}
%\eea
%so that the desired convergence follows by mean value theorem.

Now we consider the tangential component of $y^{'''}$, i.e. we prove 
\beq
\braket{y^{'''}, x_l}\rightarrow 0.
\eeq
The crucial terms are
\bea
3 F^{-4}&e^{3 \ga t}(\Fg)^2\tilde v^2 (f^{''})^2 \dot u u^k + 2 \ga F^{-3}e^{3\ga t}\tilde v f^{''}u^k\Fg \\
&+ F^{-3}e^{3\ga t}\tilde v f^{'''}u^k\dot u\Fg + F^{-5}e^{3\ga t}(\Fg)^2\tilde v^3|f^{''}|^2u^k
\eea
and can be rearranged to yield
\beq
F^{-5}e^{3 \ga t}n^2\tilde v u^k(4f^{''}(f^{''}+\tilde \ga |f^{'}|^2)-f^{'}(f^{''}+\tilde \ga |f^{'}|^2)^{'}).
\eeq
Hence the tangential components tend to zero.

The remaining mixed derivatives of $y$ which are obtained by commuting the order of differentiation in the derivatives we already treated are also continuous across the singularity in view of the Ricci identities and  Lemma \ref{arbitrary_decay_of_special_function} (iii).
\end{proof}

\end{document}